\documentclass[a4paper,11pt]{article}
\usepackage{graphicx} 
\usepackage{braket}
\usepackage{amsmath}
\usepackage{amssymb}
\usepackage{amsthm}
\usepackage{algorithm}
\usepackage{algpseudocode}
\usepackage{newtxtext,newtxmath}
\usepackage{xcolor}
\usepackage{verbatim}
\usepackage{ulem}
\usepackage{mathtools}
\usepackage{makecell}
\usepackage{bibentry}
\usepackage{ bbold }
\usepackage{ dsfont }
\usepackage[T1]{fontenc}

\definecolor{lavender}{RGB}{170,90,210}

\usepackage[utf8]{inputenc}
\usepackage{enumerate,framed,mdwlist}
\usepackage{color,colortab}
\usepackage[square,numbers,sort&compress]{natbib}

\newcommand{\R}{\mathbb{R}}

\newcommand{\floor}[1]{\left\lfloor#1\right\rfloor}
\newcommand{\ceil}[1]{\left\lceil#1\right\rceil}

\renewcommand{\epsilon}{\varepsilon}

\newtheorem{nn}{}[section]
\newtheorem{lemma}[nn]{Lemma}
\newtheorem{theorem}[nn]{Theorem}

\newtheorem{definition}[nn]{Definition}
\newtheorem{claim}[nn]{Claim}

\newtheorem{REMARK}[nn]{Remark}
\newenvironment{remark}{\begin{REMARK}}{\end{REMARK}}

\def\ve#1{\mathchoice{\mbox{\boldmath$\displaystyle\bf#1$}}
{\mbox{\boldmath$\textstyle\bf#1$}}
{\mbox{\boldmath$\scriptstyle\bf#1$}}
{\mbox{\boldmath$\scriptscriptstyle\bf#1$}}}

\newcommand{\w}{{\ve w}}

\usepackage{amsmath}

\numberwithin{equation}{section}

\allowdisplaybreaks

\usepackage{rotating}
\usepackage{array}
\usepackage{multirow}

\usepackage{authblk}
\newcommand{\email}[1]{{\small{{\texttt{#1}}}}}

\usepackage{tikz}
\usetikzlibrary{tikzmark, arrows.meta}
\usepackage{caption}

\begin{document}

\title{Identifying faulty edges in resistive electrical networks}
\author[1]{Barbara Fiedorowicz}
\author[1]{Amitabh Basu}

\affil[1]{Department of Applied Mathematics and Statistics, Johns Hopkins University\\
\email{\{bfiedor1,basu.amitabh\}@jhu.edu} }

\date{}

\maketitle

\begin{abstract}
Given a resistive electrical network, we would like to determine whether all the resistances (edges) in the network are working, and if not, identify which edge (or edges) is faulty. To make this determination, we are allowed to measure the effective resistance between certain pairs of nodes (which can be done by measuring the amount of current when one unit of voltage difference is applied at the chosen pair of nodes). The goal is to determine which edge, if any, is not working in the network using the smallest number of measurements. We prove rigorous upper and lower bounds on this optimal number of measurements for different classes of graphs. These bounds are tight for several of these classes, which shows that our measurement strategies are optimal. To establish these bounds, we provide closed-form expressions for the inverse of the reduced Laplacian for complete graphs and complete $k$-partite graphs for any $k$, and consequently, for effective resistance values in such graphs. These formulas could be of independent interest.
\end{abstract}


\section{Introduction}

Diagnosing faults in electrical circuits and wiring is an important problem with applications in power systems,
very large scale integration (VLSI), chip design, circuit board systems, to name a few. The literature is vast; the
reader is referred to the survey~\cite{furse2020fault}, a sample of more recent
work~\cite{dwivedi2010fault,wang2019electromagnetic} and the references therein. There are a variety of approaches
that are adopted in practice and new methods are regularly being invented even in the present. In this paper, we
consider a particular formalization of this problem which, to the best of our knowledge, has not been studied
rigorously in prior work (we discuss related work in  Section~\ref{sec:literature}). Nevertheless, we believe the model we
study is natural, of practical relevance, and worth studying from a mathematical perspective. It is related to and
inspired by work in detecting faults and structures in continuous conducting media; see the
textbooks~\cite{holder2004electrical,bui2007fracture,adler2015electrical} for a more recent survey. Our
problem can be seen as a discrete analogue of this problem from the literature on continuous media and related
inverse problems. The resulting optimization problem is rich in its scope, spanning topics from combinatorial and
spectral graph theory, integer and combinatorial optimization, information theory, and electrical network theory.
Let us formally introduce the problem. 

\begin{definition}\label{def:electrical-network} A \textit{(resistive) electrical network} is given by a graph $G = (V,E)$ with nonnegative weights $w_e\in \R_+$ for each edge $e\in E$. The graph is allowed to have multiple edges between any pair of vertices, but no self-loops are allowed. Each vertex $v \in V$ represents a \textit{terminal} in the network and each edge $e = uv$ in $E$ represents a \textit{resistor connection} between the terminals $u$ and $v$. The weight $w_e \in \R_+$ represents the conductance value, i.e.,~the reciprocal of the resistance value of that edge. We use ${\w} \coloneq  (w_e)_{e\in E}$ to denote the vector of weights and the notation $(V,E,\w)$ to denote an electrical network.

Given two vertices $r, s \in V$, which are called the \textit{source} and \textit{sink} terminals, the \textit{effective resistance} $R_{rs}$ of the network between $r,s$ is defined as the reciprocal of the net current that would flow through the network if a unit voltage difference was applied to the source--sink pair.
\end{definition}

The problem we wish to study is to detect whether all the resistance connections in the network are working as expected, or whether one of them has been damaged. To do this, we are allowed to \textit{probe} the network by making current value measurements (or equivalently, effective resistance values)  across pairs of terminals (vertices) and use these values to determine if any edge is faulty and if so, which one. 

\begin{definition}[Faulty Edge Detection Problem]\label{def:edge-detection} Let $(V,E,\w)$ be a given electrical network and let $T \subseteq {V \choose 2}$ be a fixed subset of (unordered) pairs of vertices. An adversary (nature) selects an edge $e^{*} \in E$ and alters the weight from $w_{e^{*}}$ to 0 (that resistor has been completely removed from the network). One has to select a subset of \textit{measurements} $S \subseteq T$ and one observes the effective resistance values for each of the pairs in $S$ (equivalently, one observes the net current flowing through the network when a unit voltage difference is applied to this pair). From these measurements and observed values, one has to figure out which edge was altered by the adversary. The goal is to select a minimum-size set $S$ of measurements to achieve this. 

A different version of the problem is obtained when the adversary changes the weight of $e^{*}$ to $+\infty$ (that edge has been shorted, which is equivalent to contracting the edge).
\end{definition}

One could allow for the possibility that no edge has been altered. Then, one has to detect if any edge has been altered and if so, identify which one.  We give a short argument below that the number of measurements in this version of the problem is at most one more than the version in  Definition~\ref{def:edge-detection} where a single, unknown edge has been altered. Consequently, we will focus on the version in  Definition~\ref{def:edge-detection} for the remainder of the paper.

Let $\mathcal{A} \in \mathbb{R}^{| \mathbb{T}| \times| \mathbb{E}| }$ be the matrix where each row corresponds to a measurement $t \in T$ and each column corresponds to an edge $e \in E$ that has been altered. Each entry in $\mathcal{A}$ contains the effective resistance value for a given pair of source and sink terminals and a given altered edge. Consider the submatrix $\mathcal{A}_S \in \mathds{R}^{| \mathds{S}| \times| \mathds{E}| }$ of values corresponding to measurements $S \subseteq T$ and all edges $E$.  If all columns in $\mathcal{A}_S$ are different from one another, then the provided measurements $S$ satisfy the version of our problem given in  Definition~\ref{def:edge-detection}. On the other hand, if there are at least two columns that are the same, then the set $S$ of measurements will not distinguish these two edges. Hence, an equivalent formulation of the problem is to find the smallest set $S\subseteq T$ of measurements such that $\mathcal{A}_S$ has distinct columns.

Suppose now we include a new column $c$ in $\overline{\mathcal{A}} = [\mathcal{A}\vert c]$ which contains the effective resistance values when no edge in $G$ has been altered. Consider any set $S$ of measurements such that $\mathcal{A}_S$ has distinct columns (i.e.,~$S$ is a set of measurements that can identify a faulty edge assuming one exists). If all columns of the updated submatrix $\overline{\mathcal{A}}_S$ are different from one another, then the measurements $S$ solve the problem including the possibility that no edge has been altered. Otherwise, there is one column in $\overline{\mathcal{A}}_S$ that is the same as column $c$ restricted to $\overline{\mathcal{A}}_S$ (there can be at most one such column since $\mathcal{A}_S$ has distinct columns). Let $e^{*}\in E$ be the edge corresponding to this column. If there is no measurement in $T$ such that the column corresponding to $e^{*}$ has a different entry compared to column $c$ in $\overline{\mathcal{A}}$, then the problem cannot be solved using measurements in $T$. In other words, no measurement is able to distinguish between the possibility that $e^{*}$ has been altered and the possibility that no edge has been altered. Thus, we may assume there exists a measurement $s^{*} \in T$ such that the corresponding entry of column of $c$ will be different from that associated with $e^{*}$.  As a result, $S^{*} = S \cup \{s^{*}\}$ is a set of measurements that will solve the problem with the additional possibility that no edge has been altered. Since $| S^{*}|  = | S|  + 1$, at most one additional measurement is required if the possibility of no edge being altered is included in the problem.

We also note that in  Definition~\ref{def:edge-detection}, the measurements $S\subseteq T$ are made \textit{non-adaptively}, i.e.,~the result of a measurement is not used to decide what the next measurement is. One must decide on the set $S$ of measurements up front and then the observed measurement values must be used to identify the faulty edge. Our results below apply to this non-adaptive setting. The analysis of adaptive measurement strategies is left for future work.

\subsection{Our contributions}\label{sec:contributions} We provide tight upper and lower bounds on the smallest number of measurements for the fault detection problem in various families of electrical networks. In particular, we are able to fully resolve the problem for complete graphs and complete $k$-partite graphs for $k\geq 2$. The formal statements are in Section~\ref{sec:formal_results} and their proofs are provided in Section~\ref{sec:proofs}.
The proof techniques draw upon combinatorial and spectral graph theory, electrical network theory, and ideas from information theory. As an important component in our proofs, we derive closed-form expressions for inverses of certain reduced Laplacian matrices of complete graphs and complete $k$-partite graphs  in  Section~\ref{sec:gen_inverse}. The entries of these inverses are used to compute effective resistance values (in Appendices \ref{incident_edge_lemma_complete} and \ref{A:edge_partition_k-partite}). These expressions for the inverses can also be computed using previous work on effective resistance in complete graphs and complete $k$-partite graphs. Our derivation is based on direct matrix computations and provides formulas that are more readily usable for the fault detection problem.  We note briefly that each effective resistance measurement provides a real number; thus, a single measurement potentially provides infinitely many bits of information. For this reason, classical information theoretic lower bound arguments based on bits of information do not immediately apply. 

While our final results are mathematically precise and rigorous, our search for the best measurement strategy in each of these cases involved modeling the problem as a set covering/integer programming problem and using computational experiments to guide us towards the optimal strategy for measurements. Thus, the problem involves a nice interplay between rigorous mathematical analysis and computer-aided search using discrete optimization techniques. More broadly, we believe that the fault detection problem we set up in  Definition~\ref{def:edge-detection} is a new, challenging combinatorial optimization problem that is rich in its mathematical structure and is motivated by a fundamental problem in fault diagnosis in electrical networks, which has diverse applications. We close the paper in  Section~\ref{sec:future avenues} with some future directions, including some concrete open questions, that we hope can spawn new and interesting research avenues for the discrete mathematics and optimization community.

\subsection{Related literature}\label{sec:literature} As mentioned before, the formalization of fault detection given in  Definition~\ref{def:edge-detection} does not seem to have received a great deal of attention in terms of obtaining provable bounds on the smallest number of measurements. Nevertheless, there are several very related strands of work on fault detection in electrical networks that we review next, which place our work in the broader context of fault diagnosis in electrical systems.

The closest line of work is that of Ibaraki et al.~\cite{ibaraki1981minimal}. They describe in this paper the
various algorithms for obtaining what they call the minimal test set for detecting and locating a faulty edge in an
undirected tree.  The same set-up as  Definition~\ref{def:edge-detection} is used: a graph $G$ that is a tree is
given with edge weights equal to 1. Note that in this case, an effective resistance measurement simply states
whether or not the faulty edge is on the path connecting the measurement vertices or not. In fact, the authors
of~\cite{ibaraki1981minimal} do not explicitly mention effective resistance, but instead consider probes that
verify whether a specified pair of vertices are still connected.  They provide tight bounds on the smallest number
of measurements in terms of the number of degree one and degree two vertices in the graph. 
A weaker version of the problem is studied in~\cite{kahng1998test} where the authors consider the problem of
testing whether all edges in a tree network are working properly;  the goal being simply to detect if any edge has
been altered or not, without necessarily pinpointing which one if a fault is indeed detected. The authors call this
the \textit{tree testing problem}. The paper~\cite{shashank2012fault} also considers the fault detection problem in
\textit{ladder graphs} with a restricted set of allowed effective resistance measurements in such networks. They do
an exhaustive analysis for 2-,3- and 4-step ladder graphs and propose breaking down larger ladder networks into
segments of smaller ladder graphs, with a qualitative analysis of such a heuristic.

A highly related line of work originates in the literature on inverse problems, especially in the so-called
\textit{Calder\'on problem}. The discrete version of the Calder\'on problem is to recover the values of
\textit{all} resistors in a resistive network given a possible set of measurements that one can make. While there
is no adversary altering the resistances as in  Definition~\ref{def:edge-detection}, obviously if one can determine all the
resistance values, then one can simply check if any one of them differs from the original value. Thus, the discrete
Calder\'on problem requires one to deduce much more information from the measurements, compared to the faulty edge
detection problem in  Definition~\ref{def:edge-detection} (and therefore may require a much larger number of measurements). A
good introduction to the discrete Calder\'on problem is the book~\cite{curtis2000inverse}, which also surveys
relevant prior work.~\cite{kazakov2025inverse} is a recent paper that surveys work in the past couple of decades
and provides some new mathematical foundations for the problem. The original Calder\'on problem is defined for a
continuous conducting medium as opposed to a discrete electrical network modeled by a graph. One is allowed to make
current and voltage measurements on the boundary of a connected domain $\Omega \subseteq \R^d$ and from this one has to infer
the electrical properties in the interior of $\Omega$. The problem forms the mathematical foundation for
\textit{Electrical Impedance Tomography
(EIT)}~\cite{holder2004electrical,adler2015electrical,borcea2002electrical}, which has numerous applications
ranging from computational medicine and geophysics to testing for structural defects in solids. Much of the
literature uses the terms \textit{EIT} and \textit{Calder\'on problem} interchangeably. The tools and techniques
are quite different from the discrete setting (and especially the techniques we use in this paper), drawing upon
methods from PDEs and classical inverse problems. A continuous version of the problem we study in this paper
(Definition~\ref{def:edge-detection}) has been studied, arising out of the work on the continuous Calder\'on problem. Here,
one considers a \textit{fracture} or a \textit{crack} in the conducting medium, modeled by a curve that is
perfectly insulating. The goal is to, again, use boundary measurements to determine the location/shape of the
curve. The problem was introduced in a seminal paper by Friedman and Vogelius~\cite{friedman1989determining}, with
several follow up works;
see~\cite{alessandrini1993stable,alessandrini1995stability,alessandrini1996unique,alessandrini1997determining,santosa1991computational}
for a sample and the textbook~\cite{bui2007fracture} for more details.

In~\cite{shi1999diagnosis,shi2001structural}, the authors consider the problem of detecting faults in wiring
networks. They build on prior work which can be traced back to a seminal paper by Kautz~\cite{kautz1974testing}.
Here, the vertices of a graph model different logical units (often called \textit{nets}) in a circuit board and
each edge denotes a possible site of a ``fault''. One is allowed to probe the logical units and use the responses
to detect these faults, if any. The mathematical formalization is the following: given a (simple) graph
$G = (V,E)$, one can query a subset $U \subseteq V$ and the response is a subset $Q(U)$ which is the set of all
vertices that can be reached by a path from a vertex in $U$. The goal is to find the smallest set of
queries to identify all the connected components of $G$ (the vertices in the connected components of size
two or more correspond to the logical units that are connected by faulty edges). A group
testing~\cite{aldridge2019group} approach to this problem was considered in~\cite{chen1989detecting} and is
discussed
in the papers~\cite{shi1999diagnosis,shi2001structural} cited above, along with other relevant work on this
problem.

Finally, we mention some related work on generalized inverses of graph Laplacians and formulas for effective
resistance. Several authors have provided analytical expressions for effective resistance, also known as
\textit{resistance distance} as named by Klein and Randi\'c~\cite{klein1993resistance}. These include expressions
for complete graphs~\cite{lukovits1999resistance}, complete bipartite
graphs~\cite{ye2019resistance,ge2020spanning}, and complete $k$-partite
graphs~\cite{gervacio2016resistance}. As mentioned in  Section~\ref{sec:contributions}, in this paper we derive
closed-form formulas for the inverse of a certain reduced Laplacian for both complete graphs and complete
$k$-partite graphs, which are then used to obtain formulas for effective resistances in such graphs both in
their original form and when the weight on an edge has been altered. This inverse, denoted $L(v)^{-1}$ where
$v$ is any vertex of the graph, is the inverse of the Laplacian with the $v$th row and column
removed (see Theorem~\ref{thm:eff-res-change}).     We note that if the matrix $L(v)^{-1}$ is extended by inserting a row
and column of zeros indexed by $v$, we obtain     a so-called \textit{symmetric $\{1\}$-inverse} of
$L$ (a generalized inverse); see~\cite{zhou2016resistance}.
These inverses can also be computed using the effective resistance formulas
from~\cite{lukovits1999resistance,ye2019resistance,ge2020spanning,gervacio2016resistance}, but we give a direct
matrix-based derivation. Our expressions are easier to work with for our main results on the minimum number of
measurements.  Generalized inverses of graph Laplacians have been discussed extensively in the literature due to
various applications in network analysis and random walks on graphs. Given how computationally expensive it is to
obtain (see~\cite{saurabh2015computing} describing the use of parallelization to speed up the process using GPUs),
finding closed-form expressions for particular graphs for different generalized inverses is a topic of research
that has gained traction in the last two decades. A closed form of the Moore--Penrose inverse, denoted by
$L^+$, has been found for connected threshold graphs~\cite{alazemi2024combinatorial}, paths, cycles, ladders,
fans, and wheel graphs~\cite{nunez2023mpinverseLaplacian}. Bapat and Gupta~\cite{bapat2010resistance} take this a
step further by providing closed-form expressions for effective resistance for wheels and fans in terms of
Fibonacci and generalized Fibonacci numbers by utilizing an alternative method of examining submatrices of the
Laplacian. These works make use of the traditional signed Laplacian obtained by taking the difference of the degree
matrix and adjacency matrix (see  Section~\ref{sec:preliminaries}). Several works also discuss the so-called
\textit{signless Laplacian} matrix, which is obtained by taking the \textit{sum} of the degree matrix and adjacency
matrix. Monfared and Mallik~\cite{monfared2019analog} describe how the signless Laplacian can be used in a similar
fashion to the signed Laplacian in its application to the Matrix Tree Theorem (which can also be used for
calculating effective resistance). Combinatorial formulas for the Moore--Penrose inverse of the signless Laplacian
have been found for trees and odd unicyclic graphs in~\cite{hessert2021moore}, as well as for bipartite graphs
in~\cite{alazemi2023moore}. Bu et al.~\cite{bu2012note} further this by providing effective resistance formulas for
connected bipartite graphs in terms of the signless Laplacian matrix. This is done via a block representation of
the so-called \textit{group inverse} (another generalized inverse) of the Laplacian. More generally, group inverses
can also be used to calculate effective resistance;  see, for
example,~\cite{kirkland1997distances,molitierno2018entries,kaliyaperumal2023resistance} for formulas for group
inverses and related generalized inverses, and their applications in computing effective resistance in structured
graphs like wheels and dumbbells.   Liao et al.~\cite{liao2023scalable} discuss how the Moore--Penrose inverse
$L^+$ can be approximated by using its connection to $L(v)^{-1}$, and provide new formulas for expressing
$L^+$ and the Moore--Penrose inverse of the normalized Laplacian in terms of $L(v)^{-1}$.

\section{Statement of main results}\label{sec:formal_results}

We consider the fault detection problem in  Definition~\ref{def:edge-detection} where all the unaltered weights are 1, i.e.,~the network is made of unit resistance (conductance) edges; we use $\mathbf{1}$ to denote the vector of all ones. Moreover, we allow all possible measurements, i.e.,~$T = {V \choose 2}$ (see  Definition~\ref{def:edge-detection}). We present upper and lower bounds when the electrical network's underlying graph is a complete graph or a complete $k$-partite graph.

\paragraph*{Complete graphs}

A complete graph is a graph such that there is an edge between every pair of vertices. The following theorem gives the number of measurements  needed for complete graphs.

\begin{theorem}\label{thm:complete}
The smallest number of measurements needed to solve the faulty edge detection problem in a complete graph on $n \geq 6$ vertices is exactly $\left\lceil \frac{2n}{3}\right\rceil$. 
\end{theorem}

\paragraph*{Complete $k$-partite graphs}
A complete $k$-partite graph is one where the set of vertices can be partitioned into $k$ independent sets and there is an edge between every pair of vertices in different partitions.

We first state exact bounds for the bipartite ($k=2$) and tripartite ($k=3$) cases, and then state a generalization for arbitrary $k \geq 2$.

\begin{theorem}\label{thm:bipartite} Let $(V,E,\mathbf{1})$ be an electrical network such that the underlying graph $G=(V,E)$ is a complete bipartite graph with $| V| =n$ and partitions $p_\beta, p_\gamma$ with $| p_\beta|  \leq | p_\gamma| $. 

When $| p_\beta|  < | p_\gamma| $, the smallest number of measurements needed to solve the faulty edge detection problem is  $\left\lfloor{\frac{2}{3}| p_\gamma|  + \frac{1}{3}| p_\beta| } \right\rfloor - 1$ when $| p_\gamma|  - | p_\beta|  \equiv 0\mod{3}$, and equal to $\left\lfloor \frac{2}{3}| p_\gamma|  + \frac{1}{3}| p_\beta|  \right\rfloor$ when $| p_\gamma|  - | p_\beta|  \equiv (1 \textrm{ or } 2)\mod{3}$.

When $| p_\beta|  = | p_\gamma| $, the smallest number of measurements needed to solve the faulty edge detection problem is $\left\lfloor\frac{2n}{3}\right\rfloor - 1$ when $| p_\beta|  \equiv (0 \textrm{ or } 1)\mod{3}$, and equal to $\left\lfloor\frac{2n}{3}\right\rfloor$ when $| p_\beta|  \equiv 2\mod{3}$. 
\end{theorem}

\begin{theorem}\label{thm:tripartite} Let $(V,E,\mathbf{1})$ be an electrical network such that the underlying graph $G=(V,E)$ is a complete tripartite graph with $| V| =n$ and partitions $p_\alpha, p_\beta, p_\gamma$ with $| p_\alpha|  \leq | p_\beta|  \leq | p_\gamma| $.  The smallest number of measurements needed to solve the faulty edge detection problem is given in Table~\ref{tab:bounds_tripart}.

\begin{table}[htbp]
    \centering
    \renewcommand{\arraystretch}{1.5}
    \begin{tabular}{|c|c|c|}
    \hline
    \multicolumn{2}{|c|}{\textbf{Partition Sizes}} & \multicolumn{1}{|c|}{\textbf{Number of Measurements}} \\
    \hline
    \hline
    \multirow{2}{3.3cm}{$|p_\alpha| < |p_\beta| < |p_\gamma|$} & $|p_\alpha| + |p_\beta| -1 \geq |p_\gamma|$ & $\left\lceil \frac{n-3}{2}\right\rceil$ \\ 
    \cline{2-3}
    & $|p_\alpha| + |p_\beta| -1 < |p_\gamma|$ & $\frac{n}{3} + \frac{|p_\gamma|}{3} - \frac{4}{3}$ \\ 
    \hline
    \multirow{2}{3.3cm}{$|p_\alpha| = |p_\beta| < |p_\gamma|$} & $2(|p_\alpha| -1) \leq |p_\gamma|-1$ & $\frac{2n}{3} + \frac{2|p_\alpha|}{3} - \frac{4}{3}$ \\ 
    \cline{2-3}
    & $2(|p_\alpha| -1) > |p_\gamma|-1$ & $\frac{2n}{3} - \frac{|p_\gamma|}{3} - \frac{5}{3}$ \\ 
    \hline
    \multicolumn{2}{|c|}{$|p_\alpha| < |p_\beta| = |p_\gamma|$} & $\frac{2n}{3} - \frac{|p_\alpha|}{3} - \frac{5}{3}$ \\ 
    \hline
    \multicolumn{2}{|c|}{$|p_\alpha| = |p_\beta| = |p_\gamma|$} & $\frac{2n}{3} -2$ \\ 
    \hline
    \end{tabular}
    \caption{Number of measurements required for tripartite graphs.}
    \label{tab:bounds_tripart}
\end{table}


\end{theorem}

We now state our result for $k$-partite graphs for arbitrary $k\geq 2$.

\begin{definition}\label{def:list} 
An \textit{ordered list} of numbers is a finite sequence $a_1, \ldots, a_k$ of real numbers such that the elements $a_1 \leq\cdots\leq a_k$, with $k\geq 1$. Given an ordered list $L$ of $k$ numbers and a subset $S \subseteq \{1, \ldots, k\}$, define $L\setminus S$ to be the new ordered list where elements $a_i$, $i\in S$, are removed from $L$.\footnote{Note that $S$ is a set of indices, and not a subset of the actual values in the list $L$. In particular, we allow for repeated values in an ordered list, i.e.,~$a_i$ could be equal to $a_j$ for $i\neq j$. The new list $L\setminus S$ is reindexed from 1 to $k - | S| $. Since the elements in $L$ and $L\setminus S$ are ordered, there is no ambiguity in the ordering of the elements in the new list.}   Also, given any ordered list $L$, we define
\begin{equation*}
val(L) \coloneq  \sum_{i \in [k]: i \mod{3} \equiv 0} (2a_i -2).
\end{equation*} 
\end{definition}

\begin{theorem}\label{thm:k-partite}
Let $k\geq 2$ be a natural number. Consider an electrical network $(V,E,\mathbf{1})$ such that the underlying graph $G=(V,E)$ is a complete $k$-partite graph with $| V| =n$, with partitions labeled $p_1, \ldots, p_k$, such that the partition sizes satisfy $2 \leq | p_1|  \leq | p_2|  \leq\cdots\leq | p_k| $. Then one needs at least $\left\lceil \frac{n - k}{2} \right\rceil$ measurements to solve the faulty edge detection problem. Let $L$ be the ordered list of these partition sizes. An upper bound on the number of measurements is given by
\begin{equation*}
\begin{array}
{rl}
val(L) & \textrm{if} \quad k \equiv 0\mod{3},\\
\min_{i\in\{1,\ldots, k\}}\left\{\left\lceil\frac{2(| p_i| )}{3}\right\rceil + val(L\setminus\{i\})\right\} & \textrm{if} \quad k \equiv 1\mod{3},\\
\min_{i, j\in \{1,\ldots, k\},\; i\neq j}\left\{\left\lceil\frac{2(| p_i|  + | p_j|  - 1)}{3}\right\rceil + val(L\setminus\{i,j\})\right\} & \textrm{if} \quad k \equiv 2\mod{3}.
\end{array}
\end{equation*}
\end{theorem}

\begin{remark}
The upper and lower bounds in Theorem~\ref{thm:k-partite} cannot be improved. Observe that for complete bipartite graphs with $| p_\gamma|  = | p_\beta|  + 1$, the bound in Theorem~\ref{thm:bipartite} reduces to $\left\lfloor\left\lfloor\frac{n}{2}\right\rfloor + \frac23\right\rfloor = \left\lfloor\frac{n}{2}\right\rfloor$. This shows that the general lower bound in Theorem~\ref{thm:k-partite} cannot be improved since it is equal to $\ceil{\frac{n}{2}} - 1 = \left\lfloor\frac{n}{2}\right\rfloor$. Similarly, the bound in Theorem~\ref{thm:bipartite} with $| p_\beta|  = | p_\gamma| $ is $\left\lfloor\frac{2n}{3}\right\rfloor$ for graphs with $| p_\beta|  \equiv 2\mod{3}$, which is the same as what one would obtain from the upper bound in Theorem~\ref{thm:k-partite}. Thus, the general upper bound in Theorem~\ref{thm:k-partite} cannot be improved.
\end{remark}

\section{Preliminaries}\label{sec:preliminaries}

In this section, we discuss matrix representations of a graph $G$ that are used to compute effective resistance values within the original graph $G$ and when an edge $(a,b)\in E$ is altered, as well as several conceptual tools used in the proofs in  Section~\ref{sec:proofs}.

\subsection{Effective Resistance in Terms of the {Laplacian}}

\begin{definition}
The \textbf{adjacency matrix} of a weighted graph $G=(V,E)$ is a square matrix where each row and column is associated with a vertex in the graph and an entry corresponding to row indexed by $u \in V$ and column $v \in V$ is the weight $w_{(u,v)}$ if $(u,v)\in E$, and 0 otherwise.
The \textbf{degree matrix} of a graph is a diagonal matrix where the rows and columns represent each vertex in the graph and the diagonal entries represent the sum of the weights of the edges incident at the vertex.
The \textbf{Laplacian} of a graph is the square matrix defined by the difference of its degree matrix and its adjacency matrix.  
\end{definition}

\begin{theorem}\label{thm:eff-res-change} Let $(V,E, \w)$ be an electrical network. Let $L$ denote the Laplacian matrix of the corresponding weighted graph and let $v\in V$ be an arbitrary, fixed vertex. Then, the effective resistance $R_{rs}$ of the network between $r,s \in V$ (Definition~\ref{def:electrical-network}) is given by
\begin{equation}
\label{eq:original-res}
R_{rs} = 
\begin{cases}
L(v)_{rr}^{-1} & \text{for $r \neq v, s = v$}\\
L(v)_{ss}^{-1} & \text{for $r = v, s \neq v$}\\
L(v)_{rr}^{-1} + L(v)_{ss}^{-1} - 2L(v)_{rs}^{-1} & \text{for $r,s \neq v$}
\end{cases}
\end{equation}
where $L(v)$ represents the Laplacian matrix of the graph with the $v$th row and column removed. Moreover, if the weight of the edge $e = (a,b)$ is altered by $-w_e \leq \alpha < +\infty$, i.e.,~the new weight is $w_e + \alpha$, then the new effective resistance values are given by
\begin{equation}
\label{eq:new-res}
R_{rs}' = 
\begin{cases}
R_{rs} - \beta \left(L(b)_{ar}^{-1}\right)^2 & \text{for $r \neq b$, $s=b$}\\
R_{rs} - \beta \left(L(b)_{as}^{-1}\right)^2 & \text{for $r = b$, $s \neq b$}\\
R_{rs} - \beta \left(L(b)_{ar}^{-1} - L(b)_{as}^{-1}\right)^2 & \text{for $r, s \neq b$}
\end{cases}
\end{equation}
where $\beta = \frac{\alpha}{1 + \alpha L(b)_{aa}^{-1}}$.
\end{theorem}

The values in Equation~\ref{eq:original-res} are standard in electrical network theory and the values
in Equation~\ref{eq:new-res} come about from using the Sherman--Morrison--Woodbury formula for computing the Laplacian
inverse entries after updating the weight of the edge. For details, see Theorem 5.2 and Equation 5.4
in~\cite{vos2016methods}.

\subsection{Inverses for reduced  {Laplacian}s of complete graphs and complete $k$-partite graphs}\label{sec:gen_inverse}

The following theorem gives a closed-form expression for the Laplacian $L$ of a complete graph and the inverse of the reduced matrix $L(v)^{-1}$.  
\begin{theorem}\label{thm:laplacian_complete}
The Laplacian $L$ of a complete graph on $n$ vertices has entries $L_{ii} = n-1$ for $i=1, \ldots, n$ and $L_{ij} = -1$ for all $i \neq j$. For any vertex $v$, after eliminating the $v$th row and column from $L$, the inverse $L(v)^{-1}$ has entries $L(v)^{-1}_{ii} = \frac{2}{n}$ and $L(v)^{-1}_{ij} = \frac{1}{n}$ for all $i \neq j$.
\end{theorem}

Theorem~\ref{thm:laplacian_complete} can be verified by direct matrix multiplication calculations.

The following theorem describes the closed-form expression for the Laplacian $L$ of a complete $k$-partite graph and the inverse of the reduced Laplacian $L(v)^{-1}$.
We use $\mathds{1}_{d\times d}$ and $\mathds{I}_{d\times d}$ to denote the all ones and the identity matrices of size $d\times d$, respectively. 

\begin{theorem}\label{thm:laplacian_kpartite}
Let $P = \{p_1,\ldots,p_k\}$ represent the set of $k$ partitions of a complete $k$-partite graph $G$ and let $n = \sum_{i=1}^k | p_i| $. The Laplacian matrix of $G$ is a block matrix with blocks of  size $| p_i| $ by $| p_j| $. For any $v\in V$,  $L(v)$ has the following block representation. Let $| \widetilde{p_i}|  = | p_i|  -1$. For ease of exposition, we give the formulas when $v\in p_1$; when $v$ is in the other partitions, the formulas are modified accordingly.
\begin{equation*}
\begin{bmatrix}
(n - | p_1| )\mathds{I}_{| \widetilde{p_1}|  x | \widetilde{p_1}| } & -\mathds{1}_{| \widetilde{p_1}|  x | p_2| } & ... & -\mathds{1}_{| \widetilde{p_1}|  x | p_k| }\\
-\mathds{1}_{| p_2|  x | \widetilde{p_1}| } & (n-| p_2| )\mathds{I}_{| p_2|  x | p_2| } & \ddots & \vdots \\
\vdots & \ddots & \ddots & -\mathds{1}_{| p_{k-1}|  x | p_{k}| } \\
-\mathds{1}_{| p_k|  x | \widetilde{p_1}| } & \dots & -\mathds{1}_{| p_{k}|  x | p_{k-1}| } & (n-| p_k| )\mathds{I}_{| p_k|  x | p_k| }
\end{bmatrix}
\end{equation*}

The inverse of $L(v)$ is given by the following block representation multiplied by the factor $\frac{1}{n-| p_1| }$.
\begin{equation*}
\begin{bmatrix}
\mathds{1}_{| \widetilde{p_1}|  x | \widetilde{p_1}| } + \mathds{I}_{| \widetilde{p_1}|  x | \widetilde{p_1}| } & \mathds{1}_{| \widetilde{p_1}|  x | p_2| } & \mathds{1}_{| \widetilde{p_1}|  x | p_3| } & \dots & \mathds{1}_{| \widetilde{p_1}|  x | p_k| } \\
\mathds{1}_{| p_2|  x | \widetilde{p_1}| } & D_{p_2} & \frac{n-1}{n}\mathds{1}_{| p_2|  x | p_3| } & \dots & \frac{n-1}{n}\mathds{1}_{| p_2|  x | p_k| }\\
\vdots & \frac{n-1}{n}\mathds{1}_{| p_3|  x | p_2| } & \ddots & \ddots & \vdots \\
\vdots & \vdots & \ddots & \ddots & \frac{n-1}{n}\mathds{1}_{| p_{k-1}|  x | p_k| } \\
\mathds{1}_{| p_k|  x | \widetilde{p_1}| } & \frac{n-1}{n}\mathds{1}_{| p_k|  x | p_2| } & \dots & \frac{n-1}{n}\mathds{1}_{| p_k|  x | p_{k-1}| } & D_{p_k}
\end{bmatrix} 
\end{equation*}
where
\begin{equation*}
D_{p_i}  =  \frac{| \widetilde{p_1}| n + \sum_{p_j \in P/\{p_1, p_i\}}{| p_j| (n-1)}}{n(n-| p_i| )} \mathds{1}_{| p_i|  x | p_i| } + \frac{n-| p_1| }{n-| p_i| }\mathds{I}_{| p_i|  x | p_i| }.
\end{equation*} 
\end{theorem}

Theorem~\ref{thm:laplacian_kpartite} can also be verified by direct matrix multiplication calculations.

\subsection{Tools for measurement strategies}

Many of our measurement strategies given in our upper bound proofs are based on the following type of measurement.
\begin{definition}\label{butterfly_def}
Let $(V,E,\mathbf{1})$ be an electrical network. A \textit{butterfly wing measurement} involving vertices $v_1,v_2,v_3 \in V$ is comprised of two measurements $(v_1,v_2)$ and $(v_2,v_3)$. We refer to $v_2$ as the center node (involved in both measurements) and $v_1, v_3$ as the wing nodes. See Figure~\ref{fig:butterfly_wing_measurement}.
\end{definition}

\begin{figure}[htbp]
    \centering
    
    \begin{tikzpicture}[x=0.75pt,y=0.75pt,yscale=-1,xscale=1]
    
    \draw   (119.5,19.5) .. controls (119.5,13.7) and (124.2,9) .. (130,9) .. controls (135.8,9) and (140.5,13.7) .. (140.5,19.5) .. controls (140.5,25.3) and (135.8,30) .. (130,30) .. controls (124.2,30) and (119.5,25.3) .. (119.5,19.5) -- cycle ;
    \draw   (119.5,61) .. controls (119.5,55.2) and (124.2,50.5) .. (130,50.5) .. controls (135.8,50.5) and (140.5,55.2) .. (140.5,61) .. controls (140.5,66.8) and (135.8,71.5) .. (130,71.5) .. controls (124.2,71.5) and (119.5,66.8) .. (119.5,61) -- cycle ;
    \draw   (180.5,40.5) .. controls (180.5,34.7) and (185.2,30) .. (191,30) .. controls (196.8,30) and (201.5,34.7) .. (201.5,40.5) .. controls (201.5,46.3) and (196.8,51) .. (191,51) .. controls (185.2,51) and (180.5,46.3) .. (180.5,40.5) -- cycle ;
    \draw    (139,24.5) -- (180.5,40.5) ;
    \draw    (139,55) -- (180.5,40.5) ;
    
    \draw (123,15) node [anchor=north west][inner sep=0.75pt]  [font=\footnotesize] [align=left] {$\displaystyle v_{1}$};
    \draw (123,57) node [anchor=north west][inner sep=0.75pt]  [font=\footnotesize] [align=left] {$\displaystyle v_{3}$};
    \draw (184,36) node [anchor=north west][inner sep=0.75pt]  [font=\footnotesize] [align=left] {$\displaystyle v_{2}$};

    \end{tikzpicture}

    \caption{Example of a butterfly wing measurement between vertices $v_1, v_2, v_3$}
    \label{fig:butterfly_wing_measurement}
\end{figure}


We now describe four types of butterfly wing measurements used in the strategies for complete $k$-partite graphs (see Figures~\ref{fig:tripartite_butterfly_wing}, \ref{fig:zig_zagging_butterfly_wing_scheme}, \ref{fig:hairpin_butterfly_wing}, \ref{fig:partition_butterfly_wing}).
\begin{definition}\label{def:tri_bw}
Given partitions $p_\alpha, p_\beta, p_\gamma$ in a $k$-partite graph with $| p_\alpha|  \leq | p_\beta|  \leq | p_\gamma| $, a \textit{tripartite butterfly wing} is formed using vertices $v_1 \in p_\alpha, v_3 \in p_\gamma, v_2 \in p_\beta$ such that $v_3$ is the center node.
\end{definition}

\begin{center}
    
    \begin{tikzpicture}[x=0.75pt,y=0.75pt,yscale=-1,xscale=1]
    
    \draw   (70,103.75) .. controls (70,57.22) and (78.95,19.5) .. (90,19.5) .. controls (101.05,19.5) and (110,57.22) .. (110,103.75) .. controls (110,150.28) and (101.05,188) .. (90,188) .. controls (78.95,188) and (70,150.28) .. (70,103.75) -- cycle ;
    \draw   (130.5,104.75) .. controls (130.5,58.22) and (139.45,20.5) .. (150.5,20.5) .. controls (161.55,20.5) and (170.5,58.22) .. (170.5,104.75) .. controls (170.5,151.28) and (161.55,189) .. (150.5,189) .. controls (139.45,189) and (130.5,151.28) .. (130.5,104.75) -- cycle ;
    \draw   (190.5,104.75) .. controls (190.5,58.22) and (199.45,20.5) .. (210.5,20.5) .. controls (221.55,20.5) and (230.5,58.22) .. (230.5,104.75) .. controls (230.5,151.28) and (221.55,189) .. (210.5,189) .. controls (199.45,189) and (190.5,151.28) .. (190.5,104.75) -- cycle ;
    \draw  [fill={rgb, 255:red, 0; green, 0; blue, 0 }  ,fill opacity=1 ] (87.5,59.25) .. controls (87.5,57.46) and (88.96,56) .. (90.75,56) .. controls (92.54,56) and (94,57.46) .. (94,59.25) .. controls (94,61.04) and (92.54,62.5) .. (90.75,62.5) .. controls (88.96,62.5) and (87.5,61.04) .. (87.5,59.25) -- cycle ;
    \draw  [fill={rgb, 255:red, 0; green, 0; blue, 0 }  ,fill opacity=1 ] (87.5,80.25) .. controls (87.5,78.46) and (88.96,77) .. (90.75,77) .. controls (92.54,77) and (94,78.46) .. (94,80.25) .. controls (94,82.04) and (92.54,83.5) .. (90.75,83.5) .. controls (88.96,83.5) and (87.5,82.04) .. (87.5,80.25) -- cycle ;
    \draw    (148.75,60.25) -- (209.75,60.25) ;
    \draw    (90.75,59.25) .. controls (131.5,79) and (169.5,80) .. (209.75,60.25) ;
    \draw  [fill={rgb, 255:red, 0; green, 0; blue, 0 }  ,fill opacity=1 ] (87.5,102.25) .. controls (87.5,100.46) and (88.96,99) .. (90.75,99) .. controls (92.54,99) and (94,100.46) .. (94,102.25) .. controls (94,104.04) and (92.54,105.5) .. (90.75,105.5) .. controls (88.96,105.5) and (87.5,104.04) .. (87.5,102.25) -- cycle ;
    \draw  [fill={rgb, 255:red, 0; green, 0; blue, 0 }  ,fill opacity=1 ] (87.5,151.25) .. controls (87.5,149.46) and (88.96,148) .. (90.75,148) .. controls (92.54,148) and (94,149.46) .. (94,151.25) .. controls (94,153.04) and (92.54,154.5) .. (90.75,154.5) .. controls (88.96,154.5) and (87.5,153.04) .. (87.5,151.25) -- cycle ;
    \draw  [fill={rgb, 255:red, 0; green, 0; blue, 0 }  ,fill opacity=1 ] (147.5,60.25) .. controls (147.5,58.46) and (148.96,57) .. (150.75,57) .. controls (152.54,57) and (154,58.46) .. (154,60.25) .. controls (154,62.04) and (152.54,63.5) .. (150.75,63.5) .. controls (148.96,63.5) and (147.5,62.04) .. (147.5,60.25) -- cycle ;
    \draw  [fill={rgb, 255:red, 0; green, 0; blue, 0 }  ,fill opacity=1 ] (147.5,81.25) .. controls (147.5,79.46) and (148.96,78) .. (150.75,78) .. controls (152.54,78) and (154,79.46) .. (154,81.25) .. controls (154,83.04) and (152.54,84.5) .. (150.75,84.5) .. controls (148.96,84.5) and (147.5,83.04) .. (147.5,81.25) -- cycle ;
    \draw  [fill={rgb, 255:red, 0; green, 0; blue, 0 }  ,fill opacity=1 ] (147.5,103.25) .. controls (147.5,101.46) and (148.96,100) .. (150.75,100) .. controls (152.54,100) and (154,101.46) .. (154,103.25) .. controls (154,105.04) and (152.54,106.5) .. (150.75,106.5) .. controls (148.96,106.5) and (147.5,105.04) .. (147.5,103.25) -- cycle ;
    \draw  [fill={rgb, 255:red, 0; green, 0; blue, 0 }  ,fill opacity=1 ] (147.5,152.25) .. controls (147.5,150.46) and (148.96,149) .. (150.75,149) .. controls (152.54,149) and (154,150.46) .. (154,152.25) .. controls (154,154.04) and (152.54,155.5) .. (150.75,155.5) .. controls (148.96,155.5) and (147.5,154.04) .. (147.5,152.25) -- cycle ;
    \draw  [fill={rgb, 255:red, 0; green, 0; blue, 0 }  ,fill opacity=1 ] (206.5,60.25) .. controls (206.5,58.46) and (207.96,57) .. (209.75,57) .. controls (211.54,57) and (213,58.46) .. (213,60.25) .. controls (213,62.04) and (211.54,63.5) .. (209.75,63.5) .. controls (207.96,63.5) and (206.5,62.04) .. (206.5,60.25) -- cycle ;
    \draw  [fill={rgb, 255:red, 0; green, 0; blue, 0 }  ,fill opacity=1 ] (206.5,81.25) .. controls (206.5,79.46) and (207.96,78) .. (209.75,78) .. controls (211.54,78) and (213,79.46) .. (213,81.25) .. controls (213,83.04) and (211.54,84.5) .. (209.75,84.5) .. controls (207.96,84.5) and (206.5,83.04) .. (206.5,81.25) -- cycle ;
    \draw  [fill={rgb, 255:red, 0; green, 0; blue, 0 }  ,fill opacity=1 ] (206.5,103.25) .. controls (206.5,101.46) and (207.96,100) .. (209.75,100) .. controls (211.54,100) and (213,101.46) .. (213,103.25) .. controls (213,105.04) and (211.54,106.5) .. (209.75,106.5) .. controls (207.96,106.5) and (206.5,105.04) .. (206.5,103.25) -- cycle ;
    \draw  [fill={rgb, 255:red, 0; green, 0; blue, 0 }  ,fill opacity=1 ] (206.5,152.25) .. controls (206.5,150.46) and (207.96,149) .. (209.75,149) .. controls (211.54,149) and (213,150.46) .. (213,152.25) .. controls (213,154.04) and (211.54,155.5) .. (209.75,155.5) .. controls (207.96,155.5) and (206.5,154.04) .. (206.5,152.25) -- cycle ;
    
    \draw (81.5,190.5) node [anchor=north west][inner sep=0.75pt]   [align=left] {$\displaystyle p_{\alpha }$};
    \draw (142.5,190) node [anchor=north west][inner sep=0.75pt]   [align=left] {$\displaystyle p_{\beta }$};
    \draw (200.5,190.5) node [anchor=north west][inner sep=0.75pt]   [align=left] {$\displaystyle p_{\gamma }$};
    \draw (93.04,123) node [anchor=north west][inner sep=0.75pt]  [rotate=-90] [align=left] {...};
    \draw (153,123) node [anchor=north west][inner sep=0.75pt]  [rotate=-90] [align=left] {...};
    \draw (212,123) node [anchor=north west][inner sep=0.75pt]  [rotate=-90] [align=left] {...};
    \draw (83,37) node [anchor=north west][inner sep=0.75pt]  [font=\small] [align=left] {$\displaystyle v_{1}$};
    \draw (143,38) node [anchor=north west][inner sep=0.75pt]  [font=\small] [align=left] {$\displaystyle v_{2}$};
    \draw (202,37) node [anchor=north west][inner sep=0.75pt]  [font=\small] [align=left] {$\displaystyle v_{3}$};

    \end{tikzpicture}

    \captionof{figure}{Example of a tripartite butterfly wing using vertices from partitions $p_\alpha$, $p_\beta$ and $p_\gamma$}
    \label{fig:tripartite_butterfly_wing}
\end{center}


\begin{definition}\label{def:zigzag_bip_bw}
Given partitions $p_\beta, p_\gamma$ in a $k$-partite graph with $| p_\beta|  \leq | p_\gamma| $,  a \textit{zig-zagging butterfly wing scheme} is formed by utilizing vertices $v_1,v_3,v_5 \in p_\gamma$ and $v_2,v_4,v_6 \in p_\beta$. This scheme involves two distinct butterfly wings. The first butterfly wing involves $v_1,v_2,v_3$ with $v_2 \in p_\beta$ as the center node. The second involves $v_4,v_5,v_6$ with $v_5 \in p_\gamma$ as the center node.  
\end{definition}

\begin{center}

    \begin{tikzpicture}[x=0.75pt,y=0.75pt,yscale=-1,xscale=1]
    
    \draw   (150.5,124.75) .. controls (150.5,78.22) and (159.45,40.5) .. (170.5,40.5) .. controls (181.55,40.5) and (190.5,78.22) .. (190.5,124.75) .. controls (190.5,171.28) and (181.55,209) .. (170.5,209) .. controls (159.45,209) and (150.5,171.28) .. (150.5,124.75) -- cycle ;
    \draw   (210.5,124.75) .. controls (210.5,78.22) and (219.45,40.5) .. (230.5,40.5) .. controls (241.55,40.5) and (250.5,78.22) .. (250.5,124.75) .. controls (250.5,171.28) and (241.55,209) .. (230.5,209) .. controls (219.45,209) and (210.5,171.28) .. (210.5,124.75) -- cycle ;
    \draw  [fill={rgb, 255:red, 0; green, 0; blue, 0 }  ,fill opacity=1 ] (167.5,79.25) .. controls (167.5,77.46) and (168.96,76) .. (170.75,76) .. controls (172.54,76) and (174,77.46) .. (174,79.25) .. controls (174,81.04) and (172.54,82.5) .. (170.75,82.5) .. controls (168.96,82.5) and (167.5,81.04) .. (167.5,79.25) -- cycle ;
    \draw  [fill={rgb, 255:red, 0; green, 0; blue, 0 }  ,fill opacity=1 ] (167.5,108.25) .. controls (167.5,106.46) and (168.96,105) .. (170.75,105) .. controls (172.54,105) and (174,106.46) .. (174,108.25) .. controls (174,110.04) and (172.54,111.5) .. (170.75,111.5) .. controls (168.96,111.5) and (167.5,110.04) .. (167.5,108.25) -- cycle ;
    \draw  [fill={rgb, 255:red, 0; green, 0; blue, 0 }  ,fill opacity=1 ] (167.5,132.25) .. controls (167.5,130.46) and (168.96,129) .. (170.75,129) .. controls (172.54,129) and (174,130.46) .. (174,132.25) .. controls (174,134.04) and (172.54,135.5) .. (170.75,135.5) .. controls (168.96,135.5) and (167.5,134.04) .. (167.5,132.25) -- cycle ;
    \draw  [fill={rgb, 255:red, 0; green, 0; blue, 0 }  ,fill opacity=1 ] (167.5,172.25) .. controls (167.5,170.46) and (168.96,169) .. (170.75,169) .. controls (172.54,169) and (174,170.46) .. (174,172.25) .. controls (174,174.04) and (172.54,175.5) .. (170.75,175.5) .. controls (168.96,175.5) and (167.5,174.04) .. (167.5,172.25) -- cycle ;
    \draw  [fill={rgb, 255:red, 0; green, 0; blue, 0 }  ,fill opacity=1 ] (227.5,79.25) .. controls (227.5,77.46) and (228.96,76) .. (230.75,76) .. controls (232.54,76) and (234,77.46) .. (234,79.25) .. controls (234,81.04) and (232.54,82.5) .. (230.75,82.5) .. controls (228.96,82.5) and (227.5,81.04) .. (227.5,79.25) -- cycle ;
    \draw  [fill={rgb, 255:red, 0; green, 0; blue, 0 }  ,fill opacity=1 ] (227.5,108.25) .. controls (227.5,106.46) and (228.96,105) .. (230.75,105) .. controls (232.54,105) and (234,106.46) .. (234,108.25) .. controls (234,110.04) and (232.54,111.5) .. (230.75,111.5) .. controls (228.96,111.5) and (227.5,110.04) .. (227.5,108.25) -- cycle ;
    \draw  [fill={rgb, 255:red, 0; green, 0; blue, 0 }  ,fill opacity=1 ] (227.5,132.25) .. controls (227.5,130.46) and (228.96,129) .. (230.75,129) .. controls (232.54,129) and (234,130.46) .. (234,132.25) .. controls (234,134.04) and (232.54,135.5) .. (230.75,135.5) .. controls (228.96,135.5) and (227.5,134.04) .. (227.5,132.25) -- cycle ;
    \draw  [fill={rgb, 255:red, 0; green, 0; blue, 0 }  ,fill opacity=1 ] (227.5,172.25) .. controls (227.5,170.46) and (228.96,169) .. (230.75,169) .. controls (232.54,169) and (234,170.46) .. (234,172.25) .. controls (234,174.04) and (232.54,175.5) .. (230.75,175.5) .. controls (228.96,175.5) and (227.5,174.04) .. (227.5,172.25) -- cycle ;
    \draw    (170.75,79.25) -- (230.75,79.25) ;
    \draw    (170.75,79.25) -- (230.75,108.25) ;
    \draw    (170.75,108.25) -- (230.75,132.25) ;
    \draw    (170.75,132.25) -- (230.75,132.25) ;
    
    \draw (162.5,210) node [anchor=north west][inner sep=0.75pt]   [align=left] {$\displaystyle p_{\beta }$};
    \draw (220.5,210.5) node [anchor=north west][inner sep=0.75pt]   [align=left] {$\displaystyle p_{\gamma }$};
    \draw (173,144) node [anchor=north west][inner sep=0.75pt]  [rotate=-90.29] [align=left] {...};
    \draw (233,144) node [anchor=north west][inner sep=0.75pt]  [rotate=-90.29] [align=left] {...};
    \draw (164,65) node [anchor=north west][inner sep=0.75pt]  [font=\small] [align=left] {$\displaystyle v_{2}$};
    \draw (224,65) node [anchor=north west][inner sep=0.75pt]  [font=\small] [align=left] {$\displaystyle v_{1}$};
    \draw (224,94) node [anchor=north west][inner sep=0.75pt]  [font=\small] [align=left] {$\displaystyle v_{3}$};
    \draw (224,118) node [anchor=north west][inner sep=0.75pt]  [font=\small] [align=left] {$\displaystyle v_{5}$};
    \draw (164,94) node [anchor=north west][inner sep=0.75pt]  [font=\small] [align=left] {$\displaystyle v_{4}$};
    \draw (164,118) node [anchor=north west][inner sep=0.75pt]  [font=\small] [align=left] {$\displaystyle v_{6}$};

    \end{tikzpicture}

    \captionof{figure}{Example of a zig-zagging butterfly wing scheme using vertices in partitions $p_\beta$ and $p_\gamma$}
    \label{fig:zig_zagging_butterfly_wing_scheme}
\end{center}


\begin{definition}\label{def:hairpin_bip_bw}
Given partitions $p_\beta, p_\gamma$ in a $k$-partite graph with $| p_\beta|  \leq | p_\gamma| $, we form a \textit{hairpin butterfly wing} by utilizing vertices $v_1,v,v_2$ with $v$ as the center node, satisfying the following: either $v \in p_\beta$ and $v_1, v_2 \in p_\gamma$, or else $v \in p_\gamma$, and $v_1, v_2 \in p_\beta$.
\end{definition}

\begin{center}
    
    \begin{tikzpicture}[x=0.75pt,y=0.75pt,yscale=-1,xscale=1]
    
    \draw   (2,85.25) .. controls (2,38.72) and (10.95,1) .. (22,1) .. controls (33.05,1) and (42,38.72) .. (42,85.25) .. controls (42,131.78) and (33.05,169.5) .. (22,169.5) .. controls (10.95,169.5) and (2,131.78) .. (2,85.25) -- cycle ;
    \draw   (62,85.25) .. controls (62,38.72) and (70.95,1) .. (82,1) .. controls (93.05,1) and (102,38.72) .. (102,85.25) .. controls (102,131.78) and (93.05,169.5) .. (82,169.5) .. controls (70.95,169.5) and (62,131.78) .. (62,85.25) -- cycle ;
    \draw  [fill={rgb, 255:red, 0; green, 0; blue, 0 }  ,fill opacity=1 ] (19,39.75) .. controls (19,37.96) and (20.46,36.5) .. (22.25,36.5) .. controls (24.04,36.5) and (25.5,37.96) .. (25.5,39.75) .. controls (25.5,41.54) and (24.04,43) .. (22.25,43) .. controls (20.46,43) and (19,41.54) .. (19,39.75) -- cycle ;
    \draw  [fill={rgb, 255:red, 0; green, 0; blue, 0 }  ,fill opacity=1 ] (19,68.75) .. controls (19,66.96) and (20.46,65.5) .. (22.25,65.5) .. controls (24.04,65.5) and (25.5,66.96) .. (25.5,68.75) .. controls (25.5,70.54) and (24.04,72) .. (22.25,72) .. controls (20.46,72) and (19,70.54) .. (19,68.75) -- cycle ;
    \draw  [fill={rgb, 255:red, 0; green, 0; blue, 0 }  ,fill opacity=1 ] (19,92.75) .. controls (19,90.96) and (20.46,89.5) .. (22.25,89.5) .. controls (24.04,89.5) and (25.5,90.96) .. (25.5,92.75) .. controls (25.5,94.54) and (24.04,96) .. (22.25,96) .. controls (20.46,96) and (19,94.54) .. (19,92.75) -- cycle ;
    \draw  [fill={rgb, 255:red, 0; green, 0; blue, 0 }  ,fill opacity=1 ] (19,132.75) .. controls (19,130.96) and (20.46,129.5) .. (22.25,129.5) .. controls (24.04,129.5) and (25.5,130.96) .. (25.5,132.75) .. controls (25.5,134.54) and (24.04,136) .. (22.25,136) .. controls (20.46,136) and (19,134.54) .. (19,132.75) -- cycle ;
    \draw  [fill={rgb, 255:red, 0; green, 0; blue, 0 }  ,fill opacity=1 ] (79,39.75) .. controls (79,37.96) and (80.46,36.5) .. (82.25,36.5) .. controls (84.04,36.5) and (85.5,37.96) .. (85.5,39.75) .. controls (85.5,41.54) and (84.04,43) .. (82.25,43) .. controls (80.46,43) and (79,41.54) .. (79,39.75) -- cycle ;
    \draw  [fill={rgb, 255:red, 0; green, 0; blue, 0 }  ,fill opacity=1 ] (79,68.75) .. controls (79,66.96) and (80.46,65.5) .. (82.25,65.5) .. controls (84.04,65.5) and (85.5,66.96) .. (85.5,68.75) .. controls (85.5,70.54) and (84.04,72) .. (82.25,72) .. controls (80.46,72) and (79,70.54) .. (79,68.75) -- cycle ;
    \draw  [fill={rgb, 255:red, 0; green, 0; blue, 0 }  ,fill opacity=1 ] (79,92.75) .. controls (79,90.96) and (80.46,89.5) .. (82.25,89.5) .. controls (84.04,89.5) and (85.5,90.96) .. (85.5,92.75) .. controls (85.5,94.54) and (84.04,96) .. (82.25,96) .. controls (80.46,96) and (79,94.54) .. (79,92.75) -- cycle ;
    \draw  [fill={rgb, 255:red, 0; green, 0; blue, 0 }  ,fill opacity=1 ] (79,132.75) .. controls (79,130.96) and (80.46,129.5) .. (82.25,129.5) .. controls (84.04,129.5) and (85.5,130.96) .. (85.5,132.75) .. controls (85.5,134.54) and (84.04,136) .. (82.25,136) .. controls (80.46,136) and (79,134.54) .. (79,132.75) -- cycle ;
    \draw    (22.25,39.75) -- (82.25,39.75) ;
    \draw    (22.25,39.75) -- (82.25,68.75) ;
    
    \draw (14,170.5) node [anchor=north west][inner sep=0.75pt]   [align=left] {$\displaystyle p_{\beta }$};
    \draw (72,171) node [anchor=north west][inner sep=0.75pt]   [align=left] {$\displaystyle p_{\gamma }$};
    \draw (34.04,103.54) node [anchor=north west][inner sep=0.75pt]  [rotate=-90.29] [align=left] {...};
    \draw (94.04,104.54) node [anchor=north west][inner sep=0.75pt]  [rotate=-90.29] [align=left] {...};
    \draw (17,25) node [anchor=north west][inner sep=0.75pt]  [font=\small] [align=left] {$\displaystyle v$};
    \draw (75,25) node [anchor=north west][inner sep=0.75pt]  [font=\small] [align=left] {$\displaystyle v_{1}$};
    \draw (75,53) node [anchor=north west][inner sep=0.75pt]  [font=\small] [align=left] {$\displaystyle v_{2}$};

    \end{tikzpicture}

    \captionof{figure}{Example of a hairpin butterfly wing composed with vertices from partitions $p_\beta$ and $p_\gamma$}
    \label{fig:hairpin_butterfly_wing}
\end{center}


\begin{definition}\label{def:part_bw}
Given a partition $p_\gamma$ in a $k$-partite graph, a \textit{partition butterfly wing} is formed by utilizing vertices $v_1,v_2,v_3 \in p_\gamma$, with any of these three vertices as the center node.
\end{definition}

\begin{center}

    \begin{tikzpicture}[x=0.75pt,y=0.75pt,yscale=-1,xscale=1]
    
    \draw   (2,86.25) .. controls (2,39.72) and (10.95,2) .. (22,2) .. controls (33.05,2) and (42,39.72) .. (42,86.25) .. controls (42,132.78) and (33.05,170.5) .. (22,170.5) .. controls (10.95,170.5) and (2,132.78) .. (2,86.25) -- cycle ;
    \draw  [fill={rgb, 255:red, 0; green, 0; blue, 0 }  ,fill opacity=1 ] (19,40.75) .. controls (19,38.96) and (20.46,37.5) .. (22.25,37.5) .. controls (24.04,37.5) and (25.5,38.96) .. (25.5,40.75) .. controls (25.5,42.54) and (24.04,44) .. (22.25,44) .. controls (20.46,44) and (19,42.54) .. (19,40.75) -- cycle ;
    \draw  [fill={rgb, 255:red, 0; green, 0; blue, 0 }  ,fill opacity=1 ] (19,67.75) .. controls (19,65.96) and (20.46,64.5) .. (22.25,64.5) .. controls (24.04,64.5) and (25.5,65.96) .. (25.5,67.75) .. controls (25.5,69.54) and (24.04,71) .. (22.25,71) .. controls (20.46,71) and (19,69.54) .. (19,67.75) -- cycle ;
    \draw  [fill={rgb, 255:red, 0; green, 0; blue, 0 }  ,fill opacity=1 ] (19,93.75) .. controls (19,91.96) and (20.46,90.5) .. (22.25,90.5) .. controls (24.04,90.5) and (25.5,91.96) .. (25.5,93.75) .. controls (25.5,95.54) and (24.04,97) .. (22.25,97) .. controls (20.46,97) and (19,95.54) .. (19,93.75) -- cycle ;
    \draw  [fill={rgb, 255:red, 0; green, 0; blue, 0 }  ,fill opacity=1 ] (19,133.75) .. controls (19,131.96) and (20.46,130.5) .. (22.25,130.5) .. controls (24.04,130.5) and (25.5,131.96) .. (25.5,133.75) .. controls (25.5,135.54) and (24.04,137) .. (22.25,137) .. controls (20.46,137) and (19,135.54) .. (19,133.75) -- cycle ;
    \draw    (22.25,93.75) .. controls (70.5,91) and (70,69) .. (22.25,67.75) ;
    \draw    (22.25,67.75) .. controls (70.5,65) and (70,42) .. (22.25,40.75) ;
    
    \draw (12,172) node [anchor=north west][inner sep=0.75pt]   [align=left] {$\displaystyle p_{\gamma }$};
    \draw (24,105.54) node [anchor=north west][inner sep=0.75pt]  [rotate=-90.29] [align=left] {...};
    \draw (15,26) node [anchor=north west][inner sep=0.75pt]  [font=\small] [align=left] {$\displaystyle v_{1}$};
    \draw (15,53) node [anchor=north west][inner sep=0.75pt]  [font=\small] [align=left] {$\displaystyle v_{2}$};
    \draw (15,79) node [anchor=north west][inner sep=0.75pt]  [font=\small] [align=left] {$\displaystyle v_{3}$};

    \end{tikzpicture}

    \captionof{figure}{Example of a partition butterfly wing composed of vertices from partition $p_\gamma$}
    \label{fig:partition_butterfly_wing}
\end{center}

A conceptual tool that we use for our lower bound proofs is the idea of a \textit{measurement graph}. 

\begin{definition}\label{def:measurement-graph}
Let $(V,E,\mathbf{1})$ be an electrical network and let $\mathcal{M}$ be some set of effective resistance measurements, i.e.,~$\mathcal{M}$ is a set of unordered pairs of vertices. The graph $G' = (V, \mathcal{M})$ is the \textit{measurement graph} corresponding to $\mathcal{M}$.
\end{definition}

\begin{center}


    \captionof{figure} {Example of a graph $G$ and its measurement graph $G'$ for a set of measurements $\mathcal{M}$}
    \label{fig:measurement_graph}

\end{center}

Note that the edges in the measurement graph $G' = (V, \mathcal{M})$ are not necessarily edges in the underlying graph $G = (V,E)$ for the electrical network. See Figure~\ref{fig:measurement_graph}.

\section{Proofs of results}\label{sec:proofs}

{In this section, we give complete proofs of Theorems~\ref{thm:complete},~\ref{thm:bipartite},~\ref{thm:tripartite}, and~\ref{thm:k-partite}.
First, we provide the proofs for complete graphs, complete bipartite graphs, and complete tripartite graphs, which illustrate the general proof techniques for establishing upper and lower bounds on the smallest number of measurements needed for the faulty edge detection problem. The most technically involved proof is that of Theorem~\ref{thm:k-partite}. Its proof analyzes $k$-partite graphs for arbitrary $k\geq 2$ and is a much more intricate implementation of the general proof ideas illustrated in this section for the complete graph, complete bipartite, and complete tripartite cases. The high level idea in the proofs for the upper and lower bounds is to precisely quantify the amount of information one gains with any given measurement. Given any proposed measurement in an electrical network, two edges can be considered \textit{information theoretically equivalent} if the effective resistance values are the same when they are altered. In other words, this measurement will not be able to distinguish two equivalent edges. Therefore, one seeks to find the smallest set of measurements such that for every pair of edges, there exists a measurement that puts them in different equivalence classes. Consequently, it becomes important to understand the equivalence classes of the edges of the graph for any measurement. 
These equivalence classes are characterized using the analyses presented in  Appendices \ref{incident_edge_lemma_complete} and \ref{A:edge_partition_k-partite}, which are based on Theorem~\ref{thm:eff-res-change} and the values of $L(v)^{-1}$ given in  Theorems~\ref{thm:laplacian_complete} and~\ref{thm:laplacian_kpartite}.

\subsection{Complete Graphs}\label{sec:complete_graphs}

The characterization of the equivalence classes of the edges for any measurement given in  Lemma~\ref{lem:incident_edge_lemma_complete} from  Appendix~\ref{incident_edge_lemma_complete} will be an important tool below. In particular, the altered resistance values in Table~\ref{table:complete_graphs} will be referred to many times.

\subsubsection{Upper Bound in Theorem~\ref{thm:complete}}

Let the vertices of the complete graph be labeled as $v_1, \ldots, v_n$. We first prove that $\left\lceil \frac{2n}{3}\right\rceil$ measurements suffice and then establish a matching lower bound.

\paragraph*{Strategy}
Partition the vertices into groups of three: $v_1, v_2, v_3$ and $v_4, v_5, v_6$ and so forth, and depending on the value of $n$, we will have zero, one or two vertices left over that cannot be grouped. We use butterfly wing measurements (Definition~\ref{butterfly_def}) on all the groups of size three (with $v_2, v_5, v_8, \ldots $ as the center nodes). If we have one or two vertices left over that were not grouped, then we add measurements with source equal to $v_2$ (the center node in the first group of three) and sink equal to the left over vertices. We claim that these measurements suffice.

\paragraph*{Measurement Counts}
The number of measurements is $\frac{2n}{3}$ if $n \equiv 0 \mod{3}$, $\frac{2(n-1)}{3} + 1 = \frac{2n}{3} + \frac13$ if $n \equiv 1 \mod{3}$, or $\frac{2(n-2)}{3} + 2 = \frac{2n}{3} + \frac23$ if $n \equiv 2 \mod{3}$. We observe that these numbers are equal to $\left\lceil \frac{2n}{3} \right\rceil$ (in the respective cases).

\paragraph*{Correctness}
Consider two edges $e, e' \in E$. We have to show there exists a measurement which gives different effective resistance values when these two edges are altered. First, observe that at least one of the edges, say $e$ without loss of generality, must have an endpoint in some group of size three. This is because we have at most two vertices in the graph that are not included in some group of size three. We now do some case analysis.

\noindent\underline{\textit{Case 1: $e$ has both endpoints in the same group of size three.}}

If $e'$ does not have an endpoint in this group of size three, then any measurement in the butterfly wing will distinguish these two edges by  Lemma~\ref{lem:incident_edge_lemma_complete}. If $e'$ has both endpoints in this group of size three, then there must be a measurement in the butterfly wing in this group of size three whose endpoints coincide with either $e$ or $e'$, but not both. This measurement will distinguish these two edges by  Lemma~\ref{lem:incident_edge_lemma_complete}. Finally, consider the possibility that $e'$ has exactly one endpoint in this group of size three. If the other endpoint of $e'$ is in a different group of size three, then the measurement from the butterfly wing on this other group that shares an endpoint with $e'$ will distinguish $e,e'$ by  Lemma~\ref{lem:incident_edge_lemma_complete}. Else, the other endpoint $u$ of $e'$ is one of the left over vertices that were not grouped into groups of size three. If the butterfly wing in the group containing $e$ has a measurement that coincides with $e$, then that measurement will distinguish $e,e'$ by  Lemma~\ref{lem:incident_edge_lemma_complete}. So now suppose the butterfly wing in the group containing $e$ is such that the wing nodes are the endpoints of $e$. If $e'$ also has one of these wing nodes as its endpoint, then the measurement with the other wing node will distinguish $e,e'$ by  Lemma~\ref{lem:incident_edge_lemma_complete}. Otherwise, we are in the case where $e'$ has the center node $v$ of this butterfly wing as one of its endpoints. In this case,  Lemma~\ref{lem:incident_edge_lemma_complete} implies that the measurement using $u$ and $v_2$ will distinguish $e,e'$. This holds in the case when $v = v_2$, i.e.,~$v$ is the center node of the first butterfly, and also when $v$ is the center node of a different butterfly.

\noindent\underline{\textit{Case 2: The two endpoints of $e$ are in different groups of size three.}}

If $e'$ has both endpoints in the same group of size three, then this is symmetrical to Case 1. Thus, we assume $e'$ does not have endpoints in the same group of size three. Suppose $e$ and $e'$ do not share an endpoint. Consider any one of the groups containing an endpoint of $e$. At least one of the measurements in this butterfly wing will share an endpoint with $e$ but not $e'$, or vice versa. This measurement will distinguish $e,e'$ by  Lemma~\ref{lem:incident_edge_lemma_complete}. Suppose now that $e$ and $e'$ do share an endpoint. Consider the group of three containing the other endpoint of $e$. There must be a measurement in the butterfly wing in this group that shares an endpoint with $e$ or $e'$ but not both. This measurement will distinguish $e,e'$ by  Lemma~\ref{lem:incident_edge_lemma_complete}.

\noindent\underline{\textit{Case 3: One endpoint of $e$ is in a group of size three and the other one is not.}}

We may assume that one endpoint of $e'$ is in a group of size three and the other one is not. Otherwise, we are in symmetrical cases to Case 1 and Case 2. Thus, both $e$ and $e'$ have endpoints $u$ and $u'$ that were not grouped into the groups of size three.      Suppose $u\neq u'$. If $v_2$ is the endpoint of either $e$ or $e'$, we can use the measurement corresponding to that edge which will distinguish it from the other edge, by  Lemma~\ref{lem:incident_edge_lemma_complete}. Otherwise, $v_2$ is not the other endpoint of $e$ or $e'$; then the measurement with source $v_2$ and sink $u$ will distinguish the edges, by  Lemma~\ref{lem:incident_edge_lemma_complete}. Suppose now $u=u'$. If the other endpoints of $e$ and $e'$ are in different groups of size three, then any measurement from the butterfly wing in the group containing an endpoint of $e$ will distinguish $e,e'$ by  Lemma~\ref{lem:incident_edge_lemma_complete}. Finally, consider the case when the other endpoints of $e$ and $e'$ are in the same group of size three. Then one of the measurements in the butterfly wing on this group of size three shares an endpoint with $e$ or $e'$ but not both, which will distinguish $e,e'$ by  Lemma~\ref{lem:incident_edge_lemma_complete}.

\noindent\underline{\textit{Case 4: Neither endpoint of $e$ is in a group of size three.}}

This means both endpoints $u_1$ and $u_2$ of $e$ are left over vertices from the grouping into size three groups. One of these vertices, say $u_1$ without loss of generality, is not an endpoint of $e'$. If $v_2$ is not an endpoint of $e'$ then the measurement with source $v_2$ and sink $u_1$ will distinguish $e,e'$ by  Lemma~\ref{lem:incident_edge_lemma_complete}. Thus, assume $v_2$ is an endpoint of $e'$. Then, the measurement with source $v_1$ and sink $v_2$ (i.e.,~a measurement from the first group of size three) will distinguish $e,e'$ by  Lemma~\ref{lem:incident_edge_lemma_complete}.



\subsubsection{Lower Bound in Theorem~\ref{thm:complete}
} Let $\mathcal{M}$ be any set of measurements that solves the faulty edge detection problem for the complete graph on $n$ vertices and consider the measurement graph $G' = (V, \mathcal{M})$ from  Definition~\ref{def:measurement-graph}. We first observe that there can be no connected component of size two in $G'$. If we had a component of size two $C_1 = \{r,s\}$, then for any vertex $d \in V\setminus\{r,s\}$ (which exists since $n\geq 6$), we will not be able to distinguish edge $(r,d)$ from $(s,d)$ by  Lemma~\ref{lem:incident_edge_lemma_complete}. We now consider two cases:

\newpage

\noindent\underline{\textit{Case 1: $G'$ has an isolated vertex.}}

First, there can be at most one isolated vertex in $G'$. Indeed, suppose $u$ and $w$ are both isolated vertices in $G'$, then for any vertex $v \in V\setminus \{u,w\}$, no measurement will distinguish the edges $(u,v)$ and $(w,v)$ by  Lemma~\ref{lem:incident_edge_lemma_complete}. Moreover, since we have no size two components in $G'$ by the observation before, all other connected components are of size three or more. We now claim that all size three connected components of $G'$ must be triangles, i.e.,~all the edges must be measurements in $\mathcal{M}$. Otherwise, suppose we have a size three component with vertices $a,b,c$ and edges $(a,b)$ and $(a,c)$. Let the isolated vertex in $G'$ be $v$. Then no measurement can distinguish the edges $(v,a)$ and $(b,c)$ by  Lemma~\ref{lem:incident_edge_lemma_complete}. Let $q$ be the number of size three components in $G'$. Every component of size $s\geq 4$ must have at least $s - 1 \geq 3s/4$ edges. Thus, the number of measurements in $\mathcal{M}$ is at least
\begin{equation*}
3q + \frac{3(n - 3q - 1)}{4} = \frac{3n}{4} + \frac{3q}{4} - \frac34 \geq \frac{3n}{4} - \frac34.
\end{equation*}
Since the number of measurements is an integer, we have a lower bound of $\left\lceil \frac{3n}{4} - \frac34 \right\rceil$ which is at least $\left \lceil \frac{2n}{3}\right \rceil$ for all $n \geq 6$.

\noindent\underline{\textit{Case 2: $G'$ has no isolated vertices.}}

In this case, each component of $G'$ has at least three vertices, and any such component of size $s \geq 3$ must have at least $s-1 \geq \frac{2s}{3}$ edges. Thus, overall, $G'$ must have at least $\frac{2n}{3}$ edges and we have our lower bound of $\left \lceil \frac{2n}{3}\right \rceil$ since the number of measurements must be an integer.

\subsection{Complete Bipartite Graphs}\label{sec:bipartite_graphs}

The characterization of the edge equivalence classes for any measurement in a complete $k$-partite graph, for $k\geq 2$, given in Lemma~\ref{incdient_edge_lemma_for_kpartite} from Appendix~\ref{A:edge_partition_k-partite} will be an important tool below. In particular, the altered resistance values in Tables~\ref{table:k_part-diff-part} and \ref{table:k_part-same-part} will be referred to many times. In these two tables, we use the convention that for any measurement $(r,s)$, we use $p_r$ and $p_s$ to denote the partitions containing $r$ and $s$, respectively.

\subsubsection{Upper Bound in Theorem~\ref{thm:bipartite}}

Let $B$ be a complete bipartite graph comprised of partitions $p_\beta,p_\gamma$ such that $| p_\beta|  \leq | p_\gamma| $. We first prove that $\left\lfloor{\frac{2}{3}| p_\gamma|  + \frac{1}{3}| p_\beta| }\right\rfloor$ measurements suffice when $| p_\beta|  < | p_\gamma| $ and $\left\lfloor{\frac{2n}{3}}\right\rfloor$ measurements suffice when $| p_\beta|  = | p_\gamma| $.

\paragraph*{Strategy}
We have two variants of the strategy to deal with the cases when the partitions are different sizes and when they have the same size.

\begin{definition}[Bipartite Graph Strategy with Different Sized Partitions]\label{def:bip_different_sizes_strategy} Given a complete bipartite graph $B$ comprised of partitions $p_\beta,p_\gamma$ such that $| p_\beta|  < | p_\gamma| $. Maintain two designated nodes $i_\beta \in p_\beta$ and $i_\gamma \in p_\gamma$. First, place as many disjoint measurements $(u,v)$ such that $u \in p_\beta\setminus\{i_\beta\}$ and $v \in p_\gamma\setminus\{i_\gamma\}$. If there are nodes left in $p_\gamma\setminus\{i_\gamma\}$ that are not designated or used in a measurement, group them into groups of size three and place butterfly wing measurements on these groups.  If one node $v_\gamma$ remains in $p_\gamma\setminus\{i_\gamma\}$ that has not been utilized in any measurement so far, make measurement $(v_\gamma,w_\gamma)$ where $w_\gamma$ is any node in $p_\gamma\setminus\{v_\gamma,i_\gamma\}$. If two nodes $u_\gamma, v_\gamma$ remain that have not been utilized in any measurement, place a butterfly wing measurement involving $u_\gamma, v_\gamma, i_\gamma$.
\end{definition}

\begin{definition}[Bipartite Graph Strategy with Same Sized Partitions]\label{def:bip_same_size_strategy}     Given a complete bipartite graph $B$ comprised of partitions $p_\beta,p_\gamma$ such that $| p_\beta|  = | p_\gamma| $. Maintain two designated nodes $i_\beta \in p_\beta$ and $i_\gamma \in p_\gamma$.
Place as many disjoint zig-zagging bipartite wing measurements (Definition~\ref{def:zigzag_bip_bw}) as possible. If one node $v \in p_\beta\setminus\{i_\beta\}$ remains that is unable to be placed in a zig-zagging bipartite wing measurement scheme, we place a hairpin butterfly wing measurement (Definition~\ref{def:hairpin_bip_bw}) disjoint from all measurements so far, with $v$ as the center node and $i_\gamma$ as one of the wing nodes. If two nodes $u,v \in p_\beta\setminus\{i_\beta\}$ remain that are unable to be placed in a zig-zagging bipartite wing measurement scheme, we place a disjoint hairpin butterfly wing measurement with $w \in p_\gamma\setminus\{i_\gamma\}$ as its center node, such that $w$ has not been used in a previously described measurement. There should be one remaining node $t \in p_\gamma\setminus\{i_\gamma\}$ that is not involved in any measurement so far. We place one additional measurement from $t$ to $w$.
\end{definition}

\begin{figure}[htbp]

    \centering      
    
    \begin{tikzpicture}[x=0.75pt,y=0.75pt,yscale=-1,xscale=1]
    
    \draw   (86.5,72.2) .. controls (86.5,43.92) and (95.45,21) .. (106.5,21) .. controls (117.55,21) and (126.5,43.92) .. (126.5,72.2) .. controls (126.5,100.48) and (117.55,123.4) .. (106.5,123.4) .. controls (95.45,123.4) and (86.5,100.48) .. (86.5,72.2) -- cycle ;
    \draw   (146.5,97.5) .. controls (146.5,56.08) and (155.45,22.5) .. (166.5,22.5) .. controls (177.55,22.5) and (186.5,56.08) .. (186.5,97.5) .. controls (186.5,138.92) and (177.55,172.5) .. (166.5,172.5) .. controls (155.45,172.5) and (146.5,138.92) .. (146.5,97.5) -- cycle ;
    \draw  [fill={rgb, 255:red, 0; green, 0; blue, 0 }  ,fill opacity=1 ] (103.5,50.75) .. controls (103.5,48.96) and (104.96,47.5) .. (106.75,47.5) .. controls (108.54,47.5) and (110,48.96) .. (110,50.75) .. controls (110,52.54) and (108.54,54) .. (106.75,54) .. controls (104.96,54) and (103.5,52.54) .. (103.5,50.75) -- cycle ;
    \draw  [fill={rgb, 255:red, 0; green, 0; blue, 0 }  ,fill opacity=1 ] (103.5,71.75) .. controls (103.5,69.96) and (104.96,68.5) .. (106.75,68.5) .. controls (108.54,68.5) and (110,69.96) .. (110,71.75) .. controls (110,73.54) and (108.54,75) .. (106.75,75) .. controls (104.96,75) and (103.5,73.54) .. (103.5,71.75) -- cycle ;
    \draw  [fill={rgb, 255:red, 0; green, 0; blue, 0 }  ,fill opacity=1 ] (103.5,91.75) .. controls (103.5,89.96) and (104.96,88.5) .. (106.75,88.5) .. controls (108.54,88.5) and (110,89.96) .. (110,91.75) .. controls (110,93.54) and (108.54,95) .. (106.75,95) .. controls (104.96,95) and (103.5,93.54) .. (103.5,91.75) -- cycle ;
    \draw  [fill={rgb, 255:red, 0; green, 0; blue, 0 }  ,fill opacity=1 ] (163.5,50.75) .. controls (163.5,48.96) and (164.96,47.5) .. (166.75,47.5) .. controls (168.54,47.5) and (170,48.96) .. (170,50.75) .. controls (170,52.54) and (168.54,54) .. (166.75,54) .. controls (164.96,54) and (163.5,52.54) .. (163.5,50.75) -- cycle ;
    \draw  [fill={rgb, 255:red, 0; green, 0; blue, 0 }  ,fill opacity=1 ] (163.5,71.75) .. controls (163.5,69.96) and (164.96,68.5) .. (166.75,68.5) .. controls (168.54,68.5) and (170,69.96) .. (170,71.75) .. controls (170,73.54) and (168.54,75) .. (166.75,75) .. controls (164.96,75) and (163.5,73.54) .. (163.5,71.75) -- cycle ;
    \draw  [fill={rgb, 255:red, 0; green, 0; blue, 0 }  ,fill opacity=1 ] (163.5,92.75) .. controls (163.5,90.96) and (164.96,89.5) .. (166.75,89.5) .. controls (168.54,89.5) and (170,90.96) .. (170,92.75) .. controls (170,94.54) and (168.54,96) .. (166.75,96) .. controls (164.96,96) and (163.5,94.54) .. (163.5,92.75) -- cycle ;
    \draw    (106.75,50.75) -- (166.75,50.75) ;
    \draw    (106.75,71.75) -- (166.75,71.75) ;
    \draw  [fill={rgb, 255:red, 0; green, 0; blue, 0 }  ,fill opacity=1 ] (163.5,141.75) .. controls (163.5,139.96) and (164.96,138.5) .. (166.75,138.5) .. controls (168.54,138.5) and (170,139.96) .. (170,141.75) .. controls (170,143.54) and (168.54,145) .. (166.75,145) .. controls (164.96,145) and (163.5,143.54) .. (163.5,141.75) -- cycle ;
    \draw  [fill={rgb, 255:red, 0; green, 0; blue, 0 }  ,fill opacity=1 ] (163.5,118.75) .. controls (163.5,116.96) and (164.96,115.5) .. (166.75,115.5) .. controls (168.54,115.5) and (170,116.96) .. (170,118.75) .. controls (170,120.54) and (168.54,122) .. (166.75,122) .. controls (164.96,122) and (163.5,120.54) .. (163.5,118.75) -- cycle ;
    \draw    (166.75,118.75) .. controls (207.5,118.5) and (206,93) .. (166.75,92.75) ;
    \draw    (166.75,141.75) .. controls (207.5,141.5) and (206,119) .. (166.75,118.75) ;
    
    \draw (98.5,125.5) node [anchor=north west][inner sep=0.75pt]   [align=left] {$\displaystyle p_{\beta }$};
    \draw (156.5,173) node [anchor=north west][inner sep=0.75pt]   [align=left] {$\displaystyle p_{\gamma }$};
    \draw (102,100) node [anchor=north west][inner sep=0.75pt]  [font=\small] [align=left] {$\displaystyle i_{\beta }$};
    \draw (161.5,150) node [anchor=north west][inner sep=0.75pt]  [font=\small] [align=left] {$\displaystyle i_{\gamma }$};
    \draw (147,87.5) node [anchor=north west][inner sep=0.75pt]  [font=\small] [align=left] {$\displaystyle u_{\gamma }$};
    \draw (147,114) node [anchor=north west][inner sep=0.75pt]  [font=\small] [align=left] {$\displaystyle v_{\gamma }$};
    \draw (102,35) node [anchor=north west][inner sep=0.75pt]  [font=\small] [align=left] {$\displaystyle u$};
    \draw (161.5,35) node [anchor=north west][inner sep=0.75pt]  [font=\small] [align=left] {$\displaystyle v$};

    \end{tikzpicture}

    \caption{Depiction of the strategy for two partitions of size 3 and 5 respectively.}
    \label{fig:bipstrategy_diff}

\end{figure}

\begin{figure}[htbp]

    \centering    
    
    \begin{tikzpicture}[x=0.75pt,y=0.75pt,yscale=-1,xscale=1]
    
    \draw   (170.5,150.5) .. controls (170.5,100.79) and (179.45,60.5) .. (190.5,60.5) .. controls (201.55,60.5) and (210.5,100.79) .. (210.5,150.5) .. controls (210.5,200.21) and (201.55,240.5) .. (190.5,240.5) .. controls (179.45,240.5) and (170.5,200.21) .. (170.5,150.5) -- cycle ;
    \draw   (230.5,150) .. controls (230.5,100.57) and (239.45,60.5) .. (250.5,60.5) .. controls (261.55,60.5) and (270.5,100.57) .. (270.5,150) .. controls (270.5,199.43) and (261.55,239.5) .. (250.5,239.5) .. controls (239.45,239.5) and (230.5,199.43) .. (230.5,150) -- cycle ;
    \draw  [fill={rgb, 255:red, 0; green, 0; blue, 0 }  ,fill opacity=1 ] (187.5,99.25) .. controls (187.5,97.46) and (188.96,96) .. (190.75,96) .. controls (192.54,96) and (194,97.46) .. (194,99.25) .. controls (194,101.04) and (192.54,102.5) .. (190.75,102.5) .. controls (188.96,102.5) and (187.5,101.04) .. (187.5,99.25) -- cycle ;
    \draw  [fill={rgb, 255:red, 0; green, 0; blue, 0 }  ,fill opacity=1 ] (187.5,128.25) .. controls (187.5,126.46) and (188.96,125) .. (190.75,125) .. controls (192.54,125) and (194,126.46) .. (194,128.25) .. controls (194,130.04) and (192.54,131.5) .. (190.75,131.5) .. controls (188.96,131.5) and (187.5,130.04) .. (187.5,128.25) -- cycle ;
    \draw  [fill={rgb, 255:red, 0; green, 0; blue, 0 }  ,fill opacity=1 ] (187.5,152.25) .. controls (187.5,150.46) and (188.96,149) .. (190.75,149) .. controls (192.54,149) and (194,150.46) .. (194,152.25) .. controls (194,154.04) and (192.54,155.5) .. (190.75,155.5) .. controls (188.96,155.5) and (187.5,154.04) .. (187.5,152.25) -- cycle ;
    \draw  [fill={rgb, 255:red, 0; green, 0; blue, 0 }  ,fill opacity=1 ] (187.5,190.25) .. controls (187.5,188.46) and (188.96,187) .. (190.75,187) .. controls (192.54,187) and (194,188.46) .. (194,190.25) .. controls (194,192.04) and (192.54,193.5) .. (190.75,193.5) .. controls (188.96,193.5) and (187.5,192.04) .. (187.5,190.25) -- cycle ;
    \draw  [fill={rgb, 255:red, 0; green, 0; blue, 0 }  ,fill opacity=1 ] (247.5,99.25) .. controls (247.5,97.46) and (248.96,96) .. (250.75,96) .. controls (252.54,96) and (254,97.46) .. (254,99.25) .. controls (254,101.04) and (252.54,102.5) .. (250.75,102.5) .. controls (248.96,102.5) and (247.5,101.04) .. (247.5,99.25) -- cycle ;
    \draw  [fill={rgb, 255:red, 0; green, 0; blue, 0 }  ,fill opacity=1 ] (247.5,128.25) .. controls (247.5,126.46) and (248.96,125) .. (250.75,125) .. controls (252.54,125) and (254,126.46) .. (254,128.25) .. controls (254,130.04) and (252.54,131.5) .. (250.75,131.5) .. controls (248.96,131.5) and (247.5,130.04) .. (247.5,128.25) -- cycle ;
    \draw  [fill={rgb, 255:red, 0; green, 0; blue, 0 }  ,fill opacity=1 ] (247.5,152.25) .. controls (247.5,150.46) and (248.96,149) .. (250.75,149) .. controls (252.54,149) and (254,150.46) .. (254,152.25) .. controls (254,154.04) and (252.54,155.5) .. (250.75,155.5) .. controls (248.96,155.5) and (247.5,154.04) .. (247.5,152.25) -- cycle ;
    \draw    (190.75,99.25) -- (250.75,99.25) ;
    \draw    (190.75,99.25) -- (250.75,128.25) ;
    \draw    (190.75,128.25) -- (250.75,152.25) ;
    \draw    (190.75,152.25) -- (250.75,152.25) ;
    \draw  [fill={rgb, 255:red, 0; green, 0; blue, 0 }  ,fill opacity=1 ] (187.5,171.25) .. controls (187.5,169.46) and (188.96,168) .. (190.75,168) .. controls (192.54,168) and (194,169.46) .. (194,171.25) .. controls (194,173.04) and (192.54,174.5) .. (190.75,174.5) .. controls (188.96,174.5) and (187.5,173.04) .. (187.5,171.25) -- cycle ;
    \draw  [fill={rgb, 255:red, 0; green, 0; blue, 0 }  ,fill opacity=1 ] (187.5,210.25) .. controls (187.5,208.46) and (188.96,207) .. (190.75,207) .. controls (192.54,207) and (194,208.46) .. (194,210.25) .. controls (194,212.04) and (192.54,213.5) .. (190.75,213.5) .. controls (188.96,213.5) and (187.5,212.04) .. (187.5,210.25) -- cycle ;
    \draw  [fill={rgb, 255:red, 0; green, 0; blue, 0 }  ,fill opacity=1 ] (247.5,190.25) .. controls (247.5,188.46) and (248.96,187) .. (250.75,187) .. controls (252.54,187) and (254,188.46) .. (254,190.25) .. controls (254,192.04) and (252.54,193.5) .. (250.75,193.5) .. controls (248.96,193.5) and (247.5,192.04) .. (247.5,190.25) -- cycle ;
    \draw  [fill={rgb, 255:red, 0; green, 0; blue, 0 }  ,fill opacity=1 ] (247.5,171.25) .. controls (247.5,169.46) and (248.96,168) .. (250.75,168) .. controls (252.54,168) and (254,169.46) .. (254,171.25) .. controls (254,173.04) and (252.54,174.5) .. (250.75,174.5) .. controls (248.96,174.5) and (247.5,173.04) .. (247.5,171.25) -- cycle ;
    \draw  [fill={rgb, 255:red, 0; green, 0; blue, 0 }  ,fill opacity=1 ] (247.5,210.25) .. controls (247.5,208.46) and (248.96,207) .. (250.75,207) .. controls (252.54,207) and (254,208.46) .. (254,210.25) .. controls (254,212.04) and (252.54,213.5) .. (250.75,213.5) .. controls (248.96,213.5) and (247.5,212.04) .. (247.5,210.25) -- cycle ;
    \draw    (190.75,171.25) -- (250.75,171.25) ;
    \draw    (190.75,190.25) -- (250.75,171.25) ;
    \draw    (250.75,190.25) .. controls (291.5,190) and (290,171.5) .. (250.75,171.25) ;
    
    \draw (182.5,244) node [anchor=north west][inner sep=0.75pt]   [align=left] {$\displaystyle p_{\beta }$};
    \draw (240.5,244.5) node [anchor=north west][inner sep=0.75pt]   [align=left] {$\displaystyle p_{\gamma }$};
    \draw (183,82) node [anchor=north west][inner sep=0.75pt]  [font=\small] [align=left] {$\displaystyle v_{2}$};
    \draw (244,82) node [anchor=north west][inner sep=0.75pt]  [font=\small] [align=left] {$\displaystyle v_{1}$};
    \draw (243,113) node [anchor=north west][inner sep=0.75pt]  [font=\small] [align=left] {$\displaystyle v_{3}$};
    \draw (244,137) node [anchor=north west][inner sep=0.75pt]  [font=\small] [align=left] {$\displaystyle v_{5}$};
    \draw (184,113) node [anchor=north west][inner sep=0.75pt]  [font=\small] [align=left] {$\displaystyle v_{4}$};
    \draw (184.5,137) node [anchor=north west][inner sep=0.75pt]  [font=\small] [align=left] {$\displaystyle v_{6}$};
    \draw (185,215) node [anchor=north west][inner sep=0.75pt]  [font=\small] [align=left] {$\displaystyle i_{\beta }$};
    \draw (246,215) node [anchor=north west][inner sep=0.75pt]  [font=\small] [align=left] {$\displaystyle i_{\gamma }$};
    \draw (253,160) node [anchor=north west][inner sep=0.75pt]  [font=\small] [align=left] {$\displaystyle w$};
    \draw (253,193) node [anchor=north west][inner sep=0.75pt]  [font=\small] [align=left] {$\displaystyle t$};
    \draw (178,160) node [anchor=north west][inner sep=0.75pt]  [font=\small] [align=left] {$\displaystyle u$};
    \draw (178,193) node [anchor=north west][inner sep=0.75pt]  [font=\small] [align=left] {$\displaystyle v$};
    
    \end{tikzpicture}

    \caption{Depiction of the strategy for two partitions of size 6.}
    \label{fig:bipstrategy_same}

\end{figure}

\paragraph*{Measurement Counts}
It can be verified that the strategies in  Definitions~\ref{def:bip_different_sizes_strategy} and \ref{def:bip_same_size_strategy} give the stated measurement counts in Theorem~\ref{thm:bipartite}.

\paragraph*{Correctness}
The following is the proof of correctness for the strategy outlined in  Definition~\ref{def:bip_different_sizes_strategy} for bipartite graphs with different sizes. The proof of the correctness of the strategy outlined in  Definition~\ref{def:bip_same_size_strategy} is provided in Case 9 of the proof of correctness for complete $k$-partite graphs (Section~\ref{UB:thm:k-partite}). 

Recall we label the partitions $p_\beta$ and $p_\gamma$ with $| p_\beta|  < | p_\gamma| $. Consider two edges $(a,b)$ and $(a',b')$ with $a, a' \in p_\beta$ and $b,b' \in p_\gamma$. Since $e \neq e'$, either $a \neq a'$ or $b\neq b'$. Recall that $p_\beta$ and $p_\gamma$ each contain at most one designated node not involved in any measurement. Further, we recognize that measurements whose endpoints are both contained in the same partition are comprised of vertices contained in $p_\gamma$. Now we consider two cases (see Figures~\ref{fig:bipstrategy_diff} and \ref{fig:bipstrategy_same}).

\noindent\underline{\textit{Case 1: $a \neq a'$}}. Since $a \neq a'$, there must exist a measurement $(r,s)$ in the strategy from  Definition~\ref{def:bip_different_sizes_strategy} that involves exactly one of $a$ or $a'$; without loss of generality, let $a$ be this vertex with $a = r$. Suppose first that $b = b'$. If $b=s$, when $e$ is altered, the effective resistance value comes from column \textbf{I} in Table \ref{table:k_part-diff-part}, and for $e'$, the value comes from column \textbf{III} as $a' \in p_r, a' \neq r$ and $b' = s$. If $b \neq s$, then  the effective resistance value comes from column \textbf{II} when $e$ is altered and, since $a' \in p_r, a' \neq r$ and $b' \in p_s, b' \neq  s$, for $e'$, the effective resistance value comes from column \textbf{IV}. Next, suppose $b \neq b'$. If $b=s$, when $e$ is altered, the effective resistance value comes from column \textbf{I} in Table \ref{table:k_part-diff-part}, and for $e'$, the value comes from column \textbf{IV} as $a' \in p_r, a' \neq r$ and $b' \neq s$. If $b' = s$, then the effective resistance value when $e$ is altered comes from column \textbf{II} in Table \ref{table:k_part-diff-part} since $a = r$ and $b \in p_s, b\neq s$, and when $e'$ is altered, the value comes from column \textbf{III} since $a' \in p_r, a' \neq r$ and $b' = s$. Note that $| p_r|  \neq | p_s| $ because the bipartite graph has partitions of different sizes. Thus, the value in column \textbf{II} is different from the value in column \textbf{III}. Finally, if $s$ is not equal to $b$ or $b'$, then the effective resistance value when $e$ is altered comes from \textbf{II} and, since $a' \in p_r, a' \neq r$ and $b' \in p_s, b' \neq  s$, for $e'$, the effective resistance value comes from column \textbf{IV}.

\noindent\underline{\textit{Case 2: $a' = a$}}. This implies $b \neq b'$. Observe that there must exist a measurement $(r,s)$ in the strategy from  Definition~\ref{def:bip_different_sizes_strategy} that involves $b$ or $b'$ but not both; without loss of generality, let $b$ be this vertex with $b = s$. If the measurement endpoints are in different partitions, i.e.,~$s \notin p_r$, when $e$ is altered, the effective resistance value comes from column \textbf{I} in Table \ref{table:k_part-diff-part} if $a=r$ and $b=s$, and for $e'$, the value comes from column \textbf{II} as $a' = r$ and $b' \in p_s, b' \neq s$. Another possibility is  that $a \in p_r, a\neq r$ and $b = s$, meaning that when $e$ is altered the effective resistance value comes from \textbf{III}, and for $e'$, the value comes from \textbf{IV} as $a' \in p_r, a' \neq r$ and $b' \in p_s, b' \neq s$. Now, suppose the measurement endpoints are in the same partition, i.e.~$s \in p_r$. When $e$ is altered, the effective resistance  value comes from column \textbf{X} in Table \ref{table:k_part-same-part} as $a \notin p_r$ and $ b = s$ and, since $a' \notin p_r$ and $b' \in p_s, b' \neq s$, for $e'$ the value comes from column \textbf{XI} (the endpoint labels on edges $(a,b)$ and $(a',b')$ should be switched to match with the convention in Table \ref{table:k_part-same-part}).


\subsubsection{Lower Bound in Theorem~\ref{thm:bipartite}}

Let $\mathcal{M}$ be any set of measurements that solves the faulty edge detection problem for the complete bipartite graph and consider the measurement graph $G' = (V,\mathcal{M})$ from  Definition~\ref{def:measurement-graph}. 

We first discuss the lower bound achieved when $| p_\beta|  < | p_\gamma| $. We observe that we can have at most one isolated vertex in each partition in $G'$. Otherwise, let $a,a'$ be two isolated vertices in the same partition $p_a$ and $b \in p_b$ be any vertex in $G$ with $p_a \neq p_b$. Since $a, a'$ are in the same partition and they are not the endpoint of any measurement, for both edges $(a,b)$ and $(a',b)$, the effective resistance value will come from the same column in Tables~\ref{table:k_part-diff-part} and \ref{table:k_part-same-part}
when they are altered. Similarly, a component of size two in $G'$ cannot be formed using two vertices in the same partition since two edges with these as endpoints and any other vertex in the other partition as the other endpoint cannot be distinguished by any measurement. Let $q_2$ be the number of size two components in $\mathcal{M}$. Since there can only be components of size two with endpoints in different partitions, the number of components of size two is at most $| p_\beta| $. The remaining components of size $s \geq 3$ must have at least $2s/3$ edges. Therefore, the number of measurements in $\mathcal{M}$ is at least
\begin{equation*}
\begin{array}
{rcl}
q_2 + \frac{2(n-2-2q_2)}{3} & = &\frac{2n}{3} - \frac{q_2}{3} - \frac43 \\
& \geq &\frac{2n}{3} - \frac{| p_\beta| }{3} - \frac43 \\
& = & \frac{2| p_\gamma| }{3} + \frac{| p_\beta| }{3} - \frac43
\end{array}
\end{equation*}

Since the number of measurements is an integer, this gives at least $\left\lfloor\frac{2| p_\gamma| }{3} + \frac{| p_\beta| }{3}\right\rfloor - 1$ measurements. This argument can be sharpened when $| p_\gamma|  - | p_\beta|   \equiv 1\mod{3}$ or $| p_\gamma|  - | p_\beta|   \equiv 2\mod{3}$. Note that in this case, if we have an isolated vertex in both partitions, then $q_2 \leq | p_\beta|  - 1$. If $q_2 \leq | p_\beta|  - 2$, then the right hand side of the second inequality above becomes $\frac{2n}{3} - \frac{| p_\beta|  - 2}{3} - \frac43 = \frac{2}{3}(| p_\gamma|  - | p_\beta| ) + | p_\beta|  -\frac23$. Since $| p_\gamma|  - | p_\beta|   \equiv 1\mod{3}$ or $| p_\gamma|  - | p_\beta|   \equiv 2\mod{3}$, $\frac{2}{3}(| p_\gamma|  - | p_\beta| ) + | p_\beta|  -\frac23 \geq \left\lfloor\frac{2| p_\gamma| }{3} + \frac{| p_\beta| }{3}\right\rfloor$. If $q_2 = | p_\beta|  - 1$, then observe that we have components of size at least three on the remaining $| p_\gamma|  - | p_\beta| $ vertices in the $p_\gamma$ component. In other words, we have at least $\left\lceil \frac{2(| p_\gamma|  - | p_\beta| )}{3} \right\rceil$ edges, and therefore, in total, $\left\lceil \frac{2(| p_\gamma|  - | p_\beta| )}{3} \right\rceil + | p_\beta|  - 1$ edges. Since $| p_\gamma|  - | p_\beta|   \not\equiv 0\mod{3}$, $\left\lceil \frac{2(| p_\gamma|  - | p_\beta| )}{3} \right\rceil + | p_\beta|  - 1 \geq \left\lfloor \frac{2(| p_\gamma|  - | p_\beta| )}{3} \right\rfloor + 1 + | p_\beta|  - 1 = \left\lfloor\frac{2| p_\gamma| }{3} + \frac{| p_\beta| }{3}\right\rfloor$.

We now consider the case where $| p_\beta|  = | p_\gamma| $. As before, we can have at most 2 isolated vertices in $G'$, and components of size two must have endpoints in both partitions. Moreover, we can have at most one component of size two. Indeed, if $r,s$ and $r',s'$
are two such measurements with $r, r'$ in the same partition, then no measurement can distinguish the edges $(r,s')$ and $(r',s)$ since the sizes of the partitions are the same. Moreover, if one has 2 isolated vertices in $G'$, then one cannot have a component of size two in $G'$: Let $u,u'$ be the isolated vertices and let $(v,v')$ be the single edge, with $u,v$ in the same partition and $u',v'$ in the same partition. Then no measurement in $G'$ can distinguish between the edges $(u,v')$ and $(u',v)$. Indeed, when looking at $(v,v')$ as a measurement, the effective resistance values of the edges $(u,v')$ and $(u',v)$ will be the same as they come from columns \textbf{II} and \textbf{III} in Table \ref{table:k_part-diff-part} with $| p_r|  = | p_s| $. When looking at any other measurement in $G'$, the values will come from column \textbf{IV} in Table \ref{table:k_part-diff-part} when the measurement endpoints are in different partitions and from columns \textbf{XI} in Table \ref{table:k_part-same-part} when the measurement endpoints are in the same partition. Therefore, the number of measurements is at least $\min\left\{\frac{2(n-2)}{3}, 1+\frac{2(n-3)}{3}\right\} \geq \frac{2n}{3} - \frac43$.
Since the number of measurements must be an integer, this means we must have at least $\left\lceil{\frac{2n}{3} - \frac43}\right\rceil$ measurements. We now observe that when $| p_\beta|  \equiv 0\mod{3}$ or $| p_\beta|  \equiv 1\mod{3}$, i.e.,~$n \equiv 0\mod{3}$ or $n \equiv 2\mod{3}$, $\ceil{\frac{2n}{3} - \frac43} \geq \floor{\frac{2n}{3}} - 1$, and when $| p_\beta|  \equiv 2\mod{3}$, i.e.,~$n \equiv 1\mod{3}$, $\ceil{\frac{2n}{3} - \frac43} = \ceil{\floor{\frac{2n}{3}} + \frac23 - \frac43} = \floor{\frac{2n}{3}} + \ceil{\frac23 - \frac43} = \floor{\frac{2n}{3}}$.

\subsection{Complete tripartite graphs}\label{sec:tripartite}

Using the same notational conventions and similar techniques used in the proof of Theorem~\ref{thm:bipartite} in  Section~\ref{sec:bipartite_graphs}, we present the proof for complete tripartite graphs.

\subsubsection{Upper Bound in Theorem~\ref{thm:tripartite}
}\label{sec:tripartite_upper}

\paragraph*{Strategy}

For the case that $| p_\alpha|  = | p_\beta|  = | p_\gamma| $, the strategy is to maintain three designated vertices, one in each partition, and to place tripartite butterfly wing measurements (Definition~\ref{def:tri_bw}) on the remaining vertices (see Figure~\ref{fig:tripstrategy_same}).

\begin{center}

    \begin{tikzpicture}[x=0.75pt,y=0.75pt,yscale=-1,xscale=1]
    
    \draw   (30,95.03) .. controls (30,56.08) and (38.95,24.5) .. (50,24.5) .. controls (61.05,24.5) and (70,56.08) .. (70,95.03) .. controls (70,133.98) and (61.05,165.56) .. (50,165.56) .. controls (38.95,165.56) and (30,133.98) .. (30,95.03) -- cycle ;
    \draw   (90.5,95.87) .. controls (90.5,56.92) and (99.45,25.34) .. (110.5,25.34) .. controls (121.55,25.34) and (130.5,56.92) .. (130.5,95.87) .. controls (130.5,134.82) and (121.55,166.4) .. (110.5,166.4) .. controls (99.45,166.4) and (90.5,134.82) .. (90.5,95.87) -- cycle ;
    \draw   (150.5,95.87) .. controls (150.5,56.92) and (159.45,25.34) .. (170.5,25.34) .. controls (181.55,25.34) and (190.5,56.92) .. (190.5,95.87) .. controls (190.5,134.82) and (181.55,166.4) .. (170.5,166.4) .. controls (159.45,166.4) and (150.5,134.82) .. (150.5,95.87) -- cycle ;
    \draw  [fill={rgb, 255:red, 0; green, 0; blue, 0 }  ,fill opacity=1 ] (47.5,51.25) .. controls (47.5,49.46) and (48.96,48) .. (50.75,48) .. controls (52.54,48) and (54,49.46) .. (54,51.25) .. controls (54,53.04) and (52.54,54.5) .. (50.75,54.5) .. controls (48.96,54.5) and (47.5,53.04) .. (47.5,51.25) -- cycle ;
    \draw  [fill={rgb, 255:red, 0; green, 0; blue, 0 }  ,fill opacity=1 ] (47.5,72.25) .. controls (47.5,70.46) and (48.96,69) .. (50.75,69) .. controls (52.54,69) and (54,70.46) .. (54,72.25) .. controls (54,74.04) and (52.54,75.5) .. (50.75,75.5) .. controls (48.96,75.5) and (47.5,74.04) .. (47.5,72.25) -- cycle ;
    \draw    (108.75,52.25) -- (169.75,52.25) ;
    \draw    (51.75,51.25) .. controls (92.5,71) and (130.5,72) .. (170.75,52.25) ;
    \draw  [fill={rgb, 255:red, 0; green, 0; blue, 0 }  ,fill opacity=1 ] (47.5,95.25) .. controls (47.5,93.46) and (48.96,92) .. (50.75,92) .. controls (52.54,92) and (54,93.46) .. (54,95.25) .. controls (54,97.04) and (52.54,98.5) .. (50.75,98.5) .. controls (48.96,98.5) and (47.5,97.04) .. (47.5,95.25) -- cycle ;
    \draw  [fill={rgb, 255:red, 0; green, 0; blue, 0 }  ,fill opacity=1 ] (47.5,120.25) .. controls (47.5,118.46) and (48.96,117) .. (50.75,117) .. controls (52.54,117) and (54,118.46) .. (54,120.25) .. controls (54,122.04) and (52.54,123.5) .. (50.75,123.5) .. controls (48.96,123.5) and (47.5,122.04) .. (47.5,120.25) -- cycle ;
    \draw  [fill={rgb, 255:red, 0; green, 0; blue, 0 }  ,fill opacity=1 ] (107.5,52.25) .. controls (107.5,50.46) and (108.96,49) .. (110.75,49) .. controls (112.54,49) and (114,50.46) .. (114,52.25) .. controls (114,54.04) and (112.54,55.5) .. (110.75,55.5) .. controls (108.96,55.5) and (107.5,54.04) .. (107.5,52.25) -- cycle ;
    \draw  [fill={rgb, 255:red, 0; green, 0; blue, 0 }  ,fill opacity=1 ] (107.5,73.25) .. controls (107.5,71.46) and (108.96,70) .. (110.75,70) .. controls (112.54,70) and (114,71.46) .. (114,73.25) .. controls (114,75.04) and (112.54,76.5) .. (110.75,76.5) .. controls (108.96,76.5) and (107.5,75.04) .. (107.5,73.25) -- cycle ;
    \draw  [fill={rgb, 255:red, 0; green, 0; blue, 0 }  ,fill opacity=1 ] (107.5,96.25) .. controls (107.5,94.46) and (108.96,93) .. (110.75,93) .. controls (112.54,93) and (114,94.46) .. (114,96.25) .. controls (114,98.04) and (112.54,99.5) .. (110.75,99.5) .. controls (108.96,99.5) and (107.5,98.04) .. (107.5,96.25) -- cycle ;
    \draw  [fill={rgb, 255:red, 0; green, 0; blue, 0 }  ,fill opacity=1 ] (107.5,121.25) .. controls (107.5,119.46) and (108.96,118) .. (110.75,118) .. controls (112.54,118) and (114,119.46) .. (114,121.25) .. controls (114,123.04) and (112.54,124.5) .. (110.75,124.5) .. controls (108.96,124.5) and (107.5,123.04) .. (107.5,121.25) -- cycle ;
    \draw  [fill={rgb, 255:red, 0; green, 0; blue, 0 }  ,fill opacity=1 ] (167.5,52.25) .. controls (167.5,50.46) and (168.96,49) .. (170.75,49) .. controls (172.54,49) and (174,50.46) .. (174,52.25) .. controls (174,54.04) and (172.54,55.5) .. (170.75,55.5) .. controls (168.96,55.5) and (167.5,54.04) .. (167.5,52.25) -- cycle ;
    \draw  [fill={rgb, 255:red, 0; green, 0; blue, 0 }  ,fill opacity=1 ] (167.5,73.25) .. controls (167.5,71.46) and (168.96,70) .. (170.75,70) .. controls (172.54,70) and (174,71.46) .. (174,73.25) .. controls (174,75.04) and (172.54,76.5) .. (170.75,76.5) .. controls (168.96,76.5) and (167.5,75.04) .. (167.5,73.25) -- cycle ;
    \draw  [fill={rgb, 255:red, 0; green, 0; blue, 0 }  ,fill opacity=1 ] (167.5,96.25) .. controls (167.5,94.46) and (168.96,93) .. (170.75,93) .. controls (172.54,93) and (174,94.46) .. (174,96.25) .. controls (174,98.04) and (172.54,99.5) .. (170.75,99.5) .. controls (168.96,99.5) and (167.5,98.04) .. (167.5,96.25) -- cycle ;
    \draw  [fill={rgb, 255:red, 0; green, 0; blue, 0 }  ,fill opacity=1 ] (167.5,121.25) .. controls (167.5,119.46) and (168.96,118) .. (170.75,118) .. controls (172.54,118) and (174,119.46) .. (174,121.25) .. controls (174,123.04) and (172.54,124.5) .. (170.75,124.5) .. controls (168.96,124.5) and (167.5,123.04) .. (167.5,121.25) -- cycle ;
    \draw    (50.75,72.25) .. controls (91.5,92) and (129.5,93) .. (169.75,73.25) ;
    \draw    (50.75,95.25) .. controls (91.5,115) and (129.5,116) .. (169.75,96.25) ;
    \draw    (110.75,73.25) -- (171.75,73.25) ;
    \draw    (110.75,96.25) -- (171.75,96.25) ;
    
    \draw (44,170) node [anchor=north west][inner sep=0.75pt]   [align=left] {$\displaystyle p_{\alpha }$};
    \draw (105,170) node [anchor=north west][inner sep=0.75pt]   [align=left] {$\displaystyle p_{\beta }$};
    \draw (164,170) node [anchor=north west][inner sep=0.75pt]   [align=left] {$\displaystyle p_{\gamma }$};
    \draw (44,130) node [anchor=north west][inner sep=0.75pt]   [align=left] {$\displaystyle i_{\alpha }$};
    \draw (105,130) node [anchor=north west][inner sep=0.75pt]   [align=left] {$\displaystyle i_{\beta }$};
    \draw (164.5,130) node [anchor=north west][inner sep=0.75pt]   [align=left] {$\displaystyle i_{\gamma }$};

    \end{tikzpicture}

    \captionof{figure}{Depiction of the strategy for three partitions of size 4.}
    \label{fig:tripstrategy_same}

\end{center}

For the case that $| p_\alpha|  < | p_\beta|  < | p_\gamma| $, the strategy is to maintain three designated vertices, $i_\alpha, i_\beta, i_\gamma$, one in each
partition. If $| p_\gamma|  - 1 \geq (| p_\alpha|  - 1) + (| p_\beta|  - 1)$, place $| p_\alpha|  - 1$ edges from $p_\alpha \setminus \{i_\alpha\}$ to $p_\gamma \setminus \{i_\gamma\}$ and $| p_\beta|  - 1$ edges from
$p_\beta \setminus \{i_\beta\}$ to $p_\gamma \setminus \{i_\gamma\}$. In any vertices remaining in $p_\gamma \setminus \{i_\gamma\}$, place as many disjoint partition butterfly
wing measurements (Definition~\ref{def:part_bw}) as possible, on the vertices that are left over. If there is one vertex
$v$ left over, which is not the designated vertex in this partition, and it is not involved in any
measurement so far, place a measurement from $v$ to any previously used vertex in that partition. If there
are two vertices left over, place a partition butterfly measurement involving these two and the designated vertex
in this partition. If  $| p_\gamma|  - 1 < (| p_\alpha|  - 1) + (| p_\beta|  - 1)$, examine the number of non-isolated vertices $L = (| p_\gamma|  - 1) + (| p_\alpha|  - 1) + (| p_\beta|  - 1)$. If $L$ is
even, there exists a maximum matching of size $\frac{L}{2}$ (see~\cite{sitton1996maximum} for proof and algorithm
for obtaining the matching). Place matching measurements according to this specified matching. If $L$ is
odd, there exists a maximum matching of size $\floor{\frac{L}{2}}$ that does not use one vertex in $p_\alpha$; denote this
vertex as $v_L$. Place matching measurements according to this specified matching. Select any vertex from
$p_\alpha \setminus \{i_\alpha\}$ and make a measurement from this vertex to $v_L$ (see Figure~\ref{fig:tripstrategy_diffall}).

\begin{center}      
    
    \begin{tikzpicture}[x=0.75pt,y=0.75pt,yscale=-1,xscale=1]
    
    \draw   (29.5,66.65) .. controls (29.5,36.47) and (38.45,12) .. (49.5,12) .. controls (60.55,12) and (69.5,36.47) .. (69.5,66.65) .. controls (69.5,96.83) and (60.55,121.3) .. (49.5,121.3) .. controls (38.45,121.3) and (29.5,96.83) .. (29.5,66.65) -- cycle ;
    \draw   (90,112.47) .. controls (90,57.44) and (98.95,12.84) .. (110,12.84) .. controls (121.05,12.84) and (130,57.44) .. (130,112.47) .. controls (130,167.49) and (121.05,212.1) .. (110,212.1) .. controls (98.95,212.1) and (90,167.49) .. (90,112.47) -- cycle ;
    \draw   (150,75.82) .. controls (150,41.03) and (158.95,12.84) .. (170,12.84) .. controls (181.05,12.84) and (190,41.03) .. (190,75.82) .. controls (190,110.6) and (181.05,138.8) .. (170,138.8) .. controls (158.95,138.8) and (150,110.6) .. (150,75.82) -- cycle ;
    \draw  [fill={rgb, 255:red, 0; green, 0; blue, 0 }  ,fill opacity=1 ] (47,39.75) .. controls (47,37.96) and (48.46,36.5) .. (50.25,36.5) .. controls (52.04,36.5) and (53.5,37.96) .. (53.5,39.75) .. controls (53.5,41.54) and (52.04,43) .. (50.25,43) .. controls (48.46,43) and (47,41.54) .. (47,39.75) -- cycle ;
    \draw  [fill={rgb, 255:red, 0; green, 0; blue, 0 }  ,fill opacity=1 ] (47,60.75) .. controls (47,58.96) and (48.46,57.5) .. (50.25,57.5) .. controls (52.04,57.5) and (53.5,58.96) .. (53.5,60.75) .. controls (53.5,62.54) and (52.04,64) .. (50.25,64) .. controls (48.46,64) and (47,62.54) .. (47,60.75) -- cycle ;
    \draw  [fill={rgb, 255:red, 0; green, 0; blue, 0 }  ,fill opacity=1 ] (47,80.75) .. controls (47,78.96) and (48.46,77.5) .. (50.25,77.5) .. controls (52.04,77.5) and (53.5,78.96) .. (53.5,80.75) .. controls (53.5,82.54) and (52.04,84) .. (50.25,84) .. controls (48.46,84) and (47,82.54) .. (47,80.75) -- cycle ;
    \draw  [fill={rgb, 255:red, 0; green, 0; blue, 0 }  ,fill opacity=1 ] (107,39.75) .. controls (107,37.96) and (108.46,36.5) .. (110.25,36.5) .. controls (112.04,36.5) and (113.5,37.96) .. (113.5,39.75) .. controls (113.5,41.54) and (112.04,43) .. (110.25,43) .. controls (108.46,43) and (107,41.54) .. (107,39.75) -- cycle ;
    \draw  [fill={rgb, 255:red, 0; green, 0; blue, 0 }  ,fill opacity=1 ] (107,60.75) .. controls (107,58.96) and (108.46,57.5) .. (110.25,57.5) .. controls (112.04,57.5) and (113.5,58.96) .. (113.5,60.75) .. controls (113.5,62.54) and (112.04,64) .. (110.25,64) .. controls (108.46,64) and (107,62.54) .. (107,60.75) -- cycle ;
    \draw  [fill={rgb, 255:red, 0; green, 0; blue, 0 }  ,fill opacity=1 ] (107,80.75) .. controls (107,78.96) and (108.46,77.5) .. (110.25,77.5) .. controls (112.04,77.5) and (113.5,78.96) .. (113.5,80.75) .. controls (113.5,82.54) and (112.04,84) .. (110.25,84) .. controls (108.46,84) and (107,82.54) .. (107,80.75) -- cycle ;
    \draw  [fill={rgb, 255:red, 0; green, 0; blue, 0 }  ,fill opacity=1 ] (107,100.75) .. controls (107,98.96) and (108.46,97.5) .. (110.25,97.5) .. controls (112.04,97.5) and (113.5,98.96) .. (113.5,100.75) .. controls (113.5,102.54) and (112.04,104) .. (110.25,104) .. controls (108.46,104) and (107,102.54) .. (107,100.75) -- cycle ;
    \draw  [fill={rgb, 255:red, 0; green, 0; blue, 0 }  ,fill opacity=1 ] (167,39.75) .. controls (167,37.96) and (168.46,36.5) .. (170.25,36.5) .. controls (172.04,36.5) and (173.5,37.96) .. (173.5,39.75) .. controls (173.5,41.54) and (172.04,43) .. (170.25,43) .. controls (168.46,43) and (167,41.54) .. (167,39.75) -- cycle ;
    \draw  [fill={rgb, 255:red, 0; green, 0; blue, 0 }  ,fill opacity=1 ] (167,60.75) .. controls (167,58.96) and (168.46,57.5) .. (170.25,57.5) .. controls (172.04,57.5) and (173.5,58.96) .. (173.5,60.75) .. controls (173.5,62.54) and (172.04,64) .. (170.25,64) .. controls (168.46,64) and (167,62.54) .. (167,60.75) -- cycle ;
    \draw  [fill={rgb, 255:red, 0; green, 0; blue, 0 }  ,fill opacity=1 ] (167,80.75) .. controls (167,78.96) and (168.46,77.5) .. (170.25,77.5) .. controls (172.04,77.5) and (173.5,78.96) .. (173.5,80.75) .. controls (173.5,82.54) and (172.04,84) .. (170.25,84) .. controls (168.46,84) and (167,82.54) .. (167,80.75) -- cycle ;
    \draw  [fill={rgb, 255:red, 0; green, 0; blue, 0 }  ,fill opacity=1 ] (167,100.75) .. controls (167,98.96) and (168.46,97.5) .. (170.25,97.5) .. controls (172.04,97.5) and (173.5,98.96) .. (173.5,100.75) .. controls (173.5,102.54) and (172.04,104) .. (170.25,104) .. controls (168.46,104) and (167,102.54) .. (167,100.75) -- cycle ;
    \draw  [fill={rgb, 255:red, 0; green, 0; blue, 0 }  ,fill opacity=1 ] (107,120.25) .. controls (107,118.46) and (108.46,117) .. (110.25,117) .. controls (112.04,117) and (113.5,118.46) .. (113.5,120.25) .. controls (113.5,122.04) and (112.04,123.5) .. (110.25,123.5) .. controls (108.46,123.5) and (107,122.04) .. (107,120.25) -- cycle ;
    \draw  [fill={rgb, 255:red, 0; green, 0; blue, 0 }  ,fill opacity=1 ] (107,141.25) .. controls (107,139.46) and (108.46,138) .. (110.25,138) .. controls (112.04,138) and (113.5,139.46) .. (113.5,141.25) .. controls (113.5,143.04) and (112.04,144.5) .. (110.25,144.5) .. controls (108.46,144.5) and (107,143.04) .. (107,141.25) -- cycle ;
    \draw  [fill={rgb, 255:red, 0; green, 0; blue, 0 }  ,fill opacity=1 ] (107,161.25) .. controls (107,159.46) and (108.46,158) .. (110.25,158) .. controls (112.04,158) and (113.5,159.46) .. (113.5,161.25) .. controls (113.5,163.04) and (112.04,164.5) .. (110.25,164.5) .. controls (108.46,164.5) and (107,163.04) .. (107,161.25) -- cycle ;
    \draw  [fill={rgb, 255:red, 0; green, 0; blue, 0 }  ,fill opacity=1 ] (107,181.25) .. controls (107,179.46) and (108.46,178) .. (110.25,178) .. controls (112.04,178) and (113.5,179.46) .. (113.5,181.25) .. controls (113.5,183.04) and (112.04,184.5) .. (110.25,184.5) .. controls (108.46,184.5) and (107,183.04) .. (107,181.25) -- cycle ;
    \draw    (50.25,39.75) -- (110.25,39.75) ;
    \draw    (50.25,60.75) -- (110.25,60.75) ;
    \draw    (110.25,80.75) -- (170.25,39.75) ;
    \draw    (110.25,100.75) -- (170.25,60.75) ;
    \draw    (110.25,120.25) -- (170.25,80.75) ;
    \draw    (110.25,141.25) .. controls (149.8,139.6) and (150.8,160.6) .. (110.25,161.25) ;
    \draw    (110.25,161.25) .. controls (149.8,159.6) and (150.8,180.6) .. (110.25,181.25) ;
    \draw   (239.5,82.05) .. controls (239.5,43.91) and (248.45,13) .. (259.5,13) .. controls (270.55,13) and (279.5,43.91) .. (279.5,82.05) .. controls (279.5,120.19) and (270.55,151.1) .. (259.5,151.1) .. controls (248.45,151.1) and (239.5,120.19) .. (239.5,82.05) -- cycle ;
    \draw   (300,113.47) .. controls (300,58.44) and (308.95,13.84) .. (320,13.84) .. controls (331.05,13.84) and (340,58.44) .. (340,113.47) .. controls (340,168.49) and (331.05,213.1) .. (320,213.1) .. controls (308.95,213.1) and (300,168.49) .. (300,113.47) -- cycle ;
    \draw   (360,92.72) .. controls (360,49.15) and (368.95,13.84) .. (380,13.84) .. controls (391.05,13.84) and (400,49.15) .. (400,92.72) .. controls (400,136.28) and (391.05,171.6) .. (380,171.6) .. controls (368.95,171.6) and (360,136.28) .. (360,92.72) -- cycle ;
    \draw  [fill={rgb, 255:red, 0; green, 0; blue, 0 }  ,fill opacity=1 ] (257,40.75) .. controls (257,38.96) and (258.46,37.5) .. (260.25,37.5) .. controls (262.04,37.5) and (263.5,38.96) .. (263.5,40.75) .. controls (263.5,42.54) and (262.04,44) .. (260.25,44) .. controls (258.46,44) and (257,42.54) .. (257,40.75) -- cycle ;
    \draw  [fill={rgb, 255:red, 0; green, 0; blue, 0 }  ,fill opacity=1 ] (257,60.75) .. controls (257,58.96) and (258.46,57.5) .. (260.25,57.5) .. controls (262.04,57.5) and (263.5,58.96) .. (263.5,60.75) .. controls (263.5,62.54) and (262.04,64) .. (260.25,64) .. controls (258.46,64) and (257,62.54) .. (257,60.75) -- cycle ;
    \draw  [fill={rgb, 255:red, 0; green, 0; blue, 0 }  ,fill opacity=1 ] (317,40.75) .. controls (317,38.96) and (318.46,37.5) .. (320.25,37.5) .. controls (322.04,37.5) and (323.5,38.96) .. (323.5,40.75) .. controls (323.5,42.54) and (322.04,44) .. (320.25,44) .. controls (318.46,44) and (317,42.54) .. (317,40.75) -- cycle ;
    \draw  [fill={rgb, 255:red, 0; green, 0; blue, 0 }  ,fill opacity=1 ] (317,60.75) .. controls (317,58.96) and (318.46,57.5) .. (320.25,57.5) .. controls (322.04,57.5) and (323.5,58.96) .. (323.5,60.75) .. controls (323.5,62.54) and (322.04,64) .. (320.25,64) .. controls (318.46,64) and (317,62.54) .. (317,60.75) -- cycle ;
    \draw  [fill={rgb, 255:red, 0; green, 0; blue, 0 }  ,fill opacity=1 ] (317,80.75) .. controls (317,78.96) and (318.46,77.5) .. (320.25,77.5) .. controls (322.04,77.5) and (323.5,78.96) .. (323.5,80.75) .. controls (323.5,82.54) and (322.04,84) .. (320.25,84) .. controls (318.46,84) and (317,82.54) .. (317,80.75) -- cycle ;
    \draw  [fill={rgb, 255:red, 0; green, 0; blue, 0 }  ,fill opacity=1 ] (317,101.75) .. controls (317,99.96) and (318.46,98.5) .. (320.25,98.5) .. controls (322.04,98.5) and (323.5,99.96) .. (323.5,101.75) .. controls (323.5,103.54) and (322.04,105) .. (320.25,105) .. controls (318.46,105) and (317,103.54) .. (317,101.75) -- cycle ;
    \draw  [fill={rgb, 255:red, 0; green, 0; blue, 0 }  ,fill opacity=1 ] (377,40.75) .. controls (377,38.96) and (378.46,37.5) .. (380.25,37.5) .. controls (382.04,37.5) and (383.5,38.96) .. (383.5,40.75) .. controls (383.5,42.54) and (382.04,44) .. (380.25,44) .. controls (378.46,44) and (377,42.54) .. (377,40.75) -- cycle ;
    \draw  [fill={rgb, 255:red, 0; green, 0; blue, 0 }  ,fill opacity=1 ] (377,61.75) .. controls (377,59.96) and (378.46,58.5) .. (380.25,58.5) .. controls (382.04,58.5) and (383.5,59.96) .. (383.5,61.75) .. controls (383.5,63.54) and (382.04,65) .. (380.25,65) .. controls (378.46,65) and (377,63.54) .. (377,61.75) -- cycle ;
    \draw  [fill={rgb, 255:red, 0; green, 0; blue, 0 }  ,fill opacity=1 ] (377,81.75) .. controls (377,79.96) and (378.46,78.5) .. (380.25,78.5) .. controls (382.04,78.5) and (383.5,79.96) .. (383.5,81.75) .. controls (383.5,83.54) and (382.04,85) .. (380.25,85) .. controls (378.46,85) and (377,83.54) .. (377,81.75) -- cycle ;
    \draw  [fill={rgb, 255:red, 0; green, 0; blue, 0 }  ,fill opacity=1 ] (377,101.75) .. controls (377,99.96) and (378.46,98.5) .. (380.25,98.5) .. controls (382.04,98.5) and (383.5,99.96) .. (383.5,101.75) .. controls (383.5,103.54) and (382.04,105) .. (380.25,105) .. controls (378.46,105) and (377,103.54) .. (377,101.75) -- cycle ;
    \draw  [fill={rgb, 255:red, 0; green, 0; blue, 0 }  ,fill opacity=1 ] (317,121.25) .. controls (317,119.46) and (318.46,118) .. (320.25,118) .. controls (322.04,118) and (323.5,119.46) .. (323.5,121.25) .. controls (323.5,123.04) and (322.04,124.5) .. (320.25,124.5) .. controls (318.46,124.5) and (317,123.04) .. (317,121.25) -- cycle ;
    \draw  [fill={rgb, 255:red, 0; green, 0; blue, 0 }  ,fill opacity=1 ] (317,142.25) .. controls (317,140.46) and (318.46,139) .. (320.25,139) .. controls (322.04,139) and (323.5,140.46) .. (323.5,142.25) .. controls (323.5,144.04) and (322.04,145.5) .. (320.25,145.5) .. controls (318.46,145.5) and (317,144.04) .. (317,142.25) -- cycle ;
    \draw  [fill={rgb, 255:red, 0; green, 0; blue, 0 }  ,fill opacity=1 ] (317,162.25) .. controls (317,160.46) and (318.46,159) .. (320.25,159) .. controls (322.04,159) and (323.5,160.46) .. (323.5,162.25) .. controls (323.5,164.04) and (322.04,165.5) .. (320.25,165.5) .. controls (318.46,165.5) and (317,164.04) .. (317,162.25) -- cycle ;
    \draw  [fill={rgb, 255:red, 0; green, 0; blue, 0 }  ,fill opacity=1 ] (257,100.75) .. controls (257,98.96) and (258.46,97.5) .. (260.25,97.5) .. controls (262.04,97.5) and (263.5,98.96) .. (263.5,100.75) .. controls (263.5,102.54) and (262.04,104) .. (260.25,104) .. controls (258.46,104) and (257,102.54) .. (257,100.75) -- cycle ;
    \draw  [fill={rgb, 255:red, 0; green, 0; blue, 0 }  ,fill opacity=1 ] (257,120.25) .. controls (257,118.46) and (258.46,117) .. (260.25,117) .. controls (262.04,117) and (263.5,118.46) .. (263.5,120.25) .. controls (263.5,122.04) and (262.04,123.5) .. (260.25,123.5) .. controls (258.46,123.5) and (257,122.04) .. (257,120.25) -- cycle ;
    \draw  [fill={rgb, 255:red, 0; green, 0; blue, 0 }  ,fill opacity=1 ] (257.25,80.8) .. controls (257.25,79.01) and (258.71,77.55) .. (260.5,77.55) .. controls (262.29,77.55) and (263.75,79.01) .. (263.75,80.8) .. controls (263.75,82.59) and (262.29,84.05) .. (260.5,84.05) .. controls (258.71,84.05) and (257.25,82.59) .. (257.25,80.8) -- cycle ;
    \draw  [fill={rgb, 255:red, 0; green, 0; blue, 0 }  ,fill opacity=1 ] (377,120.75) .. controls (377,118.96) and (378.46,117.5) .. (380.25,117.5) .. controls (382.04,117.5) and (383.5,118.96) .. (383.5,120.75) .. controls (383.5,122.54) and (382.04,124) .. (380.25,124) .. controls (378.46,124) and (377,122.54) .. (377,120.75) -- cycle ;
    \draw  [fill={rgb, 255:red, 0; green, 0; blue, 0 }  ,fill opacity=1 ] (377,140.25) .. controls (377,138.46) and (378.46,137) .. (380.25,137) .. controls (382.04,137) and (383.5,138.46) .. (383.5,140.25) .. controls (383.5,142.04) and (382.04,143.5) .. (380.25,143.5) .. controls (378.46,143.5) and (377,142.04) .. (377,140.25) -- cycle ;
    \draw    (260.25,60.75) -- (320.25,40.75) ;
    \draw    (260.25,40.75) .. controls (308.3,29.1) and (330.8,28.6) .. (380.25,40.75) ;
    \draw    (260.5,80.8) -- (320.25,60.75) ;
    \draw    (320.25,80.75) -- (380.25,61.75) ;
    \draw    (320.25,101.75) -- (380.25,81.75) ;
    \draw    (320.25,121.25) -- (380.25,101.75) ;
    \draw    (320.25,142.25) -- (380.25,120.75) ;
    \draw    (260.5,80.8) .. controls (300.8,84.8) and (300.8,100.1) .. (260.25,100.75) ;
    
    \draw (40.5,125) node [anchor=north west][inner sep=0.75pt]   [align=left] {$\displaystyle p_{\alpha }$};
    \draw (102.5,218) node [anchor=north west][inner sep=0.75pt]   [align=left] {$\displaystyle p_{\gamma }$};
    \draw (163,143) node [anchor=north west][inner sep=0.75pt]   [align=left] {$\displaystyle p_{\beta }$};
    \draw (43.5,93) node [anchor=north west][inner sep=0.75pt]   [align=left] {$\displaystyle i_{\alpha }$};
    \draw (103.5,189) node [anchor=north west][inner sep=0.75pt]   [align=left] {$\displaystyle i_{\gamma }$};
    \draw (164,110) node [anchor=north west][inner sep=0.75pt]   [align=left] {$\displaystyle i_{\beta }$};
    \draw (251.5,156) node [anchor=north west][inner sep=0.75pt]   [align=left] {$\displaystyle p_{\alpha }$};
    \draw (312.5,220) node [anchor=north west][inner sep=0.75pt]   [align=left] {$\displaystyle p_{\gamma }$};
    \draw (372,178) node [anchor=north west][inner sep=0.75pt]   [align=left] {$\displaystyle p_{\beta }$};
    \draw (254,128) node [anchor=north west][inner sep=0.75pt]   [align=left] {$\displaystyle i_{\alpha }$};
    \draw (314.5,173) node [anchor=north west][inner sep=0.75pt]   [align=left] {$\displaystyle i_{\gamma }$};
    \draw (374,150) node [anchor=north west][inner sep=0.75pt]   [align=left] {$\displaystyle i_{\beta }$};
    \draw (240,95) node [anchor=north west][inner sep=0.75pt]   [align=left] {$\displaystyle v_{L}$};

    \end{tikzpicture}

    \captionof{figure}{Depictions of the strategies for two cases when $|p_\alpha| < |p_\beta| < |p_\gamma|$. The first shows the case when $|p_\gamma| - 1 > (|p_\alpha| - 1) + (|p_\beta| - 1)$, in which partitions are of sizes 3, 4, and 8 respectively. The second shows the case when $|p_\gamma| - 1 \leq (|p_\alpha| - 1) + (|p_\beta| - 1)$, in which partitions are of size 5, 6, and 7 respectively.}
    \label{fig:tripstrategy_diffall}

\end{center}


For the case that $| p_\alpha|  = | p_\beta|  < | p_\gamma| $, the strategy is again to maintain three designated vertices $i_\alpha, i_\beta, i_\gamma$, one in each partition, and then place as many matching measurements as possible from vertices in $p_\alpha \setminus \{i_\alpha\}$ to vertices in $p_\gamma \setminus \{i_\gamma\}$, and then as many as possible from vertices in $p_\beta \setminus \{i_\beta\}$ to vertices in $p_\gamma \setminus \{i_\gamma\}$. If any vertices remain in $p_\beta \setminus \{i_\beta\}$ or $p_\gamma \setminus \{i_\gamma\}$ (note that we cannot have vertices left over in both $p_\beta \setminus \{i_\beta\}$ and $p_\gamma \setminus \{i_\gamma\}$), place as many disjoint partition butterfly wing measurements (Definition~\ref{def:part_bw}) as possible, on the vertices that are left over. If there is one vertex $v$ left over, which is not the designated vertex in this partition, and it is not involved in any measurement so far, place a measurement from $v$ to any previously used vertex in that partition. If there are two vertices left over, place a partition butterfly measurement involving these two and the designated vertex in this partition (see Figure~\ref{fig:tripstrategy_diffsmallsame}). 

\begin{center}     
    
    \begin{tikzpicture}[x=0.75pt,y=0.75pt,yscale=-1,xscale=1]
    
    \draw   (49.5,64.65) .. controls (49.5,34.47) and (58.45,10) .. (69.5,10) .. controls (80.55,10) and (89.5,34.47) .. (89.5,64.65) .. controls (89.5,94.83) and (80.55,119.3) .. (69.5,119.3) .. controls (58.45,119.3) and (49.5,94.83) .. (49.5,64.65) -- cycle ;
    \draw   (110,101.27) .. controls (110,51.32) and (118.95,10.84) .. (130,10.84) .. controls (141.05,10.84) and (150,51.32) .. (150,101.27) .. controls (150,151.21) and (141.05,191.7) .. (130,191.7) .. controls (118.95,191.7) and (110,151.21) .. (110,101.27) -- cycle ;
    \draw   (170,66.02) .. controls (170,35.54) and (178.95,10.84) .. (190,10.84) .. controls (201.05,10.84) and (210,35.54) .. (210,66.02) .. controls (210,96.49) and (201.05,121.2) .. (190,121.2) .. controls (178.95,121.2) and (170,96.49) .. (170,66.02) -- cycle ;
    \draw  [fill={rgb, 255:red, 0; green, 0; blue, 0 }  ,fill opacity=1 ] (67,37.75) .. controls (67,35.96) and (68.46,34.5) .. (70.25,34.5) .. controls (72.04,34.5) and (73.5,35.96) .. (73.5,37.75) .. controls (73.5,39.54) and (72.04,41) .. (70.25,41) .. controls (68.46,41) and (67,39.54) .. (67,37.75) -- cycle ;
    \draw  [fill={rgb, 255:red, 0; green, 0; blue, 0 }  ,fill opacity=1 ] (67,58.75) .. controls (67,56.96) and (68.46,55.5) .. (70.25,55.5) .. controls (72.04,55.5) and (73.5,56.96) .. (73.5,58.75) .. controls (73.5,60.54) and (72.04,62) .. (70.25,62) .. controls (68.46,62) and (67,60.54) .. (67,58.75) -- cycle ;
    \draw  [fill={rgb, 255:red, 0; green, 0; blue, 0 }  ,fill opacity=1 ] (67,78.75) .. controls (67,76.96) and (68.46,75.5) .. (70.25,75.5) .. controls (72.04,75.5) and (73.5,76.96) .. (73.5,78.75) .. controls (73.5,80.54) and (72.04,82) .. (70.25,82) .. controls (68.46,82) and (67,80.54) .. (67,78.75) -- cycle ;
    \draw  [fill={rgb, 255:red, 0; green, 0; blue, 0 }  ,fill opacity=1 ] (127,37.75) .. controls (127,35.96) and (128.46,34.5) .. (130.25,34.5) .. controls (132.04,34.5) and (133.5,35.96) .. (133.5,37.75) .. controls (133.5,39.54) and (132.04,41) .. (130.25,41) .. controls (128.46,41) and (127,39.54) .. (127,37.75) -- cycle ;
    \draw  [fill={rgb, 255:red, 0; green, 0; blue, 0 }  ,fill opacity=1 ] (127,58.75) .. controls (127,56.96) and (128.46,55.5) .. (130.25,55.5) .. controls (132.04,55.5) and (133.5,56.96) .. (133.5,58.75) .. controls (133.5,60.54) and (132.04,62) .. (130.25,62) .. controls (128.46,62) and (127,60.54) .. (127,58.75) -- cycle ;
    \draw  [fill={rgb, 255:red, 0; green, 0; blue, 0 }  ,fill opacity=1 ] (127,78.75) .. controls (127,76.96) and (128.46,75.5) .. (130.25,75.5) .. controls (132.04,75.5) and (133.5,76.96) .. (133.5,78.75) .. controls (133.5,80.54) and (132.04,82) .. (130.25,82) .. controls (128.46,82) and (127,80.54) .. (127,78.75) -- cycle ;
    \draw  [fill={rgb, 255:red, 0; green, 0; blue, 0 }  ,fill opacity=1 ] (127,98.75) .. controls (127,96.96) and (128.46,95.5) .. (130.25,95.5) .. controls (132.04,95.5) and (133.5,96.96) .. (133.5,98.75) .. controls (133.5,100.54) and (132.04,102) .. (130.25,102) .. controls (128.46,102) and (127,100.54) .. (127,98.75) -- cycle ;
    \draw  [fill={rgb, 255:red, 0; green, 0; blue, 0 }  ,fill opacity=1 ] (187,37.75) .. controls (187,35.96) and (188.46,34.5) .. (190.25,34.5) .. controls (192.04,34.5) and (193.5,35.96) .. (193.5,37.75) .. controls (193.5,39.54) and (192.04,41) .. (190.25,41) .. controls (188.46,41) and (187,39.54) .. (187,37.75) -- cycle ;
    \draw  [fill={rgb, 255:red, 0; green, 0; blue, 0 }  ,fill opacity=1 ] (187,58.75) .. controls (187,56.96) and (188.46,55.5) .. (190.25,55.5) .. controls (192.04,55.5) and (193.5,56.96) .. (193.5,58.75) .. controls (193.5,60.54) and (192.04,62) .. (190.25,62) .. controls (188.46,62) and (187,60.54) .. (187,58.75) -- cycle ;
    \draw  [fill={rgb, 255:red, 0; green, 0; blue, 0 }  ,fill opacity=1 ] (187,78.75) .. controls (187,76.96) and (188.46,75.5) .. (190.25,75.5) .. controls (192.04,75.5) and (193.5,76.96) .. (193.5,78.75) .. controls (193.5,80.54) and (192.04,82) .. (190.25,82) .. controls (188.46,82) and (187,80.54) .. (187,78.75) -- cycle ;
    \draw  [fill={rgb, 255:red, 0; green, 0; blue, 0 }  ,fill opacity=1 ] (127,120.25) .. controls (127,118.46) and (128.46,117) .. (130.25,117) .. controls (132.04,117) and (133.5,118.46) .. (133.5,120.25) .. controls (133.5,122.04) and (132.04,123.5) .. (130.25,123.5) .. controls (128.46,123.5) and (127,122.04) .. (127,120.25) -- cycle ;
    \draw  [fill={rgb, 255:red, 0; green, 0; blue, 0 }  ,fill opacity=1 ] (127,140.25) .. controls (127,138.46) and (128.46,137) .. (130.25,137) .. controls (132.04,137) and (133.5,138.46) .. (133.5,140.25) .. controls (133.5,142.04) and (132.04,143.5) .. (130.25,143.5) .. controls (128.46,143.5) and (127,142.04) .. (127,140.25) -- cycle ;
    \draw  [fill={rgb, 255:red, 0; green, 0; blue, 0 }  ,fill opacity=1 ] (127,160.25) .. controls (127,158.46) and (128.46,157) .. (130.25,157) .. controls (132.04,157) and (133.5,158.46) .. (133.5,160.25) .. controls (133.5,162.04) and (132.04,163.5) .. (130.25,163.5) .. controls (128.46,163.5) and (127,162.04) .. (127,160.25) -- cycle ;
    \draw    (70.25,37.75) -- (130.25,37.75) ;
    \draw    (70.25,58.75) -- (130.25,58.75) ;
    \draw    (130.25,78.75) -- (190.25,37.75) ;
    \draw    (130.25,98.75) -- (190.25,58.75) ;
    \draw    (130.25,120.25) .. controls (169.8,118.6) and (170.8,139.6) .. (130.25,140.25) ;
    \draw    (130.25,140.25) .. controls (169.8,138.6) and (170.8,159.6) .. (130.25,160.25) ;
    
    \draw (60.5,123) node [anchor=north west][inner sep=0.75pt]   [align=left] {$\displaystyle p_{\alpha }$};
    \draw (122.5,197) node [anchor=north west][inner sep=0.75pt]   [align=left] {$\displaystyle p_{\gamma }$};
    \draw (183,125) node [anchor=north west][inner sep=0.75pt]   [align=left] {$\displaystyle p_{\beta }$};
    \draw (63.5,90) node [anchor=north west][inner sep=0.75pt]   [align=left] {$\displaystyle i_{\alpha }$};
    \draw (123.5,168) node [anchor=north west][inner sep=0.75pt]   [align=left] {$\displaystyle i_{\gamma }$};
    \draw (184,90) node [anchor=north west][inner sep=0.75pt]   [align=left] {$\displaystyle i_{\beta }$};

    \end{tikzpicture}

    \captionof{figure}{Depiction of the strategy $|p_\alpha| = |p_\beta| < |p_\gamma|$. In this case, two partitions are of size 3 and one is of size 7.}
    \label{fig:tripstrategy_diffsmallsame}

\end{center}


For the case that $| p_\alpha|  < | p_\beta|  = | p_\gamma| $, the strategy is again to maintain three designated vertices $i_\alpha, i_\beta, i_\gamma$, one in each partition, and then place as many matching measurements as possible from $p_\alpha \setminus \{i_\alpha\}$ to $p_\beta \setminus \{i_\beta\}$. Next, we place as many disjoint partition butterfly wing measurements (Definition~\ref{def:part_bw}) as possible, on the vertices that are left over in $p_\beta \setminus \{i_\beta\}$ and $p_\gamma \setminus \{i_\gamma\}$. If there is one vertex left over in one partition and two left over in the other, we place a hairpin butterfly wing measurement (Definition~\ref{def:hairpin_bip_bw}) using these vertices. If there are two vertices left over in each partition, place a hairpin butterfly wing measurement using three of these vertices and place a measurement from the last vertex to any other previously used vertex in that partition. Otherwise, for both partitions $p_\beta$ and $p_\gamma$, we do the following. If there is one vertex left over, which is not the designated vertex in that partition, and it is not involved in any measurement so far, place a measurement from that vertex to any previously used vertex in that partition. If there are two vertices left over, place a partition butterfly measurement involving these two and the designated vertex in that partition (see Figure~\ref{fig:tripstrategy_diffbigsame}).

\begin{center}     
    
    \begin{tikzpicture}[x=0.75pt,y=0.75pt,yscale=-1,xscale=1]
    
    \draw   (69.5,64.65) .. controls (69.5,34.47) and (78.45,10) .. (89.5,10) .. controls (100.55,10) and (109.5,34.47) .. (109.5,64.65) .. controls (109.5,94.83) and (100.55,119.3) .. (89.5,119.3) .. controls (78.45,119.3) and (69.5,94.83) .. (69.5,64.65) -- cycle ;
    \draw   (130,101.27) .. controls (130,51.32) and (138.95,10.84) .. (150,10.84) .. controls (161.05,10.84) and (170,51.32) .. (170,101.27) .. controls (170,151.21) and (161.05,191.7) .. (150,191.7) .. controls (138.95,191.7) and (130,151.21) .. (130,101.27) -- cycle ;
    \draw  [fill={rgb, 255:red, 0; green, 0; blue, 0 }  ,fill opacity=1 ] (87,37.75) .. controls (87,35.96) and (88.46,34.5) .. (90.25,34.5) .. controls (92.04,34.5) and (93.5,35.96) .. (93.5,37.75) .. controls (93.5,39.54) and (92.04,41) .. (90.25,41) .. controls (88.46,41) and (87,39.54) .. (87,37.75) -- cycle ;
    \draw  [fill={rgb, 255:red, 0; green, 0; blue, 0 }  ,fill opacity=1 ] (87,58.75) .. controls (87,56.96) and (88.46,55.5) .. (90.25,55.5) .. controls (92.04,55.5) and (93.5,56.96) .. (93.5,58.75) .. controls (93.5,60.54) and (92.04,62) .. (90.25,62) .. controls (88.46,62) and (87,60.54) .. (87,58.75) -- cycle ;
    \draw  [fill={rgb, 255:red, 0; green, 0; blue, 0 }  ,fill opacity=1 ] (87,78.75) .. controls (87,76.96) and (88.46,75.5) .. (90.25,75.5) .. controls (92.04,75.5) and (93.5,76.96) .. (93.5,78.75) .. controls (93.5,80.54) and (92.04,82) .. (90.25,82) .. controls (88.46,82) and (87,80.54) .. (87,78.75) -- cycle ;
    \draw  [fill={rgb, 255:red, 0; green, 0; blue, 0 }  ,fill opacity=1 ] (146.5,37.75) .. controls (146.5,35.96) and (147.96,34.5) .. (149.75,34.5) .. controls (151.54,34.5) and (153,35.96) .. (153,37.75) .. controls (153,39.54) and (151.54,41) .. (149.75,41) .. controls (147.96,41) and (146.5,39.54) .. (146.5,37.75) -- cycle ;
    \draw  [fill={rgb, 255:red, 0; green, 0; blue, 0 }  ,fill opacity=1 ] (146.5,58.75) .. controls (146.5,56.96) and (147.96,55.5) .. (149.75,55.5) .. controls (151.54,55.5) and (153,56.96) .. (153,58.75) .. controls (153,60.54) and (151.54,62) .. (149.75,62) .. controls (147.96,62) and (146.5,60.54) .. (146.5,58.75) -- cycle ;
    \draw  [fill={rgb, 255:red, 0; green, 0; blue, 0 }  ,fill opacity=1 ] (146.5,78.75) .. controls (146.5,76.96) and (147.96,75.5) .. (149.75,75.5) .. controls (151.54,75.5) and (153,76.96) .. (153,78.75) .. controls (153,80.54) and (151.54,82) .. (149.75,82) .. controls (147.96,82) and (146.5,80.54) .. (146.5,78.75) -- cycle ;
    \draw  [fill={rgb, 255:red, 0; green, 0; blue, 0 }  ,fill opacity=1 ] (146.5,98.75) .. controls (146.5,96.96) and (147.96,95.5) .. (149.75,95.5) .. controls (151.54,95.5) and (153,96.96) .. (153,98.75) .. controls (153,100.54) and (151.54,102) .. (149.75,102) .. controls (147.96,102) and (146.5,100.54) .. (146.5,98.75) -- cycle ;
    \draw  [fill={rgb, 255:red, 0; green, 0; blue, 0 }  ,fill opacity=1 ] (147,120.25) .. controls (147,118.46) and (148.46,117) .. (150.25,117) .. controls (152.04,117) and (153.5,118.46) .. (153.5,120.25) .. controls (153.5,122.04) and (152.04,123.5) .. (150.25,123.5) .. controls (148.46,123.5) and (147,122.04) .. (147,120.25) -- cycle ;
    \draw  [fill={rgb, 255:red, 0; green, 0; blue, 0 }  ,fill opacity=1 ] (147,140.25) .. controls (147,138.46) and (148.46,137) .. (150.25,137) .. controls (152.04,137) and (153.5,138.46) .. (153.5,140.25) .. controls (153.5,142.04) and (152.04,143.5) .. (150.25,143.5) .. controls (148.46,143.5) and (147,142.04) .. (147,140.25) -- cycle ;
    \draw  [fill={rgb, 255:red, 0; green, 0; blue, 0 }  ,fill opacity=1 ] (147,160.25) .. controls (147,158.46) and (148.46,157) .. (150.25,157) .. controls (152.04,157) and (153.5,158.46) .. (153.5,160.25) .. controls (153.5,162.04) and (152.04,163.5) .. (150.25,163.5) .. controls (148.46,163.5) and (147,162.04) .. (147,160.25) -- cycle ;
    \draw    (90.25,37.75) -- (150.25,37.75) ;
    \draw    (90.25,58.75) -- (150.25,58.75) ;
    \draw    (149.75,78.75) .. controls (189.3,77.1) and (190.3,98.1) .. (149.75,98.75) ;
    \draw    (149.75,98.75) .. controls (189.8,99.6) and (190.8,119.6) .. (150.25,120.25) ;
    \draw   (200,100.77) .. controls (200,50.82) and (208.95,10.34) .. (220,10.34) .. controls (231.05,10.34) and (240,50.82) .. (240,100.77) .. controls (240,150.71) and (231.05,191.2) .. (220,191.2) .. controls (208.95,191.2) and (200,150.71) .. (200,100.77) -- cycle ;
    \draw  [fill={rgb, 255:red, 0; green, 0; blue, 0 }  ,fill opacity=1 ] (216.5,38.25) .. controls (216.5,36.46) and (217.96,35) .. (219.75,35) .. controls (221.54,35) and (223,36.46) .. (223,38.25) .. controls (223,40.04) and (221.54,41.5) .. (219.75,41.5) .. controls (217.96,41.5) and (216.5,40.04) .. (216.5,38.25) -- cycle ;
    \draw  [fill={rgb, 255:red, 0; green, 0; blue, 0 }  ,fill opacity=1 ] (216.5,59.25) .. controls (216.5,57.46) and (217.96,56) .. (219.75,56) .. controls (221.54,56) and (223,57.46) .. (223,59.25) .. controls (223,61.04) and (221.54,62.5) .. (219.75,62.5) .. controls (217.96,62.5) and (216.5,61.04) .. (216.5,59.25) -- cycle ;
    \draw  [fill={rgb, 255:red, 0; green, 0; blue, 0 }  ,fill opacity=1 ] (216.5,79.25) .. controls (216.5,77.46) and (217.96,76) .. (219.75,76) .. controls (221.54,76) and (223,77.46) .. (223,79.25) .. controls (223,81.04) and (221.54,82.5) .. (219.75,82.5) .. controls (217.96,82.5) and (216.5,81.04) .. (216.5,79.25) -- cycle ;
    \draw  [fill={rgb, 255:red, 0; green, 0; blue, 0 }  ,fill opacity=1 ] (216.5,99.25) .. controls (216.5,97.46) and (217.96,96) .. (219.75,96) .. controls (221.54,96) and (223,97.46) .. (223,99.25) .. controls (223,101.04) and (221.54,102.5) .. (219.75,102.5) .. controls (217.96,102.5) and (216.5,101.04) .. (216.5,99.25) -- cycle ;
    \draw  [fill={rgb, 255:red, 0; green, 0; blue, 0 }  ,fill opacity=1 ] (217,120.75) .. controls (217,118.96) and (218.46,117.5) .. (220.25,117.5) .. controls (222.04,117.5) and (223.5,118.96) .. (223.5,120.75) .. controls (223.5,122.54) and (222.04,124) .. (220.25,124) .. controls (218.46,124) and (217,122.54) .. (217,120.75) -- cycle ;
    \draw  [fill={rgb, 255:red, 0; green, 0; blue, 0 }  ,fill opacity=1 ] (217,140.75) .. controls (217,138.96) and (218.46,137.5) .. (220.25,137.5) .. controls (222.04,137.5) and (223.5,138.96) .. (223.5,140.75) .. controls (223.5,142.54) and (222.04,144) .. (220.25,144) .. controls (218.46,144) and (217,142.54) .. (217,140.75) -- cycle ;
    \draw  [fill={rgb, 255:red, 0; green, 0; blue, 0 }  ,fill opacity=1 ] (217,160.75) .. controls (217,158.96) and (218.46,157.5) .. (220.25,157.5) .. controls (222.04,157.5) and (223.5,158.96) .. (223.5,160.75) .. controls (223.5,162.54) and (222.04,164) .. (220.25,164) .. controls (218.46,164) and (217,162.54) .. (217,160.75) -- cycle ;
    \draw    (150.25,120.25) .. controls (190.3,121.1) and (190.8,139.6) .. (150.25,140.25) ;
    \draw    (219.75,38.25) .. controls (259.8,39.1) and (260.3,58.6) .. (219.75,59.25) ;
    \draw    (219.75,59.25) .. controls (259.8,60.1) and (260.3,78.6) .. (219.75,79.25) ;
    \draw    (219.75,99.25) .. controls (259.8,100.1) and (260.8,120.1) .. (220.25,120.75) ;
    \draw    (220.25,120.75) .. controls (260.3,121.6) and (260.8,140.1) .. (220.25,140.75) ;
    
    \draw (80.5,123) node [anchor=north west][inner sep=0.75pt]   [align=left] {$\displaystyle p_{\alpha }$};
    \draw (213,197) node [anchor=north west][inner sep=0.75pt]   [align=left] {$\displaystyle p_{\gamma }$};
    \draw (141.5,197) node [anchor=north west][inner sep=0.75pt]   [align=left] {$\displaystyle p_{\beta }$};
    \draw (83.5,89) node [anchor=north west][inner sep=0.75pt]   [align=left] {$\displaystyle i_{\alpha }$};
    \draw (215,168) node [anchor=north west][inner sep=0.75pt]   [align=left] {$\displaystyle i_{\gamma }$};
    \draw (143.5,168) node [anchor=north west][inner sep=0.75pt]   [align=left] {$\displaystyle i_{\beta }$};

    \end{tikzpicture}

    \captionof{figure}{Depiction of the strategy when $|p_\alpha| < |p_\beta| = |p_\gamma|$. In this case, two partitions are of size 7 and one is of size 3.}
    \label{fig:tripstrategy_diffbigsame}

\end{center}


\paragraph*{Measurement Counts}

We provide the bounds for the four cases in Table~\ref{tab:bounds_tripart}.

\paragraph*{Correctness}
The proofs follow a similar argument to that provided in Case 2 of the proof of correctness for complete $k$-partite graphs.

\subsubsection{Lower Bound in Theorem~\ref{thm:tripartite}}

\begin{claim}\label{LB:tripartite_diff_sizes}
Given a complete tripartite graph $G$ with partitions of different sizes $| p_\alpha|  < | p_\beta|  < | p_\gamma| $, the lower bound on the number of measurements needed is $\left\lceil{\frac{n-3}{2}}\right\rceil$ when $| p_\alpha|  + | p_\beta|  -1 \geq | p_\gamma| $, and $\ceil{\frac{n}{3} + \frac{1}{3}| p_\gamma|  - \frac{4}{3}}$ when $| p_\alpha|  + | p_\beta|  -1 < | p_\gamma| $.
\end{claim}

Let $\mathcal{M}$ be any set of measurements that solves the faulty edge detection problem for the complete tripartite graph and consider the measurement graph $G' = (V,\mathcal{M})$ from  Definition~\ref{def:measurement-graph}. We recognize that there can be at most one isolated vertex in each partition in $G'$. Otherwise, let $a,a'$ be two isolated vertices in the same partition and let $b$ be any vertex in a different partition of $G'$. Consider a measurement $(r,s)$. Because $a$ and $a'$ are in the same partition and they are not an endpoint of any measurements, when $(a,b)$ or $(a',b)$ are altered, they will have the same effective resistance value coming from columns \textbf{III,IV,VI,VII} or \textbf{VIII} in Table \ref{table:k_part-diff-part} if $s \not\in p_r$, and from columns \textbf{X, XI} or \textbf{XII} in Table \ref{table:k_part-same-part} if $s \in p_r$. Further, there can be no component of size two formed using vertices in the same partition. Otherwise, let $a, a'$ be two vertices in the same partition that form a component of size two in $G'$ and let $b$ be any vertex in a different partition in $G'$. Because $a$ and $a'$ are not endpoints in any other measurements, the edges $(a,b)$ and $(a',b)$ will have the same effective resistance value when altered. This value will come from column \textbf{X} in Table \ref{table:k_part-same-part}. Thus, there are at least $n-3$ vertices with degree at least 1 in $G'$, meaning the sum of the degrees in $G'$ is at least $n-3$. By the Handshaking Lemma, there must be at least $\frac{n-3}{2}$ edges in $G'$ and since the number of edges must be an integer, there are at least $\left\lceil\frac{n-3}{2}\right\rceil$ edges in $G'$. This bound holds no matter what the sizes of the partitions are, and so it holds when $| p_\alpha|  + | p_\beta|  -1 \geq | p_\gamma| $.

When $| p_\alpha|  + | p_\beta|  -1 < | p_\gamma| $, we must consider two cases. First, consider the situation where the measurement graph has three isolated vertices, one in each partition. As argued earlier, since size two components must involve vertices from two different partitions, the number $q_2$ of size two components can be at most $| p_\alpha|  - 1 + | p_\beta|  - 1$. The rest of the components of the measurement graph $G'$ are of size $s\geq 3$ and therefore must contain at least $s-1 \geq 2s/3$ edges. Thus, the number of measurements in $\mathcal{M}$ is at least
\begin{equation*}
\begin{array}
{rcl}
q_2 + \frac{2(n-3-2q_2)}{3} & = &\frac{2n}{3} - \frac{q_2}{3} - 2 \\
& \geq &\frac{2n}{3} - \frac{| p_\alpha|  + | p_\beta|  - 2}{3} - 2 \\
& = & \frac{n}{3} + \frac{| p_\gamma| }{3} - \frac43
\end{array}
\end{equation*}

Next, consider the possibility that we have at most two isolated vertices in $G'$. As argued before, the number $q_2$ of size two components can be at most $| p_\alpha|  + | p_\beta| $ (since we may not have isolated vertices in these partitions anymore). Redoing the calculations, the number of measurements is at least
\begin{equation*}
\begin{array}
{rcl}
q_2 + \frac{2(n-2-2q_2)}{3} & = &\frac{2n}{3} - \frac{q_2}{3} - \frac43 \\
& \geq &\frac{2n}{3} - \frac{| p_\alpha|  + | p_\beta| }{3} - \frac43 \\
& = & \frac{n}{3} + \frac{| p_\gamma| }{3} - \frac43
\end{array}
\end{equation*}

Since the number of measurements is an integer, we thus have at least $\ceil{\frac{n}{3} + \frac{1}{3}| p_\gamma|  - \frac{4}{3}}$ measurements when $| p_\alpha|  + | p_\beta|  -1 < | p_\gamma| $.

\begin{claim}\label{LB:tripartite_same_sizes}
Given a complete tripartite graph $G$ with partitions of the same size $| p| $, the lower bound on the number of measurements needed is $2| p| -2$.
\end{claim}

\begin{proof}
We first make the following observations about the measurement graph $G'$ (Definition~\ref{def:measurement-graph}). First, there can be no more than one isolated node in $G'$ within any given partition. This follows from the same argument described in Claim~\ref{LB:tripartite_diff_sizes}. Second, a pair of partitions of the same size can have at most one component of size two between them. Otherwise, let vertices $a \in p_a$ and $b \in p_b$ form one component and $a' \in p_a$ and $b' \in p_b$ form the other, meaning $(a,b)$ and $(a',b')$ are two measurements. Since there are no other measurement edges in $G'$ incident on $a,b,a',b'$, the edges $(a,b')$ and $(a',b)$ will give the same effective resistance value when altered. Indeed, when looking at $(a,b)$ or $(a',b')$ as a measurement, the effective resistance values of these edges come from columns \textbf{II} and \textbf{III} in Table~\ref{table:k_part-diff-part}, which are equal as $| p_r|  = | p_s| $. When looking at any other measurement, their values will be the same and come from columns \textbf{IV}, \textbf{VI}, or \textbf{VIII} in Table~\ref{table:k_part-diff-part} when the measurement endpoints are in different partitions, and from columns \textbf{XI} or \textbf{XII} in Table~\ref{table:k_part-same-part} when the measurement endpoints are in the same partition. Third, it is not possible for two partitions to each have isolated vertices in $G'$ and also contain the endpoints of a size two component in $G'$. Otherwise, suppose $(a,b) \in E(G')$ and $i_a, i_b$ are two isolated nodes such that $a,i_a \in p_a$ and $b, i_b \in p_b$. There are no other measurement edges in $G'$ incident on $a,b,i_a,i_b$, which implies that edges $(a,i_b)$ and $(i_a,b)$ will be indistinguishable. Indeed, when looking at $(a,b)$ as a measurement, the effective resistance values of the edges will be the same as they come from columns \textbf{II} and \textbf{III} in Table~\ref{table:k_part-diff-part} with $| p_r|  = | p_s| $. When looking at any other measurement in $G'$, the values will come from columns \textbf{IV}, \textbf{VI}, and \textbf{VIII} in Table~\ref{table:k_part-diff-part} when the measurement endpoints are in different partitions and from columns \textbf{XI} and \textbf{XII} in Table~\ref{table:k_part-same-part} when the measurement endpoints are in the same partition. Finally, as discussed in Claim~\ref{LB:tripartite_diff_sizes}, there can be no component of size two in $G'$ formed from two vertices in the same partition.

We now analyze three different cases. The first case is that there is one isolated node in $G'$. Let $q_2$ be the number of size two components in $\mathcal{M}$. Due to the second restriction, $q_2 \leq 3$. The remaining components will therefore be of size $s \geq 3$ and they must have at least $s - 1 \geq 2s/3$ edges. Therefore, the number of measurements in $\mathcal{M}$ is at least
\begin{equation*}
\begin{array}
{rcl}
q_2 + \frac{2(n-1-2q_2)}{3} & = &\frac{2n}{3} - \frac{q_2}{3} - \frac23 \\
& \geq &\frac{2n}{3} - \frac{3}{3} - \frac23 \\
& = & \frac{2(3| p| )}{3} - \frac53 \\
& \geq & 2| p|  - 2
\end{array}
\end{equation*}

The second case is that there are two isolated nodes in $G'$. The second and third restrictions imply that $q_2 \leq 2$. The remaining components of size $s \geq 3$ must have at least $2s/3$ edges. Therefore, the number of measurement in $\mathcal{M}$ is at least
\begin{equation*}
\begin{array}
{rcl}
q_2 + \frac{2(n-2-2q_2)}{3} & = &\frac{2n}{3} - \frac{q_2}{3} - \frac43 \\
& \geq &\frac{2n}{3} - \frac{2}{3} - \frac43 \\
& = & \frac{2(3| p| )}{3} - 2 \\
& \geq & 2| p|  - 2
\end{array}
\end{equation*}

The third and final case is that there are three isolated nodes in $G'$. Given the second and third restrictions, we recognize that there cannot be a component of size two in $\mathcal{M}$. Therefore, the remaining components in $\mathcal{M}$ must be of size $s \geq 3$ containing at least $2s/3$ edges. Therefore, the number of measurements in $\mathcal{M}$ is at least
\begin{equation*}
\begin{array}
{rcl}
\frac{2(n-3)}{3} & = &\frac{2n}{3} - 2 \\
& = & \frac{2(3| p| )}{3} - 2 \\
& \geq & 2| p|  - 2
\end{array}
\end{equation*}
In all three cases, we end up with a lower bound of $2| p|  - 2$ as claimed.
\end{proof}

\begin{claim}\label{LB:tripartite_two_same_size_smaller}
Given a complete tripartite graph $G$ with partitions of sizes $| p_\alpha|  = | p_\beta|  < | p_\gamma| $, the lower bound on the number of measurements needed is $\left\lceil{\frac{2}{3}n - \frac{2}{3}
| p_\alpha|  - \frac{4}{3}} \right\rceil$ when $2(| p_\alpha|  -1) \leq | p_\gamma| -1$ and it is $\left\lceil{\frac{2n}{3} - \frac{| p_\gamma| }{3} - \frac{5}{3}} \right\rceil$ when $2(| p_\alpha|  -1) > | p_\gamma| -1$. 
\end{claim}

\begin{proof}
We first make some structural observations about the measurement graph $G'$ obtained from any set $\mathcal{M}$ of valid measurements (Definition~\ref{def:measurement-graph}).
First, there can be no more than one isolated node within any given partition in $G'$, as argued in Claim~\ref{LB:tripartite_diff_sizes}.  Second, since $p_\alpha$ and $p_\beta$ have the same size, there can be at most one component of size two between them for the same reason given in Claim~\ref{LB:tripartite_same_sizes}. Third, if $p_\alpha$ and $p_\beta$ each contain an isolated node, they cannot have a component of size two between them. Reasoning for this is likewise given in Claim~\ref{LB:tripartite_same_sizes}. Lastly, there can be no component of size two formed using vertices from the same partition. This can be seen in Claim~\ref{LB:tripartite_diff_sizes}.

We now analyze two cases to establish the lower bound. Suppose there are three isolated nodes in $G'$. Let $q_2$ be the number of components of size two. The restrictions observed above on the size two components imply that these are formed using vertices in $p_\alpha$ and $p_\gamma$, or using vertices in $p_\beta$ and $p_\gamma$. In other words, $q_2 \leq \min\{2(| p_\alpha|  - 1), | p_\gamma|  - 1\}$.
The remaining components of size $s \geq 3$ must have at least $2s/3$ edges. This means that the number of measurements in $\mathcal{M}$ is at least $q_2 + \frac{2(n-3-2q_2)}{3} = \frac{2n}{3} - \frac{q_2}{3} - 2$. If $2(| p_\alpha| -1) \leq | p_\gamma|  -1$, then $q_2 \leq 2(| p_\alpha| -1)$ and we obtain that $\frac{2n}{3} - \frac{q_2}{3} - 2 \geq \frac{2n}{3} - \frac{(2| p_\alpha|  - 2)}{3} - 2 = \frac{2n}{3} - \frac{2| p_\alpha| }{3} - \frac43$. Since the number of measurements is an integer, we must have at least $\left\lceil{\frac{2}{3}n - \frac{2}{3}
| p_\alpha|  - \frac{4}{3}} \right\rceil$ measurements. If $2(| p_\alpha| -1) > | p_\gamma|  -1$, then $q_2 \leq | p_\gamma| -1$. Doing the same calculations with $\frac{2n}{3} - \frac{q_2}{3} - 2$ gives $\left\lceil{\frac{2n}{3} - \frac{| p_\gamma| }{3} - \frac{5}{3}} \right\rceil$.

Suppose now there are at most two isolated nodes in $G'$. Given the restrictions on the size two components, we must have $q_2 \leq \min\{2| p_\alpha| ,| p_\gamma| \}$. The remaining components of size $s \geq 3$ must have at least $2s/3$ edges, which implies that the number of measurements in $\mathcal{M}$ is at least $q_2 + \frac{2(n-2-2q_2)}{3} =\frac{2n}{3} - \frac{q_2}{3} - \frac{4}{3}$. Using the fact that $q_2 \leq \min\{2| p_\alpha| ,| p_\gamma| \}$, we obtain that $\left\lceil{\frac{2n}{3} - \frac{q_2}{3} - \frac{4}{3}}\right\rceil \geq \max\left\{\left\lceil{\frac{2}{3}n - \frac{2}{3}| p_\alpha|  - \frac{4}{3}} \right\rceil, \left\lceil{\frac{2n}{3} - \frac{| p_\gamma| }{3} - \frac{5}{3}} \right\rceil\right\}$ giving us the overall bound.
\end{proof}

\begin{claim}\label{LB:tripartite_two_same_size_bigger}
Given a complete tripartite graph $G$ with partitions of sizes $| p_\alpha|  < | p_\beta|  = | p_\gamma| $, the lower bound on the number of measurements needed is $\left\lceil{\frac{2}{3}n - \frac{2}{3}
| p_\alpha|  - \frac{5}{3}} \right\rceil$
\end{claim}

\begin{proof}
We observe the following structural restrictions on the measurement graph $G'$ obtained from any set $\mathcal{M}$ of valid measurements (Definition~\ref{def:measurement-graph}). Each partition can contain at most one isolated node and there can be no components of size two formed from vertices in the same partition, as argued in Claim~\ref{LB:tripartite_diff_sizes}. Second, there can be at most one component of size two between partitions $p_\beta$ and $p_\gamma$ (see proof of Claim~\ref{LB:tripartite_same_sizes} for reasoning). Third, if $p_\beta$ and $p_\gamma$ both contain isolated nodes respectively, then there cannot be a component of size two between $p_\beta$ and $p_\gamma$ (again, see proof of Claim~\ref{LB:tripartite_same_sizes} for reasoning). 

We now analyze two cases to establish the lower bound. Suppose there are three isolated nodes in $G'$. Let $q_2$ be the number of components of size two. The restrictions above imply that $q_2 \leq | p_\alpha| -1$. The remaining components of size $s \geq 3$ must have at least $2s/3$ edges. Therefore, the number of measurements in $\mathcal{M}$ must be at least
\begin{equation*}
\begin{array}
{rcl}
q_2 + \frac{2(n-3-2q_2)}{3} & = &\frac{2n}{3} - \frac{q_2}{3} - 2 \\
& \geq &\frac{2n}{3} - \frac{(| p_\alpha| -1)}{3} - 2 \\
& = &\frac{2n}{3} - \frac{| p_\alpha| }{3} - \frac53
\end{array}
\end{equation*}

Now suppose that there are at most two isolated nodes in $G'$. Given the restrictions on the size two components, we can have at most one size two component between $p_\beta$ and $p_\gamma$, and the rest must have one vertex in $p_\alpha$ and the other in $p_\beta$ or $p_\gamma$. Thus, we can have $q_2 \leq | p_\alpha|  + 1$. The remaining components will be of size $s \geq 3$ and will need at least $2s/3$ edges. Therefore, the number of measurements in $\mathcal{M}$ is at least:
\begin{equation*}
\begin{array}
{rcl}
q_2 + \frac{2(n-2-2q_2)}{3} & = &\frac{2n}{3} - \frac{q_2}{3} - \frac{4}{3} \\
& \geq & \frac{2n}{3} - \frac{| p_\alpha|  + 1}{3} - \frac{4}{3} \\
& \geq & \frac{2n}{3} - \frac{| p_\alpha| }{3} - \frac{5}{3}
\end{array}
\end{equation*}

So, the final integer lower bound is $\left\lceil{\frac{2}{3}n - \frac{1}{3}| p_\alpha|  - \frac{5}{3}} \right\rceil$.
\end{proof}

The provided claims give the lower bounds on the number of the measurements which match their corresponding upper bounds, resulting in the number of measurements provided in Table~\ref{tab:bounds_tripart}.

\subsection{Complete $k$-partite Graphs}

We prove Theorem~\ref{thm:k-partite} in this section.  Lemma~\ref{incdient_edge_lemma_for_kpartite} in Appendix~\ref{A:edge_partition_k-partite} will provide the equivalence classes of the edges for a given measurement. The proof will utilize altered resistance values listed in Tables~\ref{table:k_part-diff-part} and \ref{table:k_part-same-part}. As mentioned above in  Section~\ref{sec:bipartite_graphs}, we use the convention in these two tables that for any measurement $(r,s)$, we use $p_r$ and $p_s$ to denote the partitions containing $r$ and $s$, respectively. We prove the stated lower bound below, and then establish the stated upper bound.

\subsubsection{Lower Bound in Theorem~\ref{thm:k-partite}
}\label{LB:thm:k-partite}

Let $\mathcal{M}$ be any set of measurements that solves the faulty edge detection problem for the complete $k$-partite graph with $k\geq 2$ and consider the measurement graph $G' = (V,\mathcal{M})$ from  Definition~\ref{def:measurement-graph}. We first observe that we can have at most one isolated vertex in each partition in $G'$. Otherwise, let $a,a'$ be two isolated vertices in the same partition $p_a$ and let $b \in p_b$ be any vertex in $G$ with $p_a \neq p_b$. Consider any measurement $(r,s)$. Since $a, a'$ are in the same partition and they are not the endpoint of any measurement, for both edges $(a,b)$ and $(a',b)$, the effective resistance value will come from the same column (one of \textbf{III, IV, VI, VII, VIII, or IX} from Table \ref{table:k_part-diff-part} if $r,s$ are in different partitions, or columns \textbf{X, XI, XII} from Table \ref{table:k_part-same-part} if $r,s$ are in the same partition) when they are altered.

Thus, we have at least $n-k$ vertices with degree at least 1 in $G'$ and so the sum of the degrees is at least $n-k$. By the Handshaking lemma, we must have at least $\frac{n-k}{2}$ edges, and since the number of edges is an integer, it must be at least $\left\lceil\frac{n-k}{2}\right\rceil$.

\subsubsection{Upper Bound in Theorem~\ref{thm:k-partite}
}\label{UB:thm:k-partite}

\paragraph*{The measurement strategy}\label{strategy:k-partite} 

We describe a measurement strategy for any tripartite graph. This will form the building block for the general $k$-partite case. Below, we employ notation like $u \in p_u$ because we feel this is intuitive; however, the symbol/variable $u$ is being used to denote both a vertex and a natural number, i.e.,~the index of the partition. We believe the intuitive clarity justifies this slight abuse of notation.

\begin{definition}[Tripartite Graph Strategy]\label{def:tri_graph_strategy} Let $T$ be a tripartite graph with partitions $p_\alpha,p_\beta,p_\gamma$ such that $| p_\alpha|  \leq | p_\beta|  \leq | p_\gamma| $. Maintain one designated node in each partition, i.e.,~$i_\alpha \in p_\alpha,i_\beta \in p_\beta,v_\gamma \in p_\gamma$. Next, place as many disjoint tripartite wing measurements (Definition~\ref{def:tri_bw}) as possible that do not involve any of the designated nodes. If there are nodes left in $p_\beta \setminus\{i_\beta\}$ that are not used in a measurement, place as many disjoint zig-zagging bipartite wing measurement schemes (Definition~\ref{def:zigzag_bip_bw}) as possible, between $p_\beta$ and $p_\gamma$ using the remaining nodes, but without involving the designated nodes $i_\beta, v_\gamma$. If one node $v \in p_\beta$ remains that is unable to be placed in a zig-zagging butterfly wing measurement scheme, place a hairpin butterfly wing measurement (Definition~\ref{def:hairpin_bip_bw}) between $p_\beta$ and $p_\gamma$ with $v\in p_\beta$ as the center node and without involving any of the previously used nodes. Note that if $| p_\beta|  = | p_\gamma| $, then one of the wing nodes of this hairpin butterfly wing measurement will be $v_\gamma$. If there are two nodes $v_1,v_2  \in p_\beta\setminus\{i_\beta\}$ that have not been used in any measurements so far, place a hairpin butterfly wing measurement between $p_\beta$ and $p_\gamma$ with $v_1,v_2$ as its wing nodes, and any unused node in $p_\gamma$ as the center node. If there are nodes left in $p_\gamma \setminus\{v_\gamma\}$ that have not been used in any measurements so far, place as many disjoint partition butterfly wing measurements (Definition~\ref{def:part_bw}) as possible within $p_\gamma \setminus\{v_\gamma\}$. If one node $u_\gamma$ remains in $p_\gamma \setminus\{v_\gamma\}$ that has not been used in any measurements so far, make measurement $(u_\gamma,w_\gamma)$ where $w_\gamma$ is any previously used center node in $p_\gamma$. If two nodes $u_\gamma, w_\gamma$ remain that have not been used in any measurements so far, place a partition butterfly wing measurement involving $u_\gamma, w_\gamma, v_\gamma$, with $w_\gamma$ as the center node.
\end{definition}

We now describe the strategy for the general $k$-partite graph with partitions $p_1, \ldots, p_k$ satisfying $2 \leq | p_1|  \leq\cdots\leq | p_k| $. If $k\equiv 1 \mod{3}$, let $i \in \{1, \ldots, k\}$ be the index that achieves the minimum in
\begin{equation*}
\min_{i\in\{1,\ldots, k\}}\left\{\left\lceil\frac{2| p_i| }{3}\right\rceil + val(L\setminus\{i\})\right\}
\end{equation*}
and let the partition $p_i$ be called $P$. If $k\equiv 2 \mod{3} $, let $i,j \in \{1, \ldots, k\}$ be the indices that achieve the minimum in
\begin{equation*}
\min_{i, j\in \{1,\ldots, k\},\; i\neq j}\left\{\left\lceil\frac{2(| p_i|  + | p_j|  - 1)}{3}\right\rceil + val(L\setminus\{i,j\})\right\}.
\end{equation*}
In this case, the partitions $p_i$ and $p_j$ form a bipartite subgraph that will be called $B$ in the following. In both cases, we group the remaining partitions, i.e.,~indexed by $\{1, \ldots, k\}\setminus \{i\}$ in the first case and $\{1, \ldots, k\}\setminus \{i,j\}$ in the second case, into tripartite subgraphs in order of the sizes of the partitions. We refer to these tripartite subgraphs as $T_1, \ldots, T_{\lfloor k/3\rfloor}$.

We next implement the following steps. 
\begin{enumerate}
\item The tripartite graph strategy (Definition~\ref{def:tri_graph_strategy}) is applied to $T_i$ for all $i = 1, \ldots, \lfloor k/3 \rfloor$.
\item If $k\equiv 1 \mod{3}$, place as many partition butterfly wing measurements (Definition~\ref{def:part_bw}) within this partition as possible. If any vertices remain in $P$ that were not used in any of these partition butterfly measurements, place a measurement from each of these vertices to $w$, where $w$ is any center node of a partition butterfly measurement in $P$.
\item If $k\equiv 2 \mod{3}$, let $p_i, p_j$ denote the partitions used to form the bipartite subgraph $B$ such that $| p_i|  \leq | p_j| $. Maintain one designated node $u \in p_i$. Place as many disjoint zig-zagging bipartite wing measurement schemes (Definition~\ref{def:zigzag_bip_bw}) as possible, between $p_i$ and $p_j$ without involving the designated node $u$. If one node $v \in p_i\setminus\{u\}$ remains that is unable to be placed in a zig-zagging butterfly wing measurement scheme, place a hairpin butterfly wing measurement (Definition~\ref{def:hairpin_bip_bw}) between $p_i$ and $p_j$ with $v\in p_i\setminus\{u\}$ as the center node, without involving any of the previously used nodes. If there are two nodes $v_1,v_2  \in p_i\setminus\{u\}$ that have not been used in any measurements so far, place a hairpin butterfly wing measurement between $p_i$ and $p_j$ with $v_1,v_2$ as its wing nodes, and any unused node in $p_j$ as the center node. If there are nodes left in $p_j$ that have not been used in any measurements so far, place as many disjoint partition butterfly wing measurements (Definition~\ref{def:part_bw}) as possible within $p_j$. If one node $v$ remains in $p_j$ that has not been used in any measurements so far, make measurement $(v,w)$ where $w$ is any previously used center node in $p_j$. If two nodes $v_1, v_2$ remain that have not been used in any measurements so far, place a partition butterfly wing measurement involving $v_1, w, v_2$, where $w$ is any previously-used center node in $p_j$.
\end{enumerate}

We claim that the above measurement strategy suffices.

\paragraph*{Measurement Counts}\label{measurement_counts}
We first count the number of measurements used in the tripartite graph measurement strategy (Definition~\ref{def:tri_graph_strategy}). One has to consider nine different cases, depending on the relationships between the partition sizes $| p_\alpha| , | p_\beta| $ and $| p_\gamma| $. It can be verified that these counts are given by the values enumerated in Table~\ref{tab:num_of_meas_kpart}. Since $| p_\alpha|  \leq | p_\beta|  \leq | p_\gamma| $ in all cases and the number of measurements must be integer, we see from the values in Table~\ref{tab:num_of_meas_kpart} that we have a general upper bound of $2| p_\gamma| -2$.

For the general $k$-partite case, we observe that if $k \equiv 0 \mod{3}$, then we simply have the bounds from the $k/3$ tripartite subgraphs, giving us an upper bound of $\sum_{i \in [k]: i \mod{3} = 0} (2| p_i|  -2)$. If $k \equiv 1 \mod{3}$, then once we select the isolated partition $P$ as per our strategy, we can apply the tripartite upper bound of $2| p_\gamma|  - 2$ for each of the tripartite subgraphs $T_i$, $i=1, \ldots, \lfloor k/3 \rfloor$. By definition of $val(L)$ for an ordered list $L$ of numbers (Definition~\ref{def:list}) and the choice of $P$, we have the stated bound. Similarly, in the case where $k \equiv 2 \mod{3}$, our choice of the bipartite subgraph $B$ and the definition of $val(L)$ gives the stated bound.

\renewcommand{\arraystretch}{1.5}
\begin{table}
    \centering
    \begin{tabular}{|p{3.8cm}|p{3.9cm}|p{4.7cm}|}
    \hline
    \multicolumn{2}{|c|}{\textbf{Cases of Partition Sizes}} & \multicolumn{1}{|c|}{\textbf{Number of Measurements}} \\
    \hline
    \hline
    \multirow{3}{3.3cm}{$|p_\beta| - |p_\alpha| \equiv 0 \mod{3}$} & $|p_\gamma| - |p_\beta| \equiv 0 \mod{3}$ & $\frac{2|p_\alpha|}{3} + \frac{2|p_\beta|}{3} + \frac{2|p_\gamma|}{3} - 2$ \\ 
    \cline{2-3}
    & $|p_\gamma| - |p_\beta| \equiv 1 \mod{3}$ & $\frac{2|p_\alpha|}{3} + \frac{2|p_\beta|}{3} + \frac{2|p_\gamma|}{3} - \frac{5}{3}$ \\ 
    \cline{2-3}
    & $|p_\gamma| - |p_\beta| \equiv 2 \mod{3}$ & $\frac{2|p_\alpha|}{3} + \frac{2|p_\beta|}{3} + \frac{2|p_\gamma|}{3} - \frac{4}{3}$ \\ 
    \hline
    \multirow{3}{3.3cm}{$|p_\beta| - |p_\alpha| \equiv 1 \mod{3}$} & $|p_\gamma| - |p_\beta| \equiv 1 \mod{3}$ & $\frac{2|p_\alpha|}{3} + \frac{2|p_\beta|}{3} + \frac{2|p_\gamma|}{3} - 2$ \\ 
    \cline{2-3}
    & $|p_\gamma| - |p_\beta| \equiv 2 \mod{3}$ & $\frac{2|p_\alpha|}{3} + \frac{2|p_\beta|}{3} + \frac{2|p_\gamma|}{3} - \frac{5}{3}$ \\ 
    \cline{2-3}
    & $|p_\gamma| - |p_\beta| \equiv 0 \mod{3}$ & $\frac{2|p_\alpha|}{3} + \frac{2|p_\beta|}{3} + \frac{2|p_\gamma|}{3} - \frac{4}{3}$ \\ 
    \hline
    \multirow{3}{3.3cm}{$|p_\beta| - |p_\alpha| \equiv 2 \mod{3}$} & $|p_\gamma| - |p_\beta| \equiv 2 \mod{3}$ & $\frac{2|p_\alpha|}{3} + \frac{2|p_\beta|}{3} + \frac{2|p_\gamma|}{3} - 2$ \\ 
    \cline{2-3}
    & $|p_\gamma| - |p_\beta| \equiv 0 \mod{3}$ & $\frac{2|p_\alpha|}{3} + \frac{2|p_\beta|}{3} + \frac{2|p_\gamma|}{3} - \frac{5}{3}$ \\ 
    \cline{2-3}
    & $|p_\gamma| - |p_\beta| \equiv 1 \mod{3}$ & $\frac{2|p_\alpha|}{3} + \frac{2|p_\beta|}{3} + \frac{2|p_\gamma|}{3} - \frac{4}{3}$ \\ 
    \hline
    \end{tabular}
    \caption{Number of Measurements required for various cases of partition sizes in k-partite graphs.}
    \label{tab:num_of_meas_kpart}
\end{table}

\paragraph*{Correctness}

Consider any two edges $e,e' \in E(G)$ with $e = (a,b)$ and $e' = (a',b')$, where $a \in p_a, b \in p_b, a' \in p_{a'}, b' \in p_{b'} $. We must show that there exists a measurement which gives different effective resistance values when these two edges are altered. Note that this notation implies, for example, that $s \in p_b$ if and only if $b \in p_s$. We will use such facts without explicit mention below. In the case analysis, when we consider tripartite subgraphs $T_i, T_j, T_k$, we allow for the possibility that some subset of the indices $\{i,j,k\}$ (possibly all) are all equal to each other unless some pairs are explicitly assumed to be unequal in the case under consideration.

\noindent\underline{\textit{Case 1: $e$ has both endpoints in $T_i$ and $e'$ has both endpoints in $T_j$ such that $i \neq j$}}

We observe that there exists a measurement $(r,s)$ within a tripartite butterfly wing measurement in $T_i$ such that $a \in p_r$. Using this measurement, the effective resistance when $e'$ is altered has value corresponding to column \textbf{IX} in Table \ref{table:k_part-diff-part}, and for $e$, the new effective resistance will have value from any of the columns from \textbf{I} to \textbf{VI} since $a \in p_r$. This implies that these two edges will give differing values of effective resistance for measurement $(r, s)$ when altered.

\noindent\underline{\textit{Case 2: $e$ and $e'$ have both endpoints in $T_i$ for some $i$}}

Since there are only three partitions in $T_i$, one of these partitions must contain an endpoint of both edges. Without loss of generality, let $a,a'$ be in the same partition $p_a$ in $T_i$. 

\underline{\textit{Case 2a: Suppose $b$ and $b'$ are in different partitions}.} In this case, any measurement $(r,s)$ within a tripartite butterfly wing measurement in $T_i$ such that $r \in p_a$ and $s \in p_b \cup p_{b'}$ will distinguish $e$ and $e'$. Since $b, b'$ are in different partitions, we may assume (up to a relabeling of the edges) that $s \in p_b$ or equivalently, $b \in p_s$. If $a = r$ and $b =s$, when $e$ is altered, the effective resistance value comes from column \textbf{I} and for $e'$, the value comes from either column \textbf{V} if $a' = a$ or column \textbf{VI} if $a' \neq a$. Similarly, if $a = r$ and $b \in p_s, b \neq s$, for $e$, the value comes from column \textbf{II}, and for $e'$, the value comes from either column \textbf{V} if $a' = a$ or column \textbf{VI} if $a' \neq a$. If $a \neq r$ and $b = s$, when $e$ is altered, the effective resistance value comes from column \textbf{III}, and for $e'$, the value comes from column \textbf{V} or column \textbf{VI} depending on whether or not $a' = r$. If $a \neq r$ and $b \in p_s, b \neq s$, for $e$ the value comes from column \textbf{IV}, and for $e'$ value comes from column \textbf{V} or column \textbf{VI} depending on whether or not $a' = r$.

\underline{\textit{Case 2b: Suppose $b$ and $b'$ are in the same partition}.}} Suppose $a = a'$. Either $b$ or $b'$ must be involved in a measurement since there is at most one designated node in the partition containing $b, b'$. Without loss of generality, let us assume that $b$ is one of the endpoints of the measurement $(r,s)$ with $b=s$. If $r = a = a'$, then the value when $e$ is altered comes from column \textbf{I} in Table \ref{table:k_part-diff-part}, and for $e'$, the value comes from column \textbf{II}. If $r \in p_a$ but $r \neq a$, then for $e$, the value comes from \textbf{III} and for $e'$, it is column \textbf{IV}. If $r$ is in the partition that is different from $p_a$ and $p_b$, then for $e$, the value comes from column \textbf{VII} and for $e'$, we use column \textbf{VIII}. Finally, if $r \in p_b$, then for $e$, the value comes from column \textbf{X} in Table \ref{table:k_part-same-part}, and for $e'$, the value comes from \textbf{XI} (Note that to read the values in Table \ref{table:k_part-same-part}, we have to switch the labels of the endpoints of the two edges). 

Suppose $b = b'$ (and therefore, $a\neq a'$). Observe that up to relabeling of the labels on the endpoints, this case is the same as the previous case with $a=a'$.

We now consider the situation where $a \neq a'$ and $b \neq b'$. Either $a$ or $a'$ must be involved in a measurement since there is at most one designated node in the partition $p_a$. Without loss of generality, let us assume that $a$ is one of the endpoints of the measurement $(r,s)$ with $a = r$. If $s$ is in the partition that does not contain $a$ or $b$, then if $e$ is altered, the effective resistance value comes from column \textbf{V}, and for $e'$, the value comes from column \textbf{VI}. If $s \in p_b$ with $s \neq b'$, then if $e$ is altered, the effective resistance value comes from either column \textbf{I} or \textbf{II} depending on whether $s = b$ or not, and for $e'$, the value comes from column \textbf{IV}. If $s \in p_b$ with $s = b'$, then there exists a different measurement $(\bar r, \bar s)$ in the strategy incident on either $a$ or $b'$ (since there are no components of size two in the measurement graph for the outlined measurement strategy). Again, up to switching the labels on the endpoints of the edges, we may assume $\bar r = a$. Now we must fall into one of the previously considered cases and we are done. Finally, consider the case where $s \in p_a$. If $s \neq a'$, then for $e$, the value comes from column \textbf{X} in Table \ref{table:k_part-same-part} and for $e'$, it comes from column \textbf{XI}. If $s = a'$, then there exists a different measurement $(\bar r, \bar s)$ in the strategy incident on either $a$ or $a'$ (since there are no components of size two in the measurement graph for the outlined measurement strategy). In fact, both $\bar r$ and $\bar s$ must be in $p_a$ because $r,s$ must be involved in a partition butterfly wing measurement within $p_a$. We are again able to distinguish edges $e, e'$ using columns \textbf{X} and \textbf{XI}, respectively, in Table \ref{table:k_part-same-part}.

\noindent\underline{\textit{Case 3: $e$ has both endpoints in $T_i$ and $e'$ has one endpoint in $T_j$ and another in $T_k$ such that $j \neq k$}}

Without loss of generality, assume $k \neq i$ and $b' \in p_{b'} \subset T_k$. There exists a measurement $(r,s)$ within a tripartite butterfly wing measurement in $T_k$ such that $b' \in p_{s}$ that will distinguish $e$ and $e'$. This is because the effective resistance value when $e$ is altered comes from column \textbf{IX} in Table \ref{table:k_part-diff-part} since $a,b \notin \{p_r,p_s\}$, and for $e'$ the value comes from columns \textbf{VII} or \textbf{VIII} depending on whether or not $b'=s$. This implies that the edges will have different values for the effective resistance when altered.

\noindent\underline{\textit{Case 4: $e$ has one endpoint in $T_{i_1}$ and another in $T_{i_2}$ $(i_1 < i_2)$ and $e'$ has one endpoint in $T_{j_1}$}}

\noindent\underline{\textit{and another in $T_{j_2}$ $(j_1 < j_2)$}}

First, we examine the case in which either $i_1 \neq j_1$ or $i_2 \neq j_2$. We consider the case $i_2 \neq j_2$; the other case has a symmetrical argument.
There exists a measurement $(r,s)$ within a tripartite butterfly wing measurement in $T_{j_2}$ such that $b' \in p_s \subset T_{j_2}$. Using this measurement, we will be able to distinguish $e$ and $e'$ as they will have different values for effective resistance; for $e$, the value comes from column \textbf{IX} in Table \ref{table:k_part-diff-part} since $a,b \notin \{ p_r,p_s\}$ and for $e'$, the value comes from \textbf{VII} or \textbf{VIII} depending on whether or not $b' = s$.

If $i_1 = j_1$ and $i_2 = j_2$, we first consider the case where either $a' \not\in p_a$ or $b' \not\in p_b$. In the first case, we can use a tripartite measurement $(r,s)$ such that $r \in p_a \cup p_{a'}$ and $s \in T_{i_1}/\{p_a,p_{a'}\}$ to distinguish the edges. Indeed, we may assume (up to changing the labels on the edges) that $a \in p_r$. Therefore, when $e$ is altered, the effective resistance value comes from column \textbf{V} or \textbf{VI} depending on whether or not $a = r$, and for $e'$, the value comes from column \textbf{IX} as $a',b' \notin \{p_r,p_s\}$. The case where $b' \not\in p_b$ can be argued in the same way after switching the labels on the endpoints of both edges.

So, we now consider the situation where $a'\in p_a$ and $b' \in p_b$. Since $e \neq e'$, either $a \neq a'$ or $b\neq b'$. Up to a relabeling of the endpoints, we may assume $b \neq b'$. Since $b,b' \in p_b$, one of them must be involved in a measurement $r,s$ since there is at most one designated vertex in every partition that is not involved in any measurement. Further, there must exist a measurement that involves $b$ or $b'$ but not both since there are no components of size two in the measurement graph associated with the outlined measurement strategy.

Without loss of generality, we assume that $b = s$. Regardless of whether or not $a = a'$, we are able to say the following about the effective resistance values when $e$ and $e'$ are altered. If $r$ is not in the same partition $p_b$ as $s$, for $e$, the value comes from column \textbf{VII} in Table \ref{table:k_part-diff-part} since $a \notin p_r$, and for $e'$, the value comes from column \textbf{VIII} as $b \in p_s, b\neq s$ and $a' \notin p_r$. 
If $r$ is in the same partition $p_b$ as $s$, for $e$, the value comes from column \textbf{X} in Table \ref{table:k_part-same-part}, and for $e'$, the value comes from column \textbf{XI} as $a' \notin p_r$ and $b' \in p_r$ (Note that one has to switch the labels on the endpoints of the edges to read the values from Table \ref{table:k_part-same-part}).

\noindent\underline{\textit{Case 5: $e$ has one endpoint in P and another in $T_i$ and $e'$ has one endpoint in $T_j$ and another in $T_k$}}

For edge $e$, let $a \in P$ and $b \in T_i$. We use a partition butterfly wing measurement $(r,s)$ within $P$ such that $a = r$. This means that when $e$ is altered, the effective resistance value comes from column \textbf{X}, and for $e'$, the value comes from column \textbf{XII} as $a, b \notin p_r$.

\noindent\underline{\textit{Case 6: $e$ has one endpoint in $P$ and another in $T_i$ and $e'$ has one endpoint in $P$ and another in $T_j$}}

Let $a,a' \in P$, $b \in T_i$, and $b' \in T_j$. If $a \neq a'$, then there exists a partition butterfly wing measurement $(r,s)$ that uses exactly one of the vertices $a$ and $a'$, and we may assume without loss of generality that $a=r$. If $e$ is altered, the effective resistance value will come from column \textbf{X}, and for $e'$, the value comes from column \textbf{XI}.

If $a = a'$, we use a tripartite butterfly wing measurement $(r,s)$ in $T_i$ such that $s \in p_b \cup p_{b'}$. Without loss of generality, let $s \in p_b$. If $b' \notin p_b$, then the effective resistance value when $e$ is altered comes from columns \textbf{VII} or \textbf{VIII} depending on whether or not $b = s$, and for $e'$, the value comes from column \textbf{IX} as $a',b' \notin \{ p_r,p_s \}$. If $b' \in p_b$, then there exists a measurement involving $b$ or $b'$ but not both as we cannot have more than one designated node within a partition and there cannot be components of size two in the measurement graph for the outlined measurement strategy. Again, without loss of generality, let $b = s$. If $s \notin p_r$, then the effective resistance value when $e$ is altered comes from column \textbf{VII} from Table \ref{table:k_part-diff-part} as $a \notin p_r$, and for $e'$, the value comes from column \textbf{VIII} as $a' \notin p_r$ and $b' \in p_s, b' \neq s$. If $s \in p_r$, then for $e$, the value comes from column \textbf{X} in Table \ref{table:k_part-same-part}, and for $e'$ the value comes from column \textbf{XI} as $a \notin p_r$ and $b' \in p_r, b' \neq s$ (One has to switch the labels on the endpoints of the edges to read off the values from Table \ref{table:k_part-same-part}).

\noindent\underline{\textit{Case 7: $e$ has both endpoints in $B$ and $e'$ has one endpoint in $T_i$ and another in $T_j$}}

Without loss of generality, assume $b' \in T_j$. Regardless of whether $i=j$ or not, we use a measurement $(r,s)$ within a tripartite butterfly wing measurement in $T_j$ such that $b' \in p_s$ (and thus $r \in T_j \setminus p_{b'}$). This measurement allows us to say that when $e$ is altered, the effective resistance value comes from column \textbf{IX} in Table \ref{table:k_part-diff-part} as $a,b \notin \{p_r,p_s\}$ and for $e'$, the value comes from columns \textbf{I--VIII} depending on which partitions $a',b' , r$ and $s$ fall into within $T_i$ and $T_j$ (with the possibility that $i=j$).

\noindent\underline{\textit{Case 8: $e$ has both endpoints in $B$ and $e'$ has one endpoint in $B$ and another in $T_i$}}

Without loss of generality, let $b' \in p_{b'} \subset T_i$. We use a measurement $(r,s)$ within a tripartite butterfly wing measurement in $T_{i}$ such that $b' \in p_s$ and $r$ is in another partition in $T_{i}$. This means that when $e$ is altered, the effective resistance value comes from column \textbf{IX} in Table \ref{table:k_part-diff-part} as $a,b \notin \{p_r,p_s\}$, and for $e'$, the value comes from column \textbf{VII} if $b' = s$ or column \textbf{VIII} if $b' \neq s$.

\noindent\underline{\textit{Case 9: $e$ and $e'$ have both endpoints in $B$}}

Let $a, a' \in p_a$ and $b,b' \in p_b$. Since $e \neq e'$, either $a \neq a'$ or $b\neq b'$. Up to a relabeling of the endpoints, we may assume that $a \neq a'$. Recall that $p_a$ contains at most one designated node not involved in any measurement. Further, there are no components of size two in the measurement graph associated with the outlined measurement strategy. Therefore, there must exist a measurement $(r,s)$ that involves $a$ or $a'$ but not both. Without loss of generality, let $a$ be this vertex with $a = r$.

Suppose first that $b = b'$. If $s \notin p_r$ and if $b = s$, when $e$ is altered, the effective resistance value comes from column \textbf{I} in Table \ref{table:k_part-diff-part} since $a=r$, and for $e'$, the value comes from column \textbf{III} since $a' \in p_r, a' \neq r$. If $s \notin p_r$ and $b \in p_s$ with $b \neq s$, when $e$ is altered, the value comes from column \textbf{II} since $a =r$, and for $e'$, the value comes from column \textbf{IV} since $a' \in p_r, a' \neq r$. If instead $s \in p_r$, the effective resistance value when $e$ is altered comes from column \textbf{X} in Table \ref{table:k_part-same-part} as $a=r$, and for $e'$, the value comes from column \textbf{XI} as $a' \in p_r, a' \neq r$ by the choice of $(r,s)$.

Suppose now that $b  \neq b'$. If $s \in p_r$, then $(r,s)$ is a partition butterfly wing measurement, in which case for $e$, the effective resistance value comes from column \textbf{X} in Table \ref{table:k_part-same-part}, and for $e'$, the value comes from column \textbf{XI} as $a' \in p_r, a' \neq r$ by the choice of $(r,s)$, and $b \notin p_r$. Now, suppose $s \notin p_r$. If $s \neq b'$, then the effective resistance value when $e$ is altered comes from column \textbf{I} or column \textbf{II} in Table \ref{table:k_part-diff-part} depending on whether $b = s$ or not, and for $e'$, the value comes from column \textbf{IV} as $a' \in p_r, a' \neq r$ and $b' \in p_s, b'\neq s$. Finally, consider the situation that $s = b'$. Since there are no components of size two in the measurement graph associated with the outlined measurement strategy, there must exist another measurement $(\bar r, \bar s)$ with either $\bar r = a$ or $\bar s = b'$, but not both. Then, up to relabeling the endpoints of the edges, we can apply one of the previous arguments within this paragraph.

\noindent\underline{\textit{Case 10: $e$ has one endpoint in $B$ and another in $T_i$ and $e'$ has one endpoint in $T_j$ and another in $T_k$}}

Let $a \in p_a \subseteq B$. We can distinguish the edges by using any measurement $(r,s)$ that is a measurement from a zig-zagging butterfly wing scheme or a hairpin butterfly wing measurement and that uses a vertex in $p_a$. As $b \notin p_s$, when $e$ is altered, the effective resistance value comes from column \textbf{V} or column \textbf{VI} depending on whether $a = r$ or not, and for $e'$, the value comes from column \textbf{IX} since $a', b' \notin \{ p_r,p_s \}$.

\noindent\underline{\textit{Case 11: $e$ has one endpoint in $B$ and another in $T_i$ and $e'$ has one endpoint in $B$ and another in $T_j$}}

Let $a \in p_a \subseteq B$, $a' \in p_{a'} \subseteq B$ and consider the following three subcases. 

If $a = a'$, the same argument as that of Case 7 when $a=a'$ can be used to show that a tripartite butterfly wing measurement $(r,s)$ in $T_i$ such that $s \in p_b \cup p_{b'}$ exists and will distinguish $e$ and $e'$.

Suppose now that $a' \in p_a$, but $a \neq a'$. Recall that there is at most one designated node in any partition that is not involved in a measurement in the strategy, and there are no components of size two in the measurement graph associated with the strategy. Thus, there must exist a measurement $(r,s)$ in $B$ that has $a$ or $a'$, but not both, as an endpoint; without loss of generality, let $a = r$. If $s \notin p_r$, this means that when $e$ is altered, the effective resistance value comes from column \textbf{V} as $a=r$ and $b \notin p_s$, and for $e'$, the value comes from column \textbf{VI} since $a' \in p_r, a' \neq r$ and $b' \notin p_s$. If $s \in p_r$, then for $e$, the value comes from column \textbf{X} as $b \notin \{p_r,p_s\}$, and for $e'$, the value comes from \textbf{XI} since $a' \in p_r, a' \neq r$ and $b' \notin \{p_r,p_s\}$.

Finally, if $a' \notin p_a$, using the same logic, there must exist a measurement that involves exactly one of $a$ or $a'$ as an endpoint; let $a$ be involved in measurement $(r,s)$ with $a = r$. If $s \notin p_r$, this means that $e$ will have value from column \textbf{V} as $a=r$ and $b \notin p_s$, and for $e'$ the value comes from column \textbf{VI} since $a' \in p_r, a' \neq r$ and $b' \notin p_s$ (after switching the labels on $r,s$). If $s \in p_r$, then for $e$, the value comes from column \textbf{X}, and for $e'$, the value comes from \textbf{XII} since $a', b' \notin p_r$.

\section{Conclusion and future directions}\label{sec:future avenues} In this paper, we introduce a new combinatorial optimization problem (Definition~\ref{def:edge-detection}) motivated by a fundamental question in fault diagnosis in electrical systems. We believe the problem is mathematically rich and the insights are widely applicable in diverse areas. There are two broad future research directions to pursue:

\noindent{\bf Structural aspects}. In this paper, we focus on giving tight lower and upper bounds on the smallest number of measurements in complete graphs and complete $k$-partite graphs. It would be great to extend this analysis to other important families of graphs, e.g.~grid/lattice graphs (in general $d$-dimensions), wheels, polyhedral graphs for structured polyhedra such as hypercubes/fullerene graphs, and others. Beyond tight bounds on the smallest number of measurements, one can also give other insights on the structure of the optimal set of measurements for structured graph families, or relate the optimal number of measurements to other well-studied combinatorial properties such as edge covers, matchings etc., or spectral properties of the Laplacian or adjacency matrices.

\noindent{\bf Algorithmic aspects}. The design of efficient algorithms for solving the faulty edge detection problem has never been explored in the literature, to the best of our knowledge. For example, the following question is open:

\begin{quote}
\textit{For what class of graphs can the faulty edge detection problem be solved in polynomial time?}
\end{quote}

It is not hard to show that the problem can be modeled as a set covering or an integer programming problem. The polyhedral combinatorics of this special set cover/integer programming problem is also worth studying, e.g.,~deriving classes of valid inequalities for the feasible solutions. This will have computational implications in terms of speeding up integer programming based algorithms for the faulty edge detection problem. 


The authors gratefully acknowledge the support from the Air Force Office of Scientific Research (AFOSR) grant {FA9550-25-1-0038}. The authors are also very grateful to Prof.~Daniel Naiman and Prof.~Fadil Santosa at Johns Hopkins University for suggesting the fault detection problem in  Definition~\ref{def:edge-detection} in the first place, and for very helpful discussions around the problem.
%

\bibliographystyle{plain}
\bibliography{full-bib}

\appendix{}
\section[pfx={Appendix\space}]{Effective resistance values for complete graphs}\label{incident_edge_lemma_complete}

Using the effective resistance formulas in Theorem~\ref{thm:eff-res-change} and the inverse expressions for complete graphs in Theorem~\ref{thm:laplacian_complete}, we can compute the effective resistance values when different edges are altered. This allows us to characterize the various equivalence classes of indistinguishable edges for any given measurement, which in turn helps in establishing the bounds on our strategies in  Section~\ref{sec:proofs}.

\begin{lemma}\label{lem:incident_edge_lemma_complete}
Given a measurement $(r,s)$ on a complete graph on $n$ vertices and an altered edge $(a,b)$, we have $R'_{rs} = R_{rs} - \Delta$, where $R_{rs}$ is the effective resistance with no alterations, $R'_{rs}$ is the new effective resistance between $r,s$, and $\Delta$ is given by Table \ref{table:complete_graphs}, which considers both the cases where the resistance on $(a,b)$ goes to 0 and to $+\infty$. See Figure~\ref{fig:cases_for_complete}.
\end{lemma}

\begin{minipage}{\textwidth}
    \centering
    \renewcommand{\arraystretch}{1.5}
    \begin{tabular}{|p{1.8cm}|p{2.5cm}|p{2.5cm}|p{2.5cm}|p{2.5cm}|}
        \hline
        \multicolumn{5}{|c|}{Change $\Delta$ in effective resistance when edge $(a,b)$ is altered} \\
        \hline
         \centering & $a = r$ \newline $b \neq s$ & $a \neq r$ \newline $b = s$ & $a \neq r,s$ \newline $b \neq r,s$ & $a = r$ \newline $b = s$ \\
         \hline
         \centering $r_{ab} \rightarrow 0$ &  $\frac{1}{2n}$ & $\frac{1}{2n}$ & $0$ & $\frac{2}{n}$\\
         \hline
         \centering $r_{ab} \rightarrow +\infty$ & $\frac{-1}{n(n-2)}$ & $\frac{-1}{n(n-2)}$ & $0$ & $\frac{-4}{n(n-2)}$ \\
         \hline
    \end{tabular}
    \captionof{table}{Different cases for the effective resistance changes for complete graphs }\label{table:complete_graphs}
\end{minipage}




\begin{figure}[htbp]
    \centering
        \begin{tikzpicture}[x=0.75pt,y=0.75pt,yscale=-1,xscale=1]
    
    \draw  [fill={rgb, 255:red, 0; green, 0; blue, 0 }  ,fill opacity=1 ] (80,114.25) .. controls (80,111.63) and (82.13,109.5) .. (84.75,109.5) .. controls (87.37,109.5) and (89.5,111.63) .. (89.5,114.25) .. controls (89.5,116.87) and (87.37,119) .. (84.75,119) .. controls (82.13,119) and (80,116.87) .. (80,114.25) -- cycle ;
    \draw  [fill={rgb, 255:red, 0; green, 0; blue, 0 }  ,fill opacity=1 ] (115.5,140.75) .. controls (115.5,138.13) and (117.63,136) .. (120.25,136) .. controls (122.87,136) and (125,138.13) .. (125,140.75) .. controls (125,143.37) and (122.87,145.5) .. (120.25,145.5) .. controls (117.63,145.5) and (115.5,143.37) .. (115.5,140.75) -- cycle ;
    \draw  [fill={rgb, 255:red, 0; green, 0; blue, 0 }  ,fill opacity=1 ] (150.5,75.75) .. controls (150.5,73.13) and (152.63,71) .. (155.25,71) .. controls (157.87,71) and (160,73.13) .. (160,75.75) .. controls (160,78.37) and (157.87,80.5) .. (155.25,80.5) .. controls (152.63,80.5) and (150.5,78.37) .. (150.5,75.75) -- cycle ;
    \draw  [fill={rgb, 255:red, 0; green, 0; blue, 0 }  ,fill opacity=1 ] (150,114.25) .. controls (150,111.63) and (152.13,109.5) .. (154.75,109.5) .. controls (157.37,109.5) and (159.5,111.63) .. (159.5,114.25) .. controls (159.5,116.87) and (157.37,119) .. (154.75,119) .. controls (152.13,119) and (150,116.87) .. (150,114.25) -- cycle ;
    \draw  [dash pattern={on 4.5pt off 4.5pt}]  (85.75,75.75) -- (155.25,114.25) ;
    \draw  [dash pattern={on 4.5pt off 4.5pt}]  (85.75,75.75) -- (155.25,75.75) ;
    \draw    (85.25,75.75) -- (86.24,75.01) -- (119.75,49.75) ;
    \draw  [dash pattern={on 4.5pt off 4.5pt}]  (85.75,75.75) -- (120.25,140.75) ;
    \draw  [dash pattern={on 4.5pt off 4.5pt}]  (85.75,75.75) -- (85.25,114.25) ;
    \draw  [dash pattern={on 4.5pt off 4.5pt}]  (119.75,49.75) -- (154.75,75.75) ;
    \draw  [dash pattern={on 4.5pt off 4.5pt}]  (120.25,49.75) -- (155.25,114.25) ;
    \draw  [dash pattern={on 4.5pt off 4.5pt}]  (120.25,49.75) -- (120.25,140.75) ;
    \draw  [dash pattern={on 4.5pt off 4.5pt}]  (120.25,49.75) -- (85.25,114.25) ;
    \draw  [dash pattern={on 0.75pt off 0.75pt}]  (155.25,75.75) .. controls (156.92,77.42) and (156.92,79.08) .. (155.25,80.75) .. controls (153.58,82.42) and (153.58,84.08) .. (155.25,85.75) .. controls (156.92,87.42) and (156.92,89.08) .. (155.25,90.75) .. controls (153.58,92.42) and (153.58,94.08) .. (155.25,95.75) .. controls (156.92,97.42) and (156.92,99.08) .. (155.25,100.75) .. controls (153.58,102.42) and (153.58,104.08) .. (155.25,105.75) .. controls (156.92,107.42) and (156.92,109.08) .. (155.25,110.75) -- (155.25,114.25) -- (155.25,114.25) ;
    \draw  [dash pattern={on 0.75pt off 0.75pt}]  (154.75,75.75) .. controls (155.43,78.01) and (154.64,79.48) .. (152.38,80.15) .. controls (150.12,80.82) and (149.33,82.29) .. (150.01,84.55) .. controls (150.69,86.81) and (149.9,88.28) .. (147.64,88.96) .. controls (145.38,89.63) and (144.59,91.1) .. (145.27,93.36) .. controls (145.95,95.62) and (145.16,97.09) .. (142.9,97.76) .. controls (140.64,98.43) and (139.85,99.9) .. (140.53,102.16) .. controls (141.21,104.42) and (140.42,105.89) .. (138.16,106.57) .. controls (135.9,107.24) and (135.11,108.71) .. (135.79,110.97) .. controls (136.47,113.23) and (135.68,114.7) .. (133.42,115.37) .. controls (131.16,116.04) and (130.37,117.51) .. (131.05,119.77) .. controls (131.72,122.03) and (130.93,123.5) .. (128.67,124.18) .. controls (126.41,124.85) and (125.62,126.32) .. (126.3,128.58) .. controls (126.98,130.84) and (126.19,132.31) .. (123.93,132.98) .. controls (121.67,133.65) and (120.88,135.12) .. (121.56,137.38) -- (119.75,140.75) -- (119.75,140.75) ;
    \draw  [dash pattern={on 0.75pt off 0.75pt}]  (155.25,75.75) .. controls (154.6,78.02) and (153.14,78.82) .. (150.87,78.16) .. controls (148.6,77.5) and (147.14,78.3) .. (146.49,80.57) .. controls (145.84,82.84) and (144.38,83.64) .. (142.11,82.98) .. controls (139.84,82.32) and (138.38,83.12) .. (137.73,85.39) .. controls (137.07,87.66) and (135.61,88.46) .. (133.34,87.8) .. controls (131.07,87.14) and (129.61,87.94) .. (128.96,90.21) .. controls (128.31,92.48) and (126.85,93.28) .. (124.58,92.62) .. controls (122.31,91.96) and (120.85,92.76) .. (120.2,95.03) .. controls (119.55,97.3) and (118.09,98.1) .. (115.82,97.44) .. controls (113.55,96.78) and (112.09,97.58) .. (111.44,99.85) .. controls (110.79,102.12) and (109.33,102.92) .. (107.06,102.26) .. controls (104.79,101.6) and (103.33,102.4) .. (102.68,104.67) .. controls (102.02,106.93) and (100.56,107.73) .. (98.3,107.07) .. controls (96.03,106.41) and (94.57,107.21) .. (93.91,109.48) .. controls (93.26,111.75) and (91.8,112.55) .. (89.53,111.89) -- (85.25,114.25) -- (85.25,114.25) ;
    \draw  [dash pattern={on 0.75pt off 0.75pt}]  (154.75,114.25) .. controls (154.42,116.58) and (153.09,117.59) .. (150.76,117.27) .. controls (148.43,116.95) and (147.1,117.96) .. (146.78,120.29) .. controls (146.45,122.62) and (145.12,123.63) .. (142.79,123.3) .. controls (140.46,122.98) and (139.13,123.99) .. (138.8,126.32) .. controls (138.48,128.65) and (137.15,129.66) .. (134.82,129.34) .. controls (132.49,129.02) and (131.16,130.03) .. (130.83,132.36) .. controls (130.51,134.69) and (129.18,135.7) .. (126.85,135.38) .. controls (124.52,135.06) and (123.19,136.07) .. (122.86,138.4) -- (119.75,140.75) -- (119.75,140.75) ;
    \draw  [dash pattern={on 0.75pt off 0.75pt}]  (155.25,114.25) .. controls (153.58,115.92) and (151.92,115.92) .. (150.25,114.25) .. controls (148.58,112.58) and (146.92,112.58) .. (145.25,114.25) .. controls (143.58,115.92) and (141.92,115.92) .. (140.25,114.25) .. controls (138.58,112.58) and (136.92,112.58) .. (135.25,114.25) .. controls (133.58,115.92) and (131.92,115.92) .. (130.25,114.25) .. controls (128.58,112.58) and (126.92,112.58) .. (125.25,114.25) .. controls (123.58,115.92) and (121.92,115.92) .. (120.25,114.25) .. controls (118.58,112.58) and (116.92,112.58) .. (115.25,114.25) .. controls (113.58,115.92) and (111.92,115.92) .. (110.25,114.25) .. controls (108.58,112.58) and (106.92,112.58) .. (105.25,114.25) .. controls (103.58,115.92) and (101.92,115.92) .. (100.25,114.25) .. controls (98.58,112.58) and (96.92,112.58) .. (95.25,114.25) .. controls (93.58,115.92) and (91.92,115.92) .. (90.25,114.25) .. controls (88.58,112.58) and (86.92,112.58) .. (85.25,114.25) -- (85.25,114.25) ;
    \draw  [dash pattern={on 0.75pt off 0.75pt}]  (84.75,114.25) .. controls (87.08,113.93) and (88.41,114.94) .. (88.74,117.27) .. controls (89.06,119.6) and (90.39,120.61) .. (92.72,120.29) .. controls (95.05,119.96) and (96.38,120.97) .. (96.71,123.3) .. controls (97.04,125.63) and (98.37,126.64) .. (100.7,126.32) .. controls (103.03,126) and (104.36,127.01) .. (104.68,129.34) .. controls (105.01,131.67) and (106.34,132.68) .. (108.67,132.36) .. controls (111,132.04) and (112.33,133.05) .. (112.65,135.38) .. controls (112.98,137.71) and (114.31,138.72) .. (116.64,138.4) -- (119.75,140.75) -- (119.75,140.75) ;
    \draw  [fill={rgb, 255:red, 255; green, 255; blue, 255 }  ,fill opacity=1 ] (115.5,49.75) .. controls (115.5,47.13) and (117.63,45) .. (120.25,45) .. controls (122.87,45) and (125,47.13) .. (125,49.75) .. controls (125,52.37) and (122.87,54.5) .. (120.25,54.5) .. controls (117.63,54.5) and (115.5,52.37) .. (115.5,49.75) -- cycle ;
    \draw  [fill={rgb, 255:red, 255; green, 255; blue, 255 }  ,fill opacity=1 ] (81,75.75) .. controls (81,73.13) and (83.13,71) .. (85.75,71) .. controls (88.37,71) and (90.5,73.13) .. (90.5,75.75) .. controls (90.5,78.37) and (88.37,80.5) .. (85.75,80.5) .. controls (83.13,80.5) and (81,78.37) .. (81,75.75) -- cycle ;
    \draw    (309.7,65) -- (340.2,65) ;
    \draw  [dash pattern={on 4.5pt off 4.5pt}]  (310.7,86) -- (341.2,86) ;
    \draw  [dash pattern={on 4.5pt off 4.5pt}]  (310.7,107) -- (341.2,107) ;
    \draw  [dash pattern={on 0.75pt off 0.75pt}]  (310.7,128) .. controls (312.37,126.33) and (314.03,126.33) .. (315.7,128) .. controls (317.37,129.67) and (319.03,129.67) .. (320.7,128) .. controls (322.37,126.33) and (324.03,126.33) .. (325.7,128) .. controls (327.37,129.67) and (329.03,129.67) .. (330.7,128) .. controls (332.37,126.33) and (334.03,126.33) .. (335.7,128) .. controls (337.37,129.67) and (339.03,129.67) .. (340.7,128) -- (341.2,128) -- (341.2,128) ;
    
    \draw (111.5,149.5) node [anchor=north west][inner sep=0.75pt]   [align=left] {$\displaystyle G$};
    \draw (73,63) node [anchor=north west][inner sep=0.75pt]    {$r$};
    \draw (110,34) node [anchor=north west][inner sep=0.75pt]    {$s$};
    \draw (201.5,56.1) node [anchor=north west][inner sep=0.75pt]   [align=left] {$\displaystyle  \begin{array}{{>{\displaystyle}l}}
    a=r,\ b=s\\
    a=r,\ b\neq s\\
    a\neq r,\ b=s\\
    a\neq r,\ b\neq s\ 
    \end{array}$};

    \end{tikzpicture}
    \caption{Depiction of cases on a complete graph $G$ when measurement $r,s$ is performed. The style of each edge represents the case from Table~\ref{table:complete_graphs} in which the corresponding edge $(a,b)$ is removed from the graph.}
    \label{fig:cases_for_complete}
\end{figure}

\begin{proof}
Theorem~\ref{thm:laplacian_complete} is used repeatedly for all the calculations below involving the entries of the inverse of the reduced Laplacian.

First, we note that $R_{rs} = \frac{2}{n}$ by Equation~\ref{eq:original-res}. For $R'_{rs}$, the value of $\beta$ in Theorem~\ref{thm:eff-res-change} is different for when $r_{ab} \rightarrow 0$ (i.e.,~the edge is shorted and $w_{ab} \rightarrow +\infty$) and when $r_{ab} \rightarrow +\infty$ (i.e.,~the edge is removed and $w_{ab} \rightarrow 0$). Therefore, we make note of what these two values will be for our substitutions.  When the resistance value $r_{ab}$ of the altered edge approaches 0, the change $\alpha$ in weight approaches $+\infty$ and so $\beta = \frac{1}{L(b)_{aa}^{-1}} = \frac{1}{\frac{2}{n}} = \frac{n}{2}$. When the resistance value $r_{ab}$ approaches $+\infty$, the change in weight is -1 and so $\beta = \frac{-1}{1-L(b)_{aa}^{-1}} = \frac{-1}{1-\frac{2}{n}} = \frac{-n}{n-2}$.

When $a \neq r$ and $b = s$, the change in resistance is $\Delta = \beta \left( L(b)^{-1}_{ar} \right)^2$ from  Equation~\ref{eq:new-res}. Substituting our given values, we obtain $\Delta = \left( \frac{n}{2} \right) \left( \frac{1}{n} \right)^2 = \frac{1}{2n}$ when the resistance $r_{ab}$ goes to 0 and
$\Delta = \left( \frac{-n}{n-2} \right) \left( \frac{1}{n}\right)^2 = \frac{-1}{n(n-2)}$ when the resistance $r_{ab}$ goes to $+\infty$. The situation $a = r$ and $b \neq s$ is symmetric to when $a \neq r$ and $b = s$ up to a permutation of vertex labels and so the changes must be the same as that case. Alternatively, one can use Equation~\ref{eq:new-res} again with $\Delta = \beta \left( L(b)^{-1}_{aa} -  L(b)^{-1}_{as} \right)^2$ since $a=r$ which gives the same values of $\frac{1}{2n}$ and $\frac{-1}{n(n-2)}$.

When $a \neq r,s$ and $b \neq r,s$, the change in resistance is $\Delta = \beta \left( L(b)^{-1}_{ar} -  L(b)^{-1}_{as} \right)^2$ from  Equation~\ref{eq:new-res}. Therefore, we obtain
$\Delta = \left( \frac{n}{2} \right) \left( \frac{1}{n} - \frac{1}{n} \right)^2 = 0$ when the resistance $r_{ab}$ goes to 0 and 
$\Delta = \left( \frac{-n}{n-2} \right) \left( \frac{1}{n} - \frac{1}{n} \right)^2 = 0$ when the resistance goes to $+\infty$.

When $a = r$ and $b = s$, the change in resistance is $\Delta = \beta \left(L(b)^{-1}_{ar} \right)^2 = \beta \left(L(b)^{-1}_{aa} \right)^2$ from  Equation~\ref{eq:new-res}. Therefore, we obtain the following $\Delta = \left( \frac{n}{2} \right) \left( \frac{2}{n} \right)^2 = \frac{2}{n}$ when the resistance $r_{ab}$ goes to 0 and
$\Delta = \left( \frac{-n}{n-2} \right) \left( \frac{2}{n}\right)^2 = \frac{-4}{n(n-2)}$ when the resistance $r_{ab}$ goes to $+\infty$.
\end{proof}

\section[pfx={Appendix\space}]{Effective resistance values for complete k-partite graphs}\label{A:edge_partition_k-partite}

We can compute the effective resistance values when different edges are altered by substituting terms from the inverse expressions for complete $k$-partite graphs in Theorem~\ref{thm:laplacian_kpartite} into the effective resistance formulas in Theorem~\ref{thm:eff-res-change}. We are then able to characterize the various equivalence classes of indistinguishable edges for any given measurement, enabling us to establish the bounds on our strategies in  Section~\ref{sec:proofs}.

\begin{lemma}\label{incdient_edge_lemma_for_kpartite}
Given a measurement $(r,s)$ on a complete $k$-partite graph, suppose we alter the resistance value on edge $(a,b)$. The new effective resistance will be $R'_{rs} = R_{rs} - \Delta$, where $R_{rs}$ is the effective resistance with no alterations and $\Delta$ is a value obtained from Table \ref{table:k_part-diff-part} or Table \ref{table:k_part-same-part}, depending on which partitions contain $r,s,a, \text{and } b$. Table \ref{table:k_part-same-part} shows the values when the measurement endpoints $r \text{ and } s$ are in the same partition and Table \ref{table:k_part-diff-part} shows the values when they are not. Each table displays the different possible cases according to the partitions that contain $a \text{ and } b$, and the two cases of having the altered resistance values on the edge $(a,b)$ go to 0 or to $+\infty$ (indicated in the first column of the table). See Figures~\ref{fig:cases_for_kpart_diff_1to4}, \ref{fig:cases_for_kpart_diff_5to8}, \ref{fig:fig:cases_for_kpart_same}. 

We use the following notation to help aid in displaying the expressions: for every $q \in \{1,\ldots,k\}$, define
\begin{equation*}
C_{p_q} = \frac{| \widetilde{p_b}| n + \displaystyle\sum_{i \notin \{b,q\} }{| p_i| (n-1)}}{n(n-| p_q| )(n-| p_b| )} = 
\frac{(n-1)^2+| \widetilde{p_b}|  - | p_q| (n-1)}{n(n-| p_q| )(n-| p_b| )}.
\end{equation*}
\end{lemma}

\begin{proof}
For all the calculations below, the entries of the inverse of the reduced Laplacian are used from Theorem~\ref{thm:laplacian_kpartite}.

First, we want to show that the values provided in Tables~\ref{table:k_part-diff-part} and \ref{table:k_part-same-part} are valid. We note that, regardless of what $b$ is selected, $R_{rs}$ will remain constant for given $r,s$. We show how two cases are obtained, one from each table. The other values follow in a similar way by substituting values from Theorem~\ref{thm:laplacian_kpartite} into the expressions provided in Theorem~\ref{thm:eff-res-change}.

Let $ r\in p_r$, $s \in p_s$, and the altered edge has endpoints $a=r$ and $b = s$; thus, we are in column \textbf{I} of Table \ref{table:k_part-diff-part}. To get the value when $r_{ab}$ goes to 0, we use Equation~\ref{eq:new-res} with $\beta = \frac{1}{L(b)_{rr}^{-1}}$. The value we obtain when plugging in values from the block matrix form will be $\Delta = \beta\left( L(b)_{rr}^{-1}\right)^2 =  L(b)_{rr}^{-1} = C_{p_r} + \frac{1}{n-| p_r| }$ (As a sanity check, note that in this case, the unaltered effective resistance $R_{rs}$ is also equal to $L(b)_{rr}^{-1}$ by Equation~\ref{eq:original-res}, and so the new effective resistance $R'_{rs} = R_{rs} - \Delta =0$ which is to be expected since we have set the resistance to 0 on the edge joining the measurement endpoints $r,s$). To get the value when $r_{ab}$ goes to $+\infty$, we have that $\beta = \frac{-1}{1 - L(b)_{rr}^{-1}} = \frac{-1}{1 - \left(C_{p_r} + \frac{1}{n - | p_r| } \right)}$. Upon substitution, we obtain $\Delta = \beta\left( L(b)_{rr}^{-1}\right)^2 = \frac{-1}{1-\left(C_{p_r} + \frac{1}{n - | p_r| }\right)} \left( C_{p_r} + \frac{1}{n-| p_r| } \right)^2$.

Let $r,s \in p_r$,  and the altered edge has endpoints $a=r$ and $b \in p_z$ such that $p_z \neq p_s$; thus, we are in column \textbf{X} in Table \ref{table:k_part-same-part}. To get the value when $r_{ab}$ goes to 0, we note that $\beta = \frac{1}{L(b)_{rr}^{-1}} = \frac{1}{C_{p_r} + \frac{1}{n - | p_r| }}$. Then plugging in values from our block matrix form into Equation~\ref{eq:new-res}, we obtain $\Delta = \beta\left( L(b)_{rr}^{-1} - L(b)_{rs}^{-1} \right)^2 = \frac{1}{C_{p_r} + \frac{1}{n - | p_r| }} \left( C_{p_r} + \frac{1}{n-| p_r| } - C_{p_r} \right)^2 = \frac{1}{C_{p_r}(n-| p_r| )^2 + (n-| p_r| )}$. To get the value when $r_{ab}$ goes to $+\infty$, we have that $\beta = \frac{-1}{1 - L(b)_{rr}^{-1}} = \frac{-1}{1 - \left(C_{p_r} + \frac{1}{n - | p_r| } \right)}$. Therefore, the following value is obtained:
\begin{equation*}
\begin{array}
{rcl}
\Delta = \beta\left( L(b)_{rr}^{-1} - L(b)_{rs}^{-1} \right)^2 & = &  \frac{-1}{1 - \left(C_{p_r} + \frac{1}{n - | p_r| } \right)}\left( C_{p_r} + \frac{1}{n-| p_r| } - C_{p_r} \right)^2 \\
& = & \frac{1}{(C_{p_r} - 1)(n-| p_r| )^2 + (n-| p_r| )}
\end{array}
\end{equation*}

\end{proof}

\centering
    
    \renewcommand{\arraystretch}{1.5}
    \begin{tabular}{|p{1.7cm}|p{5.3cm}|p{6.5cm}|}
        \hline
        \multicolumn{3}{|c|}{$\Delta \text{ with } r \in p_r, s \in p_s$ and $p_r \neq p_s$} \\
        \hline
         & \textbf{I. } \hspace{0.2cm} $a \in p_r, a=r$ \hspace{0.2cm} $b \in p_s, b=s$ & \textbf{II. } \hspace{0.2cm} $a \in p_r, a=r$ \hspace{0.2cm} $b \in p_s, b \neq s$ \\
         \hline
        $r_{ab} \rightarrow 0 $ & $ C_{p_r} + \frac{1}{n-|p_r|}$& $\left( \frac{1}{C_{p_r} + \frac{1}{n - |p_r|}} \right)\left( C_{p_r} + \frac{1}{n-|p_r|} - \frac{1}{n-|p_s|} \right)^2$\\
        \hline
        $r_{ab} \rightarrow +\infty$ & $\frac{-1}{1-\left(C_{p_r} + \frac{1}{n - |p_r|}\right)} \left( C_{p_r} + \frac{1}{n-|p_r|} \right)^2$& $\left( \frac{-1}{1 - \left(C_{p_r} + \frac{1}{n - |p_r|}\right)} \right)\left( C_{p_r} + \frac{1}{n-|p_r|} - \frac{1}{n-|p_s|} \right)^2$\\
        \hline
    \end{tabular}
    
    \vspace{-0.05cm}
    
    \begin{tabular}{|p{1.7cm}|p{5.3cm}|p{6.5cm}|}
        \hline
         & \textbf{III. } $a \in p_r, a \neq r$ \hspace{0.2cm} $b \in p_s, b=s$ & \textbf{IV. } \hspace{0.2cm} $a \in p_r, a \neq r$ \hspace{0.2cm} $b \in p_s, b \neq s$ \\
         \hline
        $r_{ab} \rightarrow 0$ & $\frac{C_{p_r}^2}{C_{p_r} + \frac{1}{n-|p_r|}}$ & $\left( \frac{1}{C_{p_r} + \frac{1}{n-|p_r|}} \right)\left( C_{p_r} - \frac{1}{n - |p_s|} \right)^2 $ \\
        \hline
        $r_{ab} \rightarrow +\infty$ & $\frac{-C_{p_r}^2}{1-\left(C_{p_r} + \frac{1}{n-|p_r|}\right)}$ & $\left( \frac{-1}{1 - \left(C_{p_r} + \frac{1}{n-|p_r|}\right)} \right)\left( C_{p_r} - \frac{1}{n - |p_s|} \right)^2 $\\
       \hline
    \end{tabular}
    
    \vspace{-0.05cm}
    
    \begin{tabular}{|p{1.7cm}|p{12.25cm}|}
        \hline
         & \textbf{V. } \hspace{0.2cm} $a \in p_r, a=r$ \hspace{0.2cm} $ b \not\in p_s$ ($b \in p_z$)\\
         \hline
        $r_{ab} \rightarrow 0$ & $\left( \frac{1}{C_{p_r} + \frac{1}{n - |p_r|}} \right)\left( C_{p_r} + \frac{1}{n-|p_r|} - \frac{n-1}{(n-|p_z|)n} \right)^2$\\
        \hline
        $r_{ab} \rightarrow +\infty$ & $\left( \frac{-1}{1 - \left(C_{p_r} + \frac{1}{n - |p_r|}\right)} \right)\left( C_{p_r} + \frac{1}{n-|p_r|} - \frac{n-1}{(n-|p_z|)n} \right)^2$\\
        \hline
    \end{tabular}
    
    \vspace{-0.05cm}
    
    \begin{tabular}{|p{1.7cm}|p{6.0cm}|p{5.8cm}|}
        \hline
         & \textbf{VI. }\hspace{0.2cm} $a \in p_r, a \neq r$ \hspace{0.2cm} $b \not\in p_s$ ($b \in p_z$) & \textbf{VII. } \hspace{0.2cm} $a \not\in p_r$ ($a \in p_z$)\hspace{0.2cm} $b \in p_s, b = s$\\
         \hline
        $r_{ab} \rightarrow 0$ & $\left( \frac{1}{C_{p_r} + \frac{1}{n-|p_r|}} \right)\left( C_{p_r} - \frac{n-1}{(n - |p_z|)n} \right)^2 $ &  $\left( \frac{1}{C_{p_z} + \frac{1}{n- |p_z|}} \right)  \left( \frac{n-1}{n(n-|p_s|)} \right)^2$\\
        \hline
        $r_{ab} \rightarrow +\infty$ & $\left( \frac{-1}{1 - \left(C_{p_r} + \frac{1}{n-|p_r|}\right)} \right)\left( C_{p_r} - \frac{n-1}{(n - |p_z|)n} \right)^2$ & $\left( \frac{-1}{1 - \left(C_{p_z} + \frac{1}{n-|p_z|} \right)} \right)  \left( \frac{n-1}{n(n-|p_s|)} \right)^2$\\
        \hline
    \end{tabular}
    
    \vspace{-0.05cm}
    
    \begin{tabular}{|p{1.7cm}|p{6.0cm}|p{5.8cm}|}
         & \textbf{VIII. } \hspace{0.2cm} $a \not\in p_r$ ($a \in p_z$)\hspace{0.2cm} $b \in p_s, b \neq s$ & \textbf{IX. } \hspace{0.2cm} $a,b \not\in p_r \cup p_s$\\
         \hline
        $r_{ab} \rightarrow 0$ & $\left( \frac{1}{C_{p_z} + \frac{1}{n - |p_z|}} \right) \left( \frac{-1}{n(n-|p_s|)} \right)^2$ & $0$\\
        \hline
        $r_{ab} \rightarrow +\infty$ & $\left( \frac{-1}{1-\left(C_{p_z} + \frac{1}{n - |p_z|}\right)} \right) \left( \frac{-1}{n(n-|p_s|)} \right)^2$ & $0$\\
        \hline
    \end{tabular}
    \captionof{table}{Cases for effective resistance changes when measurement end points are in different partitions.}\label{table:k_part-diff-part}

    \newpage
    
        \begin{tabular}{|p{1.7cm}|p{6cm}|p{2.7cm}|p{2.7cm}|}
        \hline
        \multicolumn{4}{|c|}{$\Delta$ with $r,s \in p_r$} \\
        \hline
         \centering & \textbf{X.} \newline $a \in p_r, a = r$ \hspace{0.1cm} or \hspace{0.1cm} $a \in p_r, a = s$ \newline $b \not\in p_r$ & \textbf{XI.} \newline $a \in p_r, a \neq r,s$ \newline $b \not\in p_r$ & \textbf{XII.} \newline $a ,b \not\in p_r$ \\
         \hline
         \centering $r_{ab} \rightarrow 0$ &  $\frac{1}{C_{p_r}(n-|p_r|)^2 + (n-|p_r|)}$ & $0$ & $0$\\
         \hline
         \rule{0pt}{3ex}\centering $r_{ab} \rightarrow +\infty$ & $\frac{1}{(C_{p_r}-1)(n-|p_r|)^2 + (n-|p_r|)}$ & $0$ & $0$ \\
         \hline
        \end{tabular}
    \captionof{table}{Cases for effective resistance changes when measurement end points are in same partition. }\label{table:k_part-same-part}

\begin{figure}[htbp]
    \centering   

    \begin{tikzpicture}[x=0.75pt,y=0.75pt,yscale=-1,xscale=1]
    
    \draw  [color={rgb, 255:red, 0; green, 0; blue, 0 }  ,draw opacity=1 ] (89.5,76.95) .. controls (89.5,49.64) and (98.45,27.5) .. (109.5,27.5) .. controls (120.55,27.5) and (129.5,49.64) .. (129.5,76.95) .. controls (129.5,104.26) and (120.55,126.4) .. (109.5,126.4) .. controls (98.45,126.4) and (89.5,104.26) .. (89.5,76.95) -- cycle ;
    \draw   (260,83.52) .. controls (260,53.04) and (268.95,28.34) .. (280,28.34) .. controls (291.05,28.34) and (300,53.04) .. (300,83.52) .. controls (300,113.99) and (291.05,138.7) .. (280,138.7) .. controls (268.95,138.7) and (260,113.99) .. (260,83.52) -- cycle ;
    \draw [color={rgb, 255:red, 0; green, 0; blue, 0 }  ,draw opacity=1 ][line width=1.5]    (110.25,55.25) -- (280.25,55.25) ;
    \draw [color={rgb, 255:red, 0; green, 0; blue, 0 }  ,draw opacity=1 ]   (110.25,76.25) .. controls (111.7,74.39) and (113.35,74.19) .. (115.21,75.64) .. controls (117.07,77.09) and (118.72,76.88) .. (120.17,75.02) .. controls (121.62,73.16) and (123.28,72.96) .. (125.14,74.41) .. controls (127,75.86) and (128.65,75.65) .. (130.1,73.79) .. controls (131.55,71.93) and (133.2,71.73) .. (135.06,73.18) .. controls (136.92,74.63) and (138.57,74.42) .. (140.02,72.56) .. controls (141.47,70.7) and (143.12,70.5) .. (144.98,71.95) .. controls (146.84,73.4) and (148.5,73.19) .. (149.95,71.33) .. controls (151.4,69.47) and (153.05,69.27) .. (154.91,70.72) .. controls (156.77,72.17) and (158.42,71.96) .. (159.87,70.1) .. controls (161.32,68.24) and (162.97,68.04) .. (164.83,69.49) .. controls (166.69,70.94) and (168.34,70.73) .. (169.79,68.87) .. controls (171.24,67.01) and (172.9,66.81) .. (174.76,68.26) .. controls (176.62,69.71) and (178.27,69.5) .. (179.72,67.64) .. controls (181.17,65.78) and (182.82,65.58) .. (184.68,67.03) .. controls (186.54,68.48) and (188.19,68.27) .. (189.64,66.41) .. controls (191.09,64.55) and (192.75,64.35) .. (194.61,65.8) .. controls (196.47,67.25) and (198.12,67.04) .. (199.57,65.18) .. controls (201.02,63.32) and (202.67,63.12) .. (204.53,64.57) .. controls (206.39,66.02) and (208.04,65.81) .. (209.49,63.95) .. controls (210.94,62.09) and (212.59,61.89) .. (214.45,63.34) .. controls (216.31,64.79) and (217.97,64.59) .. (219.42,62.73) .. controls (220.87,60.87) and (222.52,60.66) .. (224.38,62.11) .. controls (226.24,63.56) and (227.89,63.36) .. (229.34,61.5) .. controls (230.79,59.64) and (232.44,59.43) .. (234.3,60.88) .. controls (236.16,62.33) and (237.81,62.13) .. (239.26,60.27) .. controls (240.71,58.41) and (242.37,58.2) .. (244.23,59.65) .. controls (246.09,61.1) and (247.74,60.9) .. (249.19,59.04) .. controls (250.64,57.18) and (252.29,56.97) .. (254.15,58.42) .. controls (256.01,59.87) and (257.66,59.67) .. (259.11,57.81) .. controls (260.56,55.95) and (262.21,55.74) .. (264.07,57.19) .. controls (265.93,58.64) and (267.59,58.44) .. (269.04,56.58) .. controls (270.49,54.72) and (272.14,54.51) .. (274,55.96) .. controls (275.86,57.41) and (277.51,57.21) .. (278.96,55.35) -- (279.75,55.25) -- (279.75,55.25) ;
    \draw [color={rgb, 255:red, 0; green, 0; blue, 0 }  ,draw opacity=1 ]   (111.75,96.95) .. controls (112.96,94.93) and (114.58,94.53) .. (116.6,95.75) .. controls (118.62,96.96) and (120.24,96.56) .. (121.46,94.54) .. controls (122.67,92.52) and (124.29,92.12) .. (126.31,93.34) .. controls (128.33,94.55) and (129.95,94.15) .. (131.16,92.13) .. controls (132.37,90.11) and (133.99,89.71) .. (136.01,90.93) .. controls (138.03,92.14) and (139.65,91.74) .. (140.87,89.72) .. controls (142.08,87.7) and (143.7,87.3) .. (145.72,88.52) .. controls (147.74,89.73) and (149.36,89.33) .. (150.57,87.31) .. controls (151.78,85.29) and (153.4,84.89) .. (155.42,86.11) .. controls (157.44,87.32) and (159.06,86.92) .. (160.28,84.9) .. controls (161.49,82.88) and (163.11,82.48) .. (165.13,83.7) .. controls (167.15,84.92) and (168.77,84.52) .. (169.98,82.5) .. controls (171.2,80.48) and (172.82,80.08) .. (174.84,81.29) .. controls (176.86,82.51) and (178.48,82.11) .. (179.69,80.09) .. controls (180.9,78.07) and (182.52,77.67) .. (184.54,78.88) .. controls (186.56,80.1) and (188.18,79.7) .. (189.39,77.68) .. controls (190.61,75.66) and (192.23,75.26) .. (194.25,76.47) .. controls (196.27,77.69) and (197.89,77.29) .. (199.1,75.27) .. controls (200.31,73.25) and (201.93,72.85) .. (203.95,74.06) .. controls (205.97,75.28) and (207.59,74.88) .. (208.8,72.86) .. controls (210.02,70.84) and (211.64,70.44) .. (213.66,71.66) .. controls (215.68,72.87) and (217.3,72.47) .. (218.51,70.45) .. controls (219.72,68.43) and (221.34,68.03) .. (223.36,69.25) .. controls (225.38,70.46) and (227,70.06) .. (228.22,68.04) .. controls (229.43,66.02) and (231.05,65.62) .. (233.07,66.84) .. controls (235.09,68.05) and (236.71,67.65) .. (237.92,65.63) .. controls (239.13,63.61) and (240.75,63.21) .. (242.77,64.43) .. controls (244.79,65.64) and (246.41,65.24) .. (247.63,63.22) -- (249.34,62.8) -- (249.34,62.8) .. controls (250.55,60.78) and (252.17,60.38) .. (254.19,61.59) .. controls (256.21,62.81) and (257.83,62.41) .. (259.04,60.39) .. controls (260.25,58.37) and (261.87,57.97) .. (263.89,59.19) .. controls (265.91,60.4) and (267.53,60) .. (268.75,57.98) .. controls (269.96,55.96) and (271.58,55.56) .. (273.6,56.78) .. controls (275.62,57.99) and (277.24,57.59) .. (278.45,55.57) -- (279.75,55.25) -- (279.75,55.25) ;
    \draw [color={rgb, 255:red, 0; green, 0; blue, 0 }  ,draw opacity=1 ] [dash pattern={on 4.5pt off 4.5pt}]  (110.25,55.25) -- (280.25,76.25) ;
    \draw [color={rgb, 255:red, 0; green, 0; blue, 0 }  ,draw opacity=1 ] [dash pattern={on 4.5pt off 4.5pt}]  (110.25,55.25) -- (280.25,95.25) ;
    \draw [color={rgb, 255:red, 0; green, 0; blue, 0 }  ,draw opacity=1 ] [dash pattern={on 4.5pt off 4.5pt}]  (110.25,55.25) -- (279.25,118.75) ;
    \draw [color={rgb, 255:red, 0; green, 0; blue, 0 }  ,draw opacity=1 ] [dash pattern={on 0.84pt off 2.51pt}]  (109.5,76.95) -- (280.25,75.25) ;
    \draw [color={rgb, 255:red, 0; green, 0; blue, 0 }  ,draw opacity=1 ] [dash pattern={on 0.84pt off 2.51pt}]  (110.25,76.25) -- (280.25,96.25) ;
    \draw [color={rgb, 255:red, 0; green, 0; blue, 0 }  ,draw opacity=1 ] [dash pattern={on 0.84pt off 2.51pt}]  (110.25,76.25) -- (279.75,118.75) ;
    \draw [color={rgb, 255:red, 0; green, 0; blue, 0 }  ,draw opacity=1 ] [dash pattern={on 0.84pt off 2.51pt}]  (110.25,96.25) -- (280.25,76.25) ;
    \draw [color={rgb, 255:red, 0; green, 0; blue, 0 }  ,draw opacity=1 ] [dash pattern={on 0.84pt off 2.51pt}]  (110.25,96.25) -- (280.25,96.25) ;
    \draw [color={rgb, 255:red, 0; green, 0; blue, 0 }  ,draw opacity=1 ] [dash pattern={on 0.84pt off 2.51pt}]  (111.75,96.95) -- (279.75,118.75) ;
    \draw  [fill={rgb, 255:red, 255; green, 255; blue, 255 }  ,fill opacity=1 ] (107,55.25) .. controls (107,53.46) and (108.46,52) .. (110.25,52) .. controls (112.04,52) and (113.5,53.46) .. (113.5,55.25) .. controls (113.5,57.04) and (112.04,58.5) .. (110.25,58.5) .. controls (108.46,58.5) and (107,57.04) .. (107,55.25) -- cycle ;
    \draw  [fill={rgb, 255:red, 0; green, 0; blue, 0 }  ,fill opacity=1 ] (107,96.25) .. controls (107,94.46) and (108.46,93) .. (110.25,93) .. controls (112.04,93) and (113.5,94.46) .. (113.5,96.25) .. controls (113.5,98.04) and (112.04,99.5) .. (110.25,99.5) .. controls (108.46,99.5) and (107,98.04) .. (107,96.25) -- cycle ;
    \draw  [fill={rgb, 255:red, 0; green, 0; blue, 0 }  ,fill opacity=1 ] (107,76.25) .. controls (107,74.46) and (108.46,73) .. (110.25,73) .. controls (112.04,73) and (113.5,74.46) .. (113.5,76.25) .. controls (113.5,78.04) and (112.04,79.5) .. (110.25,79.5) .. controls (108.46,79.5) and (107,78.04) .. (107,76.25) -- cycle ;
    \draw  [fill={rgb, 255:red, 0; green, 0; blue, 0 }  ,fill opacity=1 ] (277,76.25) .. controls (277,74.46) and (278.46,73) .. (280.25,73) .. controls (282.04,73) and (283.5,74.46) .. (283.5,76.25) .. controls (283.5,78.04) and (282.04,79.5) .. (280.25,79.5) .. controls (278.46,79.5) and (277,78.04) .. (277,76.25) -- cycle ;
    \draw  [fill={rgb, 255:red, 0; green, 0; blue, 0 }  ,fill opacity=1 ] (277,96.25) .. controls (277,94.46) and (278.46,93) .. (280.25,93) .. controls (282.04,93) and (283.5,94.46) .. (283.5,96.25) .. controls (283.5,98.04) and (282.04,99.5) .. (280.25,99.5) .. controls (278.46,99.5) and (277,98.04) .. (277,96.25) -- cycle ;
    \draw  [fill={rgb, 255:red, 0; green, 0; blue, 0 }  ,fill opacity=1 ] (276.5,118.75) .. controls (276.5,116.96) and (277.96,115.5) .. (279.75,115.5) .. controls (281.54,115.5) and (283,116.96) .. (283,118.75) .. controls (283,120.54) and (281.54,122) .. (279.75,122) .. controls (277.96,122) and (276.5,120.54) .. (276.5,118.75) -- cycle ;
    \draw  [fill={rgb, 255:red, 255; green, 255; blue, 255 }  ,fill opacity=1 ] (277,55.25) .. controls (277,53.46) and (278.46,52) .. (280.25,52) .. controls (282.04,52) and (283.5,53.46) .. (283.5,55.25) .. controls (283.5,57.04) and (282.04,58.5) .. (280.25,58.5) .. controls (278.46,58.5) and (277,57.04) .. (277,55.25) -- cycle ;
    \draw [color={rgb, 255:red, 0; green, 0; blue, 0 }  ,draw opacity=1 ][line width=1.5]    (409.5,77.9) -- (430,77.9) ;
    \draw [color={rgb, 255:red, 0; green, 0; blue, 0 }  ,draw opacity=1 ][line width=0.75]  [dash pattern={on 4.5pt off 4.5pt}]  (409.5,97.9) -- (430,97.9) ;
    \draw [color={rgb, 255:red, 0; green, 0; blue, 0 }  ,draw opacity=1 ][line width=0.75]    (409,116.9) .. controls (410.67,115.23) and (412.33,115.23) .. (414,116.9) .. controls (415.67,118.57) and (417.33,118.57) .. (419,116.9) .. controls (420.67,115.23) and (422.33,115.23) .. (424,116.9) .. controls (425.67,118.57) and (427.33,118.57) .. (429,116.9) -- (429.5,116.9) -- (429.5,116.9) ;
    \draw [color={rgb, 255:red, 0; green, 0; blue, 0 }  ,draw opacity=1 ][line width=0.75]  [dash pattern={on 0.84pt off 2.51pt}]  (409.5,137.9) -- (430,137.9) ;
    
    \draw (100.5,130) node [anchor=north west][inner sep=0.75pt]   [align=left] {$\displaystyle p_{r}$};
    \draw (273,142) node [anchor=north west][inner sep=0.75pt]   [align=left] {$\displaystyle p_{s}$};
    \draw (98.5,47) node [anchor=north west][inner sep=0.75pt]  [font=\small] [align=left] {$\displaystyle r$};
    \draw (285.5,48) node [anchor=north west][inner sep=0.75pt]  [font=\small] [align=left] {$\displaystyle s$};
    \draw (359,67.5) node [anchor=north west][inner sep=0.75pt]  [font=\normalsize] [align=left] {$\displaystyle  \begin{array}{{>{\displaystyle}l}}
    {\displaystyle I.}\\
    {\displaystyle II.}\\
    {\displaystyle III.}\\
    {\displaystyle IV.}
    \end{array}$};

    \end{tikzpicture}

    \caption{Depiction of Cases \textbf{I.} through \textbf{IV.} on a bipartite graph when measurement $r,s$ is performed. The style of each edge represents the case from Table~\ref{table:k_part-diff-part} in which the corresponding edge $(a,b)$ is removed from the graph.}
    \label{fig:cases_for_kpart_diff_1to4}

\end{figure}


\begin{figure}[htbp]
    \centering

    
    \caption{Depiction of Cases \textbf{V.} through \textbf{VIII.} on a tripartite graph when measurement $r,s$ is performed. The style of each edge represents the case from Table~\ref{table:k_part-diff-part} in which the corresponding edge $(a,b)$ is removed from the graph.}
    \label{fig:cases_for_kpart_diff_5to8}

\end{figure}


\begin{figure}[htbp]
    \centering

    \begin{tikzpicture}[x=0.75pt,y=0.75pt,yscale=-1,xscale=1]
    
    \draw   (89,127.55) .. controls (89,100.24) and (97.95,78.1) .. (109,78.1) .. controls (120.05,78.1) and (129,100.24) .. (129,127.55) .. controls (129,154.86) and (120.05,177) .. (109,177) .. controls (97.95,177) and (89,154.86) .. (89,127.55) -- cycle ;
    \draw   (175.5,134.12) .. controls (175.5,103.64) and (184.45,78.94) .. (195.5,78.94) .. controls (206.55,78.94) and (215.5,103.64) .. (215.5,134.12) .. controls (215.5,164.59) and (206.55,189.3) .. (195.5,189.3) .. controls (184.45,189.3) and (175.5,164.59) .. (175.5,134.12) -- cycle ;
    \draw [color={rgb, 255:red, 0; green, 0; blue, 0 }  ,draw opacity=1 ]   (109.75,105.85) -- (195.75,126.85) ;
    \draw [color={rgb, 255:red, 0; green, 0; blue, 0 }  ,draw opacity=1 ]   (109.75,105.85) -- (192.5,105.85) ;
    \draw [color={rgb, 255:red, 0; green, 0; blue, 0 }  ,draw opacity=1 ]   (109.75,105.85) -- (195.75,146.85) ;
    \draw [color={rgb, 255:red, 0; green, 0; blue, 0 }  ,draw opacity=1 ]   (109.75,105.85) -- (195.25,169.35) ;
    \draw [color={rgb, 255:red, 0; green, 0; blue, 0 }  ,draw opacity=1 ]   (109,127.55) -- (195.75,105.85) ;
    \draw [color={rgb, 255:red, 0; green, 0; blue, 0 }  ,draw opacity=1 ]   (109.75,126.85) -- (195.75,126.85) ;
    \draw [color={rgb, 255:red, 0; green, 0; blue, 0 }  ,draw opacity=1 ]   (109.75,126.85) -- (195.75,146.85) ;
    \draw [color={rgb, 255:red, 0; green, 0; blue, 0 }  ,draw opacity=1 ]   (109.75,126.85) -- (195.25,169.35) ;
    \draw [color={rgb, 255:red, 0; green, 0; blue, 0 }  ,draw opacity=1 ] [dash pattern={on 0.84pt off 2.51pt}]  (109.75,146.85) -- (195.75,105.85) ;
    \draw [color={rgb, 255:red, 0; green, 0; blue, 0 }  ,draw opacity=1 ] [dash pattern={on 0.84pt off 2.51pt}]  (113,146.85) -- (195.75,126.85) ;
    \draw [color={rgb, 255:red, 0; green, 0; blue, 0 }  ,draw opacity=1 ] [dash pattern={on 0.84pt off 2.51pt}]  (109.75,146.85) -- (195.75,146.85) ;
    \draw [color={rgb, 255:red, 0; green, 0; blue, 0 }  ,draw opacity=1 ] [dash pattern={on 0.84pt off 2.51pt}]  (109.75,146.85) -- (195.25,169.35) ;
    \draw  [fill={rgb, 255:red, 255; green, 255; blue, 255 }  ,fill opacity=1 ] (106.5,105.85) .. controls (106.5,104.06) and (107.96,102.6) .. (109.75,102.6) .. controls (111.54,102.6) and (113,104.06) .. (113,105.85) .. controls (113,107.64) and (111.54,109.1) .. (109.75,109.1) .. controls (107.96,109.1) and (106.5,107.64) .. (106.5,105.85) -- cycle ;
    \draw  [fill={rgb, 255:red, 255; green, 255; blue, 255 }  ,fill opacity=1 ] (106.5,126.85) .. controls (106.5,125.06) and (107.96,123.6) .. (109.75,123.6) .. controls (111.54,123.6) and (113,125.06) .. (113,126.85) .. controls (113,128.64) and (111.54,130.1) .. (109.75,130.1) .. controls (107.96,130.1) and (106.5,128.64) .. (106.5,126.85) -- cycle ;
    \draw  [fill={rgb, 255:red, 0; green, 0; blue, 0 }  ,fill opacity=1 ] (106.5,146.85) .. controls (106.5,145.06) and (107.96,143.6) .. (109.75,143.6) .. controls (111.54,143.6) and (113,145.06) .. (113,146.85) .. controls (113,148.64) and (111.54,150.1) .. (109.75,150.1) .. controls (107.96,150.1) and (106.5,148.64) .. (106.5,146.85) -- cycle ;
    \draw  [fill={rgb, 255:red, 0; green, 0; blue, 0 }  ,fill opacity=1 ] (192.5,105.85) .. controls (192.5,104.06) and (193.96,102.6) .. (195.75,102.6) .. controls (197.54,102.6) and (199,104.06) .. (199,105.85) .. controls (199,107.64) and (197.54,109.1) .. (195.75,109.1) .. controls (193.96,109.1) and (192.5,107.64) .. (192.5,105.85) -- cycle ;
    \draw  [fill={rgb, 255:red, 0; green, 0; blue, 0 }  ,fill opacity=1 ] (192.5,126.85) .. controls (192.5,125.06) and (193.96,123.6) .. (195.75,123.6) .. controls (197.54,123.6) and (199,125.06) .. (199,126.85) .. controls (199,128.64) and (197.54,130.1) .. (195.75,130.1) .. controls (193.96,130.1) and (192.5,128.64) .. (192.5,126.85) -- cycle ;
    \draw  [fill={rgb, 255:red, 0; green, 0; blue, 0 }  ,fill opacity=1 ] (192.5,146.85) .. controls (192.5,145.06) and (193.96,143.6) .. (195.75,143.6) .. controls (197.54,143.6) and (199,145.06) .. (199,146.85) .. controls (199,148.64) and (197.54,150.1) .. (195.75,150.1) .. controls (193.96,150.1) and (192.5,148.64) .. (192.5,146.85) -- cycle ;
    \draw  [fill={rgb, 255:red, 0; green, 0; blue, 0 }  ,fill opacity=1 ] (192,169.35) .. controls (192,167.56) and (193.46,166.1) .. (195.25,166.1) .. controls (197.04,166.1) and (198.5,167.56) .. (198.5,169.35) .. controls (198.5,171.14) and (197.04,172.6) .. (195.25,172.6) .. controls (193.46,172.6) and (192,171.14) .. (192,169.35) -- cycle ;
    \draw [color={rgb, 255:red, 0; green, 0; blue, 0 }  ,draw opacity=1 ][line width=0.75]    (285.5,96) -- (306,96) ;
    \draw [color={rgb, 255:red, 0; green, 0; blue, 0 }  ,draw opacity=1 ][line width=0.75]  [dash pattern={on 0.84pt off 2.51pt}]  (285.5,114) -- (306,114) ;
    
    \draw (100,181) node [anchor=north west][inner sep=0.75pt]   [align=left] {$\displaystyle p_{r}$};
    \draw (188.5,193) node [anchor=north west][inner sep=0.75pt]   [align=left] {$\displaystyle p_{s}$};
    \draw (96.5,100) node [anchor=north west][inner sep=0.75pt]  [font=\small] [align=left] {$\displaystyle r$};
    \draw (96,123) node [anchor=north west][inner sep=0.75pt]  [font=\small] [align=left] {$\displaystyle s$};
    \draw (252,86.1) node [anchor=north west][inner sep=0.75pt]  [font=\normalsize] [align=left] {$\displaystyle  \begin{array}{{>{\displaystyle}l}}
    {\displaystyle X.}\\
    {\displaystyle XI.}
    \end{array}$};

    \end{tikzpicture}

    \caption{Depiction of Cases \textbf{X.} and \textbf{XI.} on a bipartite graph when $r,s$ is performed. The color and style of each edge represents the case from Table~\ref{table:k_part-same-part} in which the corresponding edge $(a,b)$ is removed from the graph.}
    \label{fig:fig:cases_for_kpart_same}

\end{figure}

\end{document}